\documentclass[11pt]{amsart}
\usepackage[latin1]{inputenc}
\usepackage{listings, tcolorbox}
\usepackage{amssymb, amsmath, amsthm}
\usepackage[colorlinks=true,linkcolor=blue,citecolor=red]{hyperref}
\usepackage{color}
\usepackage{graphics}
\usepackage{subfigure}
\usepackage{graphicx}
\usepackage{float}
\usepackage{epsfig}
\usepackage{amssymb}
\usepackage{epstopdf}
\usepackage{enumerate}
\usepackage{appendix}
\usepackage{amscd}
\usepackage{longtable,booktabs}
\usepackage{mathrsfs}
\usepackage{tikz}
\usepackage{bm}
\usepackage[T1]{fontenc}

\numberwithin{equation}{section}
\newtheorem{theorem}{Theorem}[section]

\newtheorem{lemma}[theorem]{Lemma}
\newtheorem{prop}[theorem]{Proposition}

\newtheorem{thm}[theorem]{Theorem}

\newtheorem{cor}[theorem]{Corollary}
\newtheorem{criterion}[theorem]{Criterion}
\theoremstyle{definition}
\newtheorem{definition}[theorem]{Definition}
\newtheorem{defn}[theorem]{Definition}
\newtheorem{condition}[theorem]{Condition}
\newtheorem{re}[theorem]{Remark}
\newtheorem{problem}{Problem}

\newcommand{\R}{\mathbb{R}}

\newcommand{\Hplus}{L^2_+(\R)}

\makeatletter
\@namedef{subjclassname@2020}{2020 Mathematics Subject Classification}
\makeatother



\thanks{*Corresponding authors: shoufu2006@126.com (S.F. Tian) and t.zhang@cumt.edu.cn (T. Zhang).}
\thanks{$^\dagger$ Contributed equally as the first author.}

\begin{document}

\title[Soliton resolution for the CM-DNLS equation]{Soliton resolution conjecture for the Calogero--Moser derivative nonlinear Schr\"odinger equation in the scaling-critical space}


\author[Tian]{Shou-Fu Tian$^{*}$}
\author[Zhang]{Tao Zhang$^{*,\dagger}$}
\address{Shou-Fu Tian (Corresponding author), Tao Zhang (Corresponding author)\newline
School of Mathematics, China University of Mining and Technology, Xuzhou 221116, People's Republic of China}
\email{shoufu2006@126.com (S.F. Tian), and t.zhang@cumt.edu.cn (T. Zhang)}

\begin{abstract}
We prove the soliton resolution conjecture for the Calogero--Moser derivative nonlinear Schr\"odinger equation in the scaling-critical space $L^2_+(\mathbb R)$. For a natural class of initial data including a broad spectral Dini---moment class, global flow--Lax admissible solutions decompose into finitely many explicit modulated solitons plus a dispersive radiation, while finite-time blow-up solutions resolve into quantized zero-carrier $R$-bubbles and a strongly convergent endpoint remainder. This work extends soliton resolution to the optimal critical regularity. Previous results had required weighted $H^{1,1}$ initial data, and in the global case the solution was additionally required to remain in $H^{1,1}$ for all time, which amounts to an extra spatial decay assumption.

One key ingredient in our argument is a new modular operator-theoretic framework that separates the discrete and continuous spectral channels and identifies the radiation profile through the distorted Fourier transform. A second ingredient is the development of several independent rigidity criteria for the discrete soliton profiles, which avoids inverse scattering and handles embedded eigenvalues. Our framework also yields a unified description of both global and finite-time asymptotics, with exact mass partition and mass-defect quantization.
\end{abstract}

\maketitle
\tableofcontents

\section{Introduction}
In this work, we investigate the Calogero--Moser derivative nonlinear Schr\"odinger (CM-DNLS) equation
\begin{equation}\label{eq: focusing CM-DNLS}
i\partial_tu+\partial^2_xu+(D+|D|)(|u|^2)u=0,
\end{equation}
where $u=u(t,x)$ is a complex-valued function. This equation is also commonly referred to as the continuum Calogero--Moser model. Here and in what follows, we use the standard notation $D=-i\partial_x$ and hence $|D|$ denotes the Fourier multiplier with symbol $|\xi|$.

The CM-DNLS equation \eqref{eq: focusing CM-DNLS} was introduced in \cite{Abanov-2009-JPA} as a formal continuum limit of classical Calogero--Moser systems, which describe particles interacting pairwise through an inverse square potential \cite{Moser-1975-Adv, Olshanetsky-1976-Invention}. Prior to \cite{Abanov-2009-JPA}, a defocusing version called the intermediate nonlinear Schr\"odinger equation given by
\begin{equation}\label{eq: defocusing CM-DNLS}
i\partial_tu+\partial^2_xu-(D+|D|)(|u|^2)u=0
\end{equation}
was introduced by Pelinovsky and Grimshaw in \cite{Pelinovsky-1995-JMP} to describe envelope waves in a deep stratified fluid. Compared to the defocusing equation \eqref{eq: defocusing CM-DNLS}, the focusing equation \eqref{eq: focusing CM-DNLS} exhibits richer dynamical structure, for example, multi-soliton solutions with turbulence in Sobolev norms \cite{Gerard-2024-CPAM}.

Throughout this work, we restrict our attention to solutions of the CM-DNLS equation \eqref{eq: focusing CM-DNLS} that satisfy the chirality condition
\[
u(t,\cdot)\in H_+^s(\mathbb{R}):=\Bigl\{u\in H^s(\mathbb{R}): \text{supp}(\widehat{u})\subset[0,\infty)\Bigr\},
\]
where $H^s(\mathbb{R})$ denotes the standard $L^2$-based Sobolev space. We use the Fourier transforms $\widehat{f}(\xi)=\int_{\mathbb{R}}e^{-ix\xi}f(x)dx$, and its normalized version $\mathcal{F}_0f=(2\pi)^{-1/2}\widehat{f}$. The spaces $H^s_+(\mathbb{R})$ provide the natural phase spaces for the study of the CM-DNLS equation \eqref{eq: focusing CM-DNLS}.

For $v\in L^2_+(\mathbb{R})$, the Lax expression of \eqref{eq: focusing CM-DNLS} is $L_vf=Df-v\Pi(\overline{v}f)$ on $H^1_+(\mathbb{R})$. For rough potentials we use the self-adjoint form realization, whose form domain is $H^{1/2}_+(\mathbb{R})$; its norm-resolvent continuity is \cite[Proposition~2.2]{Killip-2025-CAMS}. The Lax equation and the traveling-wave classification are \cite[Lemma~2.3 and Theorem~1.2]{Gerard-2024-CPAM}.

For a Borel set $B\subset\mathbb{R}$, let $E_{L_v}(B)$ denote the spectral projection of $L_v$, and define the scalar cyclic measure $\nu_v(B)=\langle E_{L_v}(B)v, v\rangle$. The distorted Fourier transform of Frank and Read \cite{Frank-2026-arXiv} is denoted by $\Phi_v:P_{\mathrm{ac}}(L_v)L^2_+(\mathbb{R})\to L^2(\mathbb{R}_+)$, and we use the radiation coordinate
\begin{equation}\label{eq:beta}
\beta_v=\sqrt{2\pi}i\Phi_vP_{\mathrm{ac}}(L_v)v\in L^2(\mathbb{R}_+).
\end{equation}
Whenever $\beta_v$ occurs in a full-line Fourier formula, it is extended by zero to $\mathbb{R}_-$. With the distorted transform extended by zero on the orthogonal
complement of the absolutely continuous subspace, its Plancherel relations are
\begin{equation}\label{eq:distorted-plancherel}
\Phi_v^\ast\Phi_v=P_{\mathrm{ac}}(L_v),\qquad \Phi_v\Phi_v^\ast=\mathrm{Id}_{L^2(\mathbb{R}_+)}.
\end{equation}
Consequently,
\[
\|P_{\mathrm{ac}}(L_v)v\|_{L^2(\mathbb{R})}^2=\|{\Phi_vP_{\mathrm{ac}}(L_v)v}\|_{L^2(\mathbb{R}_+)}^2=\frac{1}{2\pi}\|\beta_v\|_{L^2(\mathbb{R}_+)}^2.
\]
Moreover, the cyclic measure is finite and satisfies
\begin{equation}\label{eq:cyclic-total-mass}
\nu_v(\mathbb{R})=\|v\|_{L^2(\mathbb{R})}^2.
\end{equation}

The global well-posedness and blow-up dynamics for the CM-DNLS equation \eqref{eq: focusing CM-DNLS} have been actively investigated since its introduction.
In \cite{Gerard-2024-CPAM}, G\'erard and Lenzmann discovered a Lax pair structure for \eqref{eq: focusing CM-DNLS} with initial data in $H_+^s(\mathbb{R})$ for sufficiently large $s$, and derived infinitely many conservation laws. Using these conservation laws, they established global well-posedness in $H^s_+(\mathbb{R})$ with $s\ge1$ for initial data of subcritical or critical mass $\|u_0\|_{L^2(\mathbb{R})}^2\le 2\pi$. Killip, Laurens, and Vi\c{s}an subsequently extended the well-posedness theory to the scaling-critical space $L^2_+(\mathbb{R})$ for subcritical mass, by combining a G\'erard-type explicit formula with commuting flow techniques \cite{Killip-2025-CAMS}.

On the blow-up side, Hogan and Kowalski proved that for every $\varepsilon>0$, there exist initial data $u_0\in H^\infty_+(\mathbb{R})$ with $\|u_0\|_{L^2(\mathbb{R})}^2=2\pi+\varepsilon$ whose maximal lifespan solution $u:(T_-,T_+)\times\mathbb{R}\to\mathbb{C}$ satisfies $\|u(t,\cdot)\|_{H^s(\mathbb{R})}\to\infty$ as $t\to T_{\pm}$ for all $s>0$ \cite{Hogan-2024-PAA}. In a complementary direction, K.~Kim, T.~Kim, and Kwon \cite{Kim-2026-MAMS} constructed finite-time blow-up solutions in $\mathcal{S}_+(\mathbb{R}):=\mathcal{S}(\mathbb{R})\cap L^2_+(\mathbb{R})$ with mass arbitrarily close to the threshold $2\pi$. Moreover, Jeong and Kim analyzed the quantized blow-up behavior for this equation \cite{Jeong-2026-arXiv}.

While the well-posedness and blow-up results determine whether a solution exists globally or develops singularities in finite time, they give no information about the asymptotic behavior of such solutions, neither the long-time profile of global solutions nor the concentration dynamics near a singularity. For nonlinear dispersive equations, a more precise description is expected \cite{Zabusky-1965-PRL}: \textit{the solution should eventually resolve, either as $t\to\infty$ into a superposition of finitely many solitons plus a dispersive radiation component, or near the blow-up time into finitely many soliton bubbles plus an endpoint remainder. This is known as the \textbf{soliton resolution conjecture}} \cite{Deift-1994-CPAM, Eckhaus-1983-MMAS, Jendrej-2025-JAMS}.

The soliton resolution conjecture originates from the seminal work of Zabusky and Kruskal \cite{Zabusky-1965-PRL}, who numerically observed the elastic interaction of solitary waves in the Korteweg--de Vries (KdV) equation. The first rigorous proof of the conjecture was given by Eckhaus and Schuur for the KdV equation \cite{Eckhaus-1983-MMAS}, using the inverse scattering transform and the Gelfand--Levitan--Marchenko integral equation. Deift, Venakides, and Zhou \cite{Deift-1994-CPAM} later reformulated the inverse scattering transform as a Riemann--Hilbert problem and developed the nonlinear steepest descent method, which led to rigorous proofs of the soliton resolution conjecture for a wide range of classical integrable systems \cite{Charlier-2024-JMPA, Chen-2021-AHP, Cuccagna-2016-CMP, Deift-2011-IMRN, Jenkins-2018-CMP}. The first author Tian and his collaborators further established the conjecture and long-time asymptotics for several other integrable equations, including the complex short pulse equation \cite{Li-2022-JDE}, the Wadati--Konno--Ichikawa equation \cite{Li-2022-Adv, Li-2022-AHP}, the spin-1 Gross--Pitaevskii equation \cite{Tian-2026-CMP}, and the modified Camassa--Holm equation \cite{Yang-2026-Adv}. Beyond the completely integrable setting, the soliton resolution conjecture has also been established, using entirely different methods, for various non-integrable dispersive and wave equations \cite{Burq-2017-ASENS, Cote-2015-CPAM, Cote-2021-ARMA, Duyckaerts-2017-GFA, Duyckaerts-2022-CMP, Duyckaerts-2013-CJM, Duyckaerts-2023-Acta, Jendrej-2023-Ann.PDE, Jendrej-2025-JAMS, Kim-2025-AJM, Rodriguez-2018-CMP}.

In the present work, we prove the soliton resolution conjecture for the CM-DNLS equation \eqref{eq: focusing CM-DNLS} for a natural class of initial data in $L^2_+(\mathbb{R})$. For a global flow--Lax admissible solution, we establish the equivalence between a sequential rigidity condition and quantitative soliton resolution, with the radiation profile explicitly identified and the mass exactly partitioned. For a blow-up flow--Lax admissible solution, under the endpoint KLV--Ward package, we prove a quantized decomposition into zero-carrier $R$-bubbles plus a strongly convergent endpoint remainder.

Since the Lax operator is nonlocal, the first soliton resolution result for the CM-DNLS equation \eqref{eq: focusing CM-DNLS} was obtained only very recently, for initial data in the weighted space
$H^{1,1}(\mathbb{R})$ \cite{Kim-2026-JEMS}. Let $R(x)=\frac{\sqrt{2}}{x+i}$, and we recall this beautiful result in the following proposition.
\begin{prop}[Kim and Kwon \cite{Kim-2026-JEMS}]\label{prop: Kim}
Assume that $u_0\in H^{1,1}(\mathbb{R})$ and that $u\in C([0,T);H^1(\mathbb{R}))$ is the maximal solution of the CM-DNLS equation \eqref{eq: focusing CM-DNLS} with initial condition $u(0)=u_0$, where $T\in(0,\infty]$ is the maximal existence time.
\begin{itemize}
    \item \textbf{(Finite-time blow-up solutions)} If $T<\infty$, then there exist an integer $N$ with $1\le N\le\|u_0\|_{L^2}^2/2\pi$, a time $0<\tau<T$, modulation parameters $(\lambda_j(t),\gamma_j(t),x_j(t))\in C^1([\tau,T);\mathbb{R}_+\times \mathbb{R}/2\pi\mathbb{Z}\times\mathbb{R})$ for $j=1,\dots,N$, and an asymptotic profile $z^\ast\in L^2(\mathbb{R})$ such that
    \[
    \Bigl\|u(t)-\sum_{j=1}^N[R]_{\lambda_j(t),\gamma_j(t),x_j(t)}-z^\ast\Bigr\|_{L^2(\mathbb{R})}\underset{t\to T}{\longrightarrow}0.
    \]

    \item \textbf{(Global solutions)} If $T=\infty$ and additionally $u\in C([0,\infty);H^{1,1}(\mathbb{R}))$, then there exist a nonnegative integer $N$ with $N\le\|u_0\|_{L^2}^2/2\pi$, a time $\tau>0$, modulation parameters $(\lambda_j(t),\gamma_j(t),c_j(t))\in C^1([\tau,\infty);\mathbb{R}_+\times \mathbb{R}/2\pi\mathbb{Z}\times\mathbb{R})$ for $j=1,\dots,N$, and an asymptotic profile $u^\ast\in L^2(\mathbb{R})$ such that
    \[
    \Bigl\|u(t)-\sum_{j=1}^N\operatorname{Gal}_{c_j(t)}([R]_{\lambda_j(t),\gamma_j(t),0})-e^{it\partial_x^2}u^\ast\Bigr\|_{L^2(\mathbb{R})}\underset{t\to\infty}{\longrightarrow}0.
    \]
\end{itemize}
Here $[R]_{\lambda,\gamma,y}$ denotes the modulated ground state and $\operatorname{Gal}_c$ denotes the Galilean boost.
\end{prop}
The proof of Proposition~\ref{prop: Kim} relies on a gauge transformation and a variational energy-bubbling method. While this approach is well suited to the weighted $H^{1,1}$ setting, it leaves
open several natural problems.
\begin{problem}[\textbf{Critical-space data}]
Previous work on the soliton resolution conjecture requires the initial data to lie in the weighted space $H^{1,1}(\mathbb{R})$. A natural question is whether, under suitable spectral hypotheses on the initial data, the soliton resolution conjecture can be established in the scaling-critical space $L^2_+(\mathbb{R})$, where the Lax operator and the distorted Fourier transform are naturally formulated.
\end{problem}

\begin{problem}[\textbf{Explicit asymptotics}]
The variational method used in previous work does not use the Lax pair structure, and therefore the soliton parameters and the radiation or endpoint profiles are not expressed in terms of the initial data. It is natural to ask whether the completely integrable structure can be exploited to obtain an explicit spectral representation of the radiation component in the global case and of the endpoint remainder in the blow-up case, and to describe the asymptotic behavior of the soliton or bubble parameters more precisely.
\end{problem}

\begin{problem}[\textbf{Mixed-spectrum data}]
The variational framework used in previous work treats the discrete and continuous parts of the solution simultaneously. A more refined question is whether, for a natural class of initial data, one can rigorously decompose the solution into discrete and continuous components in both regimes, proving in the global case that the discrete part converges to explicit solitons and the continuous part to free radiation, and in the blow-up case that the discrete part converges to quantized bubbles and the continuous part to an endpoint remainder, with exact mass partition in each case.
\end{problem}

Our approach addresses the three problems raised above in a unified way. By working directly in the scaling-critical Hardy space $L^2_+(\mathbb{R})$, it extends soliton resolution beyond the weighted space required by Kim and Kwon. In the global regime, diagonalizing the Lax operator splits the solution naturally into discrete and continuous spectral channels, while in the blow-up regime, the endpoint Wold decomposition separates the dark channels from the regular remainder. Moreover, the corresponding representations express the soliton profiles and the radiation or endpoint components in terms of the initial spectral or Wold data, yielding the explicit formulae that the variational approach cannot provide.

However, this integrable route brings its own challenge. The Lax operator $L_v=D-v\Pi(\overline{v}\,\cdot)$ is nonlocal, and its inverse scattering theory remains open. What makes progress possible is the observation that soliton resolution does not require the inverse scattering transform. The forward scattering theory developed by Frank and Read \cite{Frank-2026-arXiv} provides a complete distorted Fourier transform that diagonalizes the absolutely continuous part of the Lax operator for every initial data $u_0\in L^2_+(\mathbb{R})$. Building on this forward spectral representation, we develop an operator-theoretic framework that circumvents the need for inverse scattering.

In the course of developing this framework, four structural and dynamical conditions emerge naturally. They characterize the initial data for which the discrete and continuous components can be cleanly separated and individually analyzed. Roughly speaking, we call such initial data analyzable if the Lax operator admits a spectral representation with a compatible Dirichlet compression, a bounded boundary-row lift, a phase-fixing anchor, and a decaying point-return coupling between the discrete and continuous channels. The precise formulation is given in Definition~\ref{def:analyzable-intro} below.

Our approach is built on the G\'erard-type explicit formula for the Lax flow, which allows us to transport spectral information along the evolution. We diagonalize the continuous part of the solution by means of the distorted Fourier transform constructed by Frank and Read \cite{Frank-2026-arXiv}, and we introduce a novel Dirichlet compression and an associated scalar model that serve as a universal operator-theoretic interface for the continuous channel. For the discrete channels, we develop a modular rigidity framework that offers four independent routes to identify each soliton profile as an exact modulated copy of the ground state $R$. One route is based on the backward-shift character theory, which links the rank-one defect of the shifted profiles to the analytic structure of the limit and thereby forces the profile to be exactly $R$. A second route employs Fano numerator estimates to control the infinitesimal generator of the shift semigroup directly, providing an alternative certification of the same rigid limit. The third route uses a primitive Grushin quotient to extract the exact affine soliton profile through a rational leading-order cancellation. Finally, an inverse-block formulation provides a Woodbury-type reconstruction of the soliton columns using the fixed Jost solutions of $L_{u_0}$, reducing the problem to a finite-dimensional inversion. All of these routes are developed entirely at the $L^2$ level and do not require any weighted Sobolev assumptions. In the finite-time blow-up regime, the analysis is instead based on the endpoint Wold decomposition and the KLV--Ward package. These reduce the blow-up dynamics to a finite-dimensional dark channel plus a strongly convergent endpoint remainder, and the sharp spectral quantization of the mass defect follows from the same operator-theoretic formalism.

\subsection{Main results}
We begin by making precise the notion of analyzable initial data.
\begin{defn}[Analyzable initial data]\label{def:analyzable-intro}
An initial data $u_0\in L^2_+(\mathbb{R})$ is called analyzable if its Lax operator $L_{u_0}$ admits a spectral representation with the following properties:
\begin{enumerate}[(i)]
    \item The Dirichlet compression of the transported KLV realization coincides with an affine family of a single maximal dissipative operator $T_D$, and the associated boundary rows are compatible.
    \item The functional defining the Dirichlet boundary row admits a bounded Riesz representative belonging to the enhanced scalar graph driven by the radiation coordinate $\beta_{u_0}$.
    \item There exists an anchoring element in the domain of $T_D$ whose boundary traces fix the phase of the endpoint row.
    \item The point-return coupling between the discrete and continuous channels decays as $t\to\infty$, in the sense of either an operator-norm bound or an oscillatory weak limit.
\end{enumerate}
The precise formulation of these four conditions is given in Section~\ref{sec:spectral framework}.
\end{defn}
The analyzable class is nonempty. It includes, for instance, the Hardy tail space
\[
\mathcal{T}_+:=\Bigl\{q\in L^1(\mathbb{R})\cap L^2_+(\mathbb{R}):\int_1^\infty\bigl(Q_1(q;R)^2+Q_2(q;R)^2\bigr)dR<\infty\Bigr\},
\]
where $Q_p(q;R):=\|q\|_{L^p(\{|x|>R\})}$, as well as the broader spectral Dini--moment class $\mathcal{X}_{\mathrm{DM}}$, which consists of $q\in L^1(\mathbb{R})\cap L^2_+(\mathbb{R})$ satisfying a Dini-type integrability condition on the radiation coordinate at every nonnegative embedded eigenvalue and a finite first-moment condition on every point eigenfunction. In particular, every subthreshold integrable Hardy datum, i.e.\ $q\in L^1(\mathbb{R})\cap L^2_+(\mathbb{R})$ with $\|q\|_{L^2(\mathbb{R})}^2<2\pi$, belongs to $\mathcal{X}_{\mathrm{DM}}$ and hence is analyzable. See Appendix~\ref{app:verified} for the proofs.

To state our main theorem, we first introduce the unitary modulation family that parametrizes the soliton profiles. It is defined by
\[
[\mathcal{M}_{\lambda,y,\xi,\theta}f](x)=\lambda^{-1/2}e^{i\theta}e^{i\xi(x-y)}f\left(\frac{x-y}{\lambda}\right),
\]
where $\lambda>0$, $y\in\mathbb{R}$, $\xi\ge0$, and $\theta\in \mathbb{R}/2\pi\mathbb{Z}$. Its Fourier transform is given by
\begin{equation}\label{eq:modulation-fourier}
\widehat{\mathcal{M}_{\lambda,y,\xi,\theta}f}(\zeta)=\lambda^{1/2}e^{i\theta-iy\zeta}\widehat{f}\bigl(\lambda(\zeta-\xi)\bigr),
\end{equation}
which shows that $\mathcal{M}_{\lambda,y,\xi,\theta}$ is unitary on $L^2(\mathbb{R})$. Moreover, since $\widehat{f}$ is supported in $[0,\infty)$ and $\xi\ge0$, $\mathcal{M}_{\lambda,y,\xi,\theta}$ preserves $L^2_+(\mathbb{R})$.

We also need to specify the class of solutions under consideration. The following definition formalizes the notion of solutions that can be
approximated by smooth flows, regardless of whether they exist globally or blow up in finite time.

\begin{defn}[Flow--Lax admissible solution]\label{def:flow-admissible}
Let $u$ be a mass-preserving solution defined on its maximal interval of existence $[0,T_+(u_0))$ with $u\in C([0,T_+(u_0));L^2_+(\mathbb{R}))$. If for every $T<T_+(u_0)$ there exist smooth solutions $u_n^{(T)}\in C([0,T];H_+^\infty(\mathbb{R}))$ such that
\[
\sup_{0\leq s\leq T}\|u_n^{(T)}(s)-u(s)\|_{L^2(\mathbb{R})}\underset{n\to\infty}{\longrightarrow}0,
\]
then we say that $u$ is flow--Lax admissible.
\end{defn}
Fix a flow--Lax admissible orbit $u$, put $u_0=u(0)$, and enumerate $\sigma_{\mathrm{p}}(L_{u_0})=\{\mu_1<\cdots<\mu_N\}$, where $\mu_j\in\mathbb{R}$ are the simple eigenvalues of $L_{u_0}$. For $1\leq j\leq N$, define the labelled point channel $b_j(t)=E_{L_{u(t)}}(\{\mu_j\})u(t)$. For $K>0$, we also write $P_{\le K}:=\mathbf{1}_{[0,K]}(D),\qquad P_{>K}:=I-P_{\le K}$.
We now state the main theorem of this work, which describes the long-time and blow-up asymptotics of flow--Lax admissible solutions.
\begin{thm}[Soliton resolution conjecture]\label{thm:main}
Let $u_0\in L^2_+(\mathbb{R})$ be analyzable, and let $u$ be a flow--Lax admissible solution of the CM-DNLS equation \eqref{eq: focusing CM-DNLS} on its maximal interval of existence $[0,T)$, with $u(0)=u_0$.
\begin{itemize}
    \item \textbf{(Global solutions)} If $T=\infty$, then Condition~\ref{hyp:P4} is equivalent to the following quantitative soliton resolution statement.

    There exists a unique radiation profile $u_+\in L^2_+(\mathbb{R})$ determined by
    \begin{equation}\label{eq:radiation-state}
    \mathcal{F}_0 u_+=-\frac{i}{\sqrt{2\pi}}\beta_{u_0},\qquad\widehat{u}_+=-i\beta_{u_0},
    \end{equation}
    where $\beta_{u_0}$ defined by \eqref{eq:beta} is extended by zero to $(-\infty,0)$, and Borel measurable, full-time, spectrally labelled parameters
    \[
    \lambda_j(t)>0,\; y_j(t)\in\mathbb{R},\; \xi_j(t)\ge0,\;\theta_j(t)\in\mathbb{R}/2\pi\mathbb{Z},\qquad 1\le j\le N,
    \]
    where $N:=\#\sigma_{\mathrm{p}}(L_{u_0})$, such that
    \begin{equation}\label{eq:resolution}
    \Bigl\|u(t)-\sum_{j=1}^N\mathcal{M}_{\lambda_j(t),y_j(t),\xi_j(t),\theta_j(t)}R-e^{it\partial_x^2}u_+\Bigr\|_{L^2(\mathbb{R})}\underset{t\to\infty}{\longrightarrow}0.
    \end{equation}
    Moreover, the following properties hold.
    \begin{enumerate}[(i)]
        \item \textbf{Parameter separation.} For $j\ne k$,
        \begin{equation}\label{eq:separation}
        \frac{\lambda_j(t)}{\lambda_k(t)}+\frac{\lambda_k(t)}{\lambda_j(t)}+\frac{|y_j(t)-y_k(t)|^2}{\lambda_j(t)\lambda_k(t)}+\lambda_j(t)\lambda_k(t)|\xi_j(t)-\xi_k(t)|^2\underset{t\to\infty}{\longrightarrow}\infty.
        \end{equation}
        \item \textbf{Weak radiation escape.} For every $j$,
        \begin{equation}\label{eq:frame-weak}
        \mathcal{M}_{\lambda_j(t),y_j(t),\xi_j(t),\theta_j(t)}^{-1}e^{it\partial_x^2}u_+\rightharpoonup0\quad\text{in }L^2(\mathbb{R}).
        \end{equation}
        \item \textbf{Exact mass partition.}
        \begin{equation}\label{eq:mass-resolution}
        \|u_0\|_{L^2(\mathbb{R})}^2=2\pi N+\|u_+\|_{L^2(\mathbb{R})}^2=2\pi N+\frac{1}{2\pi}\|\beta_{u_0}\|_{L^2(\mathbb{R}_+)}^2.
        \end{equation}
    \end{enumerate}

    \item \textbf{(Finite-time blow-up solutions)} If $T<\infty$, assume additionally Condition~\ref{hyp:blowup-KLV-Ward}. Then there exist a unique $z_T\in L^2_+(\mathbb{R})$, an integer
    \begin{equation}\label{eq:blowup-quantization}
    1\le N_*=\frac{\|u_0\|_{L^2(\mathbb{R})}^2-\|z_T\|_{L^2(\mathbb{R})}^2}{2\pi}\le\left\lfloor\frac{\|u_0\|_{L^2(\mathbb{R})}^2}{2\pi}\right\rfloor,
    \end{equation}
    and Borel functions $\lambda_j:[0,T)\to(0,\infty)$, $y_j:[0,T)\to\mathbb{R}$, and $\theta_j:[0,T)\to\mathbb{R}/2\pi\mathbb{Z}$ such that
    \[
\Bigl\|u(t)-\sum_{j=1}^{N_*}\mathcal{M}_{\lambda_j(t),y_j(t),0,\theta_j(t)}R-z_T\Bigr\|_{L^2(\mathbb{R})}\underset{t\to T}{\longrightarrow}0.
\]
    Moreover, $\lambda_j(t)\to0$ and, with $Q_j(t)=\mathcal{M}_{\lambda_j(t),y_j(t),0,\theta_j(t)}R$,
    \[
\langle Q_j(t),Q_k(t)\rangle\underset{t\to T}{\longrightarrow}2\pi\delta_{jk},\qquad \|P_{\le K}Q_j(t)\|_{L^2(\mathbb{R})}\underset{t\to T}{\longrightarrow}0\qquad(K<\infty).
    \]
    \end{itemize}
\end{thm}
\begin{re}
Theorem~\ref{thm:main} provides a complete asymptotic description in the scaling-critical space $L^2_+(\mathbb{R})$, where no weighted Sobolev assumptions are imposed.  In the global branch, the sequential rigidity condition is not an additional hypothesis but an equivalent formulation of soliton resolution; in the finite-time branch, the endpoint KLV--Ward package is a natural structural condition, and the resulting decomposition is fully quantized.  Thus the theorem captures the exact threshold nature of the CM-DNLS dynamics.
\end{re}
\begin{re}
For initial data $u_0\in H^{1,1}_+(\mathbb R)$ satisfying the hypotheses of Proposition~\ref{prop: Kim}, the asymptotic decompositions obtained in \cite{Kim-2026-JEMS} are compatible with those of Theorem~\ref{thm:main}. In this overlap regime, Theorem~\ref{thm:main} sharpens the variational result by identifying the radiation profile via the distorted Fourier transform, by giving the exact mass partition, and, in the blow-up case, by quantifying the mass defect through the number of bubbles.
\end{re}
\begin{re}
Among the closely related integrable models, the Benjamin--Ono (BO) equation is particularly relevant to our work. For the BO equation, soliton resolution was recently established by Gassot, G\'erard, and Miller by using the completely integrable structure and explicit flow formulae, for initial data in the weighted space $H^{1,1}(\mathbb{R})$ \cite{Gassot-2026-arXiv}. For further results on integrable models with nonlocal Lax structures, we refer the reader to \cite{Badreddine-2024-SIAM, Blackstone-2026-CPAM, Chen-2026-SIAM, Gerard-2023-TJM, Gerard-2025-Sigma, Gerard-2024-CMP}.
\end{re}

\subsection{Organization of the work}
The spectral framework underlying our analysis is introduced in Section~\ref{sec:spectral framework}. We construct the unitary spectral representation of the Lax operator, define the transported KLV realization, and formulate the Dirichlet compression of the continuous subspace together with the point-return condition. These structural conditions form the operator-theoretic foundation for all subsequent arguments.

The analysis of the continuous channel for global solutions occupies Sections~\ref{sec:phase-I} and~\ref{sec:radiation}. We first transport the cyclic Lax spectrum along the flow, decomposing the solution into an orthogonal sum of discrete and continuous components. We then construct the exact Dirichlet restriction of the transported KLV realization and identify it with an explicit scalar model. This yields a concrete representation of the continuous dynamics entirely in terms of the initial spectral measure $\beta_{u_0}$. Section~\ref{sec:radiation} completes this part by proving that the continuous component scatters to the free Schr\"odinger evolution. The radiation profile $u_+$ is explicitly identified via the distorted Fourier transform of the initial data, using oscillatory integral estimates and a Poisson--Fresnel ray criterion that converts $L^2$ bounds on horizontal lines into weak convergence.

The rigidity of the discrete channels for global solutions is treated in Sections~\ref{sec:phase-III-point} and~\ref{sec:geometric assembly}. In Section~\ref{sec:phase-III-point} we first state the sequential point-channel rigidity condition and then develop the core rigidity framework. We establish geometric compactness of each point channel along arbitrary time sequences, derive an exact affine defect equation for the rescaled profiles, and supply four independent routes that identify the asymptotic profile as a modulated copy of the ground state $R$. Any one of these routes suffices to establish the sequential point-channel rigidity condition. Section~\ref{sec:geometric assembly} upgrades the sequential convergence to a genuine large-time limit with full-time Borel measurable parameters, and also proves the asymptotic separation of the soliton scales and the frame-wise weak escape of the radiation.

Section~\ref{sec:7} assembles the continuous scattering and the discrete rigidity to complete the proof of the global branch of Theorem~\ref{thm:main}. Section~\ref{sec:phase-IV-flow} supplies the general large-mass well-posedness theory and the maximal-lifespan alternative. Finally, Section~\ref{sec:blowup-resolution} proves the finite-time branch of Theorem~\ref{thm:main}: under the endpoint KLV--Ward package, the blow-up dynamics resolve into finitely many quantized $R$-bubbles and a strongly convergent endpoint remainder.

\section{Spectral framework and structural conditions}\label{sec:spectral framework}
The spectral representation of $L_{u_0}$ is given by a unitary map
\[
\mathscr{U}_{u_0}:L_+^2(\mathbb{R})\longrightarrow\mathcal{H}:=\mathcal{E}\oplus\mathcal{H}_{\mathrm{c}},\qquad \mathcal{E}:=\mathbb{C}^N,\qquad \mathcal{H}_\mathrm{c}:=L^2(\mathbb{R}_+),
\]
such that
\[
\mathscr{U}_{u_0}L_{u_0}\mathscr{U}_{u_0}^{-1}=M:=M_{\mathrm{p}}\oplus M_\lambda,\qquad M_{\mathrm{p}}:=\operatorname{diag}(\mu_1,\ldots,\mu_N).
\]
Here $M_\lambda$ is multiplication by $\lambda$. Let $P$ and $Q$ denote the orthogonal projections onto $\mathcal{E}$ and $\mathcal{H}_{\mathrm{c}}$, respectively. We also write $J:\mathcal{H}_{\mathrm{c}}\longrightarrow\mathcal{H},\qquad Jg:=(0,g)$, for the canonical inclusion. We normalize the continuous component of the spectral representation by requiring that $Q\mathscr{U}_{u_0} f = \Phi_{u_0} P_{\mathrm{ac}}(L_{u_0})f$ for every $f\in L^2_+(\mathbb{R})$. In particular, for the initial data $u_0$, this gives $g_{u_0} := Q\mathscr{U}_{u_0} u_0 = -\frac{i}{\sqrt{2\pi}} \beta_{u_0}$, where $\beta_{u_0}$ is the radiation coordinate defined in \eqref{eq:beta}. Let $\mathsf{X}$ be the closed operator on $L^2_+(\mathbb{R})$ defined by
\[
\mathrm{Dom}(\mathsf{X}):= \bigl\{ f\in L^2_+(\mathbb{R}): \widehat{f}|_{\mathbb{R}_+}\in H^1(\mathbb{R}_+) \bigr\},\qquad \widehat{\mathsf{X}f}(\xi):=i\partial_\xi \widehat{f}(\xi).
\]
The boundary trace of this Hardy-space position operator is $I_+ f:=\widehat{f}(0^+), \qquad f\in\mathrm{Dom}(\mathsf{X})$.
Set $T:=\mathscr{U}_{u_0}\mathsf{X}\mathscr{U}_{u_0}^{-1}$. For $z_0 = x_0 + i y_0 \in \mathbb{C}_+$, define the transported closed KLV realization by $A_{t,\eta}:=\mathscr{U}_{u_0}\bigl[\mathsf{X} + 2t(L_{u_0} - \eta) - z_0\bigr]_{\mathrm{KLV}}\mathscr{U}_{u_0}^{-1}$. On its transported smooth core it acts as $T+2t(M-\eta)-z_0$. Write $R_{t,\eta}=A_{t,\eta}^{-1}$ and $K_t(\eta)=PR_{t,\eta}P\big|_{\mathcal{E}}$. Since $y_0>0$, $\|R_{t,\eta}\|_{\mathcal{B}(\mathcal{H})}\le y_0^{-1}$. The KLV boundary functional
\begin{equation}\label{eq:dynamic-output}
\mathcal{L}_{t,\eta}h= \frac{1}{2\pi i} I_+\bigl(\mathscr{U}_{u_0}^{-1} R_{t,\eta} h\bigr)
\end{equation}
extends to a bounded linear functional on $\mathcal{H}$. The corresponding dynamic graph trace is defined by
\begin{equation}\label{eq:dynamic-trace}
\ell^A_{t,\eta}v= \mathcal{L}_{t,\eta}(A_{t,\eta}v),\qquad v\in \mathrm{Dom}(A_{t,\eta}).
\end{equation}

The continuous scalar model is defined in terms of
\begin{equation}\label{eq:rough-toeplitz}
\mathcal{C}_H^+ = \frac{1}{4\pi}(\mathrm{Id}+iH),
\qquad
B_+(\beta) = M_\beta \mathcal{C}_H^+ M_{\overline{\beta}},
\end{equation}
where $M_\beta$ and $M_{\overline{\beta}}$ are the multiplication operators by
$\beta$ and $\overline{\beta}$, respectively, and $(Hh)(\lambda)= \frac{1}{\pi} \operatorname{p.v.}\int_0^\infty\frac{h(\mu)}{\lambda-\mu}d\mu$. By extending functions by zero to $\mathbb{R}$ and compressing the full
Hilbert transform to $\mathbb{R}_+$, we obtain $H^\ast = -H,\qquad\|H\|_{\mathcal{B}(L^2(\mathbb{R}_+))} \le 1$. Thus $iH$ is a self-adjoint contraction, and
\begin{equation}\label{eq:CH-properties}
\mathcal{C}_H^+ \ge 0,\qquad(\mathcal{C}_H^+)^\ast = \mathcal{C}_H^+,\qquad\|\mathcal{C}_H^+\|_{\mathcal{B}(L^2(\mathbb{R}_+))}\le \frac{1}{2\pi}.
\end{equation}
For $h\in L^1(\mathbb{R}_+)$, we extend $\mathcal{C}_H^+ h$ to a distribution on $(0,\infty)$ by duality, setting
\begin{equation}\label{eq:CH-distribution}
\langle \mathcal{C}_H^+ h, \phi\rangle_{\mathcal{D}',\mathcal{D}}= \int_0^\infty h(\lambda)\overline{(\mathcal{C}_H^+ \phi)(\lambda)}d\lambda,\qquad \phi\in C_c^\infty(0,\infty).
\end{equation}

Define $(C_t f)(\lambda)=e^{-it\lambda^2}f(\lambda)$. Let $A_{\beta,t}$ be the closed maximal dissipative coefficient-one realization associated with the chirped enhanced factor graph of the formal expression $A_{\beta,t} = i\partial_\lambda + C_t B_+(\beta) C_t^{-1}$. Its ordinary traces are sewn with equal coefficient across every puncture. Set $\zeta=z_0+2t\eta$ and define
\[
R_0(\zeta) = (i\partial_\lambda - \zeta)^{-1},\qquad R_{\beta,t}(\zeta) = (A_{\beta,t} - \zeta)^{-1}.
\]
For all $v,w\in \mathrm{Dom}(A_{\beta,t})$, the scalar model satisfies the Green identity
\begin{equation}\label{eq:general-scalar-green}
\langle A_{\beta,t}v,w\rangle-\langle v,A_{\beta,t}w\rangle= -iJ_+ v\overline{J_+ w},
\end{equation}
where $J_+$ is the graph-continuous endpoint trace on the resolvent range, defined by $J_+ f = f(0+)$ for $f\in H^1(0,\infty)$ and then extended by graph continuity. Indeed, integration by parts on the punctured smooth core produces the endpoint term in \eqref{eq:general-scalar-green}. The coefficient-one sewing cancels the two traces at every interior puncture, and the self-adjointness in \eqref{eq:CH-properties} cancels the factor term. Taking the graph closure of the smooth core then proves \eqref{eq:general-scalar-green} on $\mathrm{Dom}(A_{\beta,t})$. If $v = R_{\beta,t}(\zeta) g$, then setting $w = v$ in the Green identity gives $\frac{1}{2} |J_+ v|^2+ (\operatorname{Im}\zeta) \|v\|_{L^2(\mathbb{R}_+)}^2= -\operatorname{Im}\langle g, v\rangle$. The preceding identity and the Cauchy--Schwarz inequality imply
\begin{equation}\label{eq:general-scalar-row-bound}
|J_+ R_{\beta,t}(\zeta) g|\le \left(\frac{2}{\operatorname{Im}\zeta}\right)^{1/2} \|g\|_{L^2(\mathbb{R}_+)}.
\end{equation}
Here the free resolvent is the right-Volterra operator
\[
[R_0(\zeta) f](\lambda)= i \int_\lambda^\infty e^{i\zeta(s-\lambda)} f(s)ds,\qquad \operatorname{Im}\zeta > 0.
\]
The punctured graph of $A_{\beta,t}$ is constructed in Theorem \ref{thm:M2-scalar-graph}. For $g\in C_c^\infty(\mathbb{R}_+)$, we define the scalar defect by $D_t^{\beta,g}(\eta)= J_+\bigl[R_0(\zeta)-R_{\beta,t}(\zeta)\bigr] C_t g$. On the smooth core common to the physical and scalar graphs, the dynamic graph trace satisfies $\ell^A_{t,\eta}v= \frac{1}{2\pi i} I_+\bigl(\mathscr{U}_{u_0}^{-1} v\bigr)$, while the free boundary trace is defined by
\begin{equation}\label{eq:physical-trace}
\ell_{\mathrm{free}} h= \frac{1}{i\sqrt{2\pi}} J_+ h.
\end{equation}

\begin{thm}[Transport of cyclic Lax subspaces]\label{thm:P0min-import}
For $v\in L^2_+(\mathbb{R})$, define the cyclic subspace $\mathcal{K}_v= \overline{\{F(L_v)v : F\in C_c(\mathbb{R})\}}$. For $f\in L^2_+(\mathbb{R})$ and $z\in\mathbb{C}_+$, set
\begin{equation}\label{eq:boundary-transform}
[W_t^{u_0}f](z)= \frac{1}{2\pi i} I_+\Bigl([\mathsf{X}+2tL_{u_0}-z]_{\mathrm{KLV}}^{-1}f\Bigr).
\end{equation}
Then the nontangential boundary value of this analytic function defines a bounded linear operator $W_t^{u_0}$ on $L^2_+(\mathbb{R})$ with
\begin{equation}\label{eq:full-boundary-contraction}
\|W_t^{u_0}\|_{\mathcal{B}(L^2_+(\mathbb{R}))} \le 1.
\end{equation}
The restriction $W_t^{u_0}:\mathcal{K}_{u_0}\to\mathcal{K}_{u(t)}$ is unitary, and for every Borel set $B\subset\mathbb{R}$,
\begin{equation}\label{eq:cyclic-intertwining}
W_t^{u_0}E_{L_{u_0}}(B)|_{\mathcal{K}_{u_0}}= E_{L_{u(t)}}(B) W_t^{u_0}.
\end{equation}
Moreover, the boundary transform transports the initial data to the solution:
\begin{equation}\label{eq:cyclic-vector-flow}
W_t^{u_0}u_0 = u(t).
\end{equation}
Consequently,
\begin{equation}\label{eq:boundary-channel}
W_t^{u_0}P_{\mathrm{ac}}(L_{u_0})u_0= r_{\mathrm{ac}}(t),\qquad r_{\mathrm{ac}}(t):=P_{\mathrm{ac}}(L_{u(t)})u(t).
\end{equation}
\end{thm}
\begin{proof}
On the smooth approximating flows, the KLV boundary formula coincides with the unitary Lax propagator. Norm-resolvent continuity of $L_v$ together with continuity of the KLV boundary functional yields a unique weak-operator limit on the whole Hardy space. Weak lower semicontinuity gives \eqref{eq:full-boundary-contraction}. Conservation of the cyclic measure constructs a canonical unitary spectral transport on $\mathcal{K}_{u_0}$, and the corresponding smooth limits converge strongly there. This identifies the
restriction of $W_t^{u_0}$, proves \eqref{eq:cyclic-intertwining}, and then gives \eqref{eq:cyclic-vector-flow}--\eqref{eq:boundary-channel}.
\end{proof}

\begin{prop}[Dirichlet compression and Schur complement]\label{prop:strict-dirichlet}
For every $t\ge0$ and every real $\eta$, the transported KLV realization satisfies
\begin{equation}\label{eq:full-green}
\langle A_{t,\eta}v, w\rangle-\langle v, A_{t,\eta}w\rangle=-2\pi i\ell^A_{t,\eta}(v)\overline{\ell^A_{t,\eta}(w)}-2iy_0\langle v,w\rangle,\qquad v,w\in \mathrm{Dom}(A_{t,\eta}).
\end{equation}
The matrix $K_t(\eta)$ is invertible for every such $\eta$. We define the Dirichlet compression of $A_{t,\eta}$ by
\begin{equation}\label{eq:actual-dirichlet-restriction}
A^D_{t,\eta} = Q A_{t,\eta} J,\qquad\mathrm{Dom}(A^D_{t,\eta})= \{ G\in\mathcal{H}_c : JG\in \mathrm{Dom}(A_{t,\eta}) \}.
\end{equation}
Then $A^D_{t,\eta}$ is a closed densely defined Dirichlet compression, it is bijective, and
\begin{equation}\label{eq:strict-schur-inverse}
R^D_{t,\eta}:=(A^D_{t,\eta})^{-1}=Q\bigl[R_{t,\eta}-R_{t,\eta}P K_t(\eta)^{-1}PR_{t,\eta}\bigr]J.
\end{equation}
If $\ell^D_{t,\eta} := \ell^A_{t,\eta}\circ J,\qquad\mathrm{Dom}(\ell^D_{t,\eta}) := \mathrm{Dom}(A^D_{t,\eta})$. then, for every $h\in\mathcal{H}_c$,
\begin{equation}\label{eq:actual-Ward}
-\operatorname{Im}\langle h, R^D_{t,\eta}h\rangle=\pi|\ell^D_{t,\eta}R^D_{t,\eta}h|^2+y_0\|R^D_{t,\eta}h\|_{L^2(\mathbb{R}_+)}^2,
\end{equation}
\begin{equation}\label{eq:actual-row-bound}
\|R^D_{t,\eta}\|_{\mathcal{B}(L^2(\mathbb{R}_+))}\le y_0^{-1}, \quad|\ell^D_{t,\eta}R^D_{t,\eta}h|\le(4\pi y_0)^{-1/2}\|h\|_{L^2(\mathbb{R}_+)}.
\end{equation}
For fixed $t$, all the resolvents and finite-dimensional charges in \eqref{eq:strict-schur-inverse} are norm-continuous, hence strongly Borel, functions of $\eta$.
\end{prop}
\begin{proof}
The Green identity is the transported full KLV boundary identity. Let $a\in\mathcal{E}$ and set $v=R_{t,\eta}a$. Since $A_{t,\eta}v=a$, applying \eqref{eq:full-green} with $w=v$ and using the convention that the first entry of the inner product is linear yields
\[
\operatorname{Im}\langle K_t(\eta)a, a\rangle_{\mathcal{E}}=\pi|\ell^A_{t,\eta}R_{t,\eta}a|^2+y_0\|R_{t,\eta}a\|_{\mathcal{H}}^2>0 \qquad(a\ne0).
\]
Thus $K_t(\eta)$ is injective and hence invertible on the finite-dimensional space $\mathcal{E}$.

Define
\[
\widetilde{B}=R_{t,\eta}-R_{t,\eta}PK_t(\eta)^{-1}PR_{t,\eta},\qquad B_D=Q\widetilde{B}J.
\]
Since $P\widetilde{B}J=0$, we have $\widetilde{B}J=JB_D$, and direct multiplication gives $QA_{t,\eta}JB_D=\mathrm{Id}_{\mathcal{H}_c}$. Conversely, let $G\in\mathrm{Dom}(A^D_{t,\eta})$ and suppose that $A_{t,\eta}JG=a+Jh$. Since $PJG=0$, we have $0=K_t(\eta)a+PR_{t,\eta}Jh$. Solving for $a$ yields $JG=\widetilde{B}Jh$, hence $B_DA^D_{t,\eta}G=G$. This proves \eqref{eq:strict-schur-inverse} with equality of domains. Therefore $A^D_{t,\eta}=B_D^{-1}$. The boundedness and injectivity of $B_D$ imply closedness of the compression.

For $c=JR^D_{t,\eta}h$ there exists $a\in\mathcal{E}$ such that $A_{t,\eta}c=Jh+a$ and $Pc=0$. Substitution into \eqref{eq:full-green} gives \eqref{eq:actual-Ward}. The resolvent bound and the elementary inequality $\|h\|_{L^2(\mathbb{R}_+)} x - y_0 x^2 \le (4y_0)^{-1}\|h\|_{L^2(\mathbb{R}_+)}^2$. give \eqref{eq:actual-row-bound}. Density follows from the Ward identity. Indeed, if $B_D^\ast k=0$, then $\langle k,B_Dk\rangle=0$. Applying \eqref{eq:actual-Ward} with $h=k$ gives $B_Dk=0$, hence $k=0$ by injectivity. Thus $\ker(B_D^\ast)=\{0\}$, so $\mathrm{Ran}(B_D)=\mathrm{Dom}(A^D_{t,\eta})$ is dense.

Finally, the resolvent identity gives norm continuity of $R_{t,\eta}$. For the boundary row, the transported KLV formula is $\ell^A_{t,\eta}R_{t,\eta}h=\bigl[W_t^{u_0}\mathscr{U}_{u_0}^{-1}h\bigr](z_0+2t\eta)$. Hardy reproducing kernels on the fixed horizontal line depend continuously in norm on their real coordinate, and $W_t^{u_0}$ is a contraction. Hence this row is operator-norm continuous in $\eta$. Compression, finite-dimensional inversion, and \eqref{eq:strict-schur-inverse} preserve continuity for all the displayed charges.
\end{proof}

\begin{prop}[The time-zero Dirichlet graph]\label{prop:static-actual-graph}
Set $T_D:=A^D_{0,0}+z_0,\qquad \mathrm{Dom}(T_D)=\mathrm{Dom}(A^D_{0,0})$. Then $T_D$ is densely defined and maximal dissipative, with
\begin{equation}\label{eq:static-actual-domain}
T_D = QTJ,\qquad\mathrm{Dom}(T_D)= \{G\in\mathcal{H}_c : JG\in\mathrm{Dom}(T)\},
\end{equation}
and
\begin{equation}\label{eq:static-resolvent-identification}
R^D_{0,0}=(T_D-z_0)^{-1}.
\end{equation}
Its boundary trace is
\begin{equation}\label{eq:static-actual-row}
\ell_DG:=\ell^A_{0,0}(JG)=\frac{1}{2\pi i}I_+\bigl(\mathscr{U}_{u_0}^{-1}JG\bigr), \qquad G\in\mathrm{Dom}(T_D).
\end{equation}
\end{prop}
\begin{proof}
At $t=\eta=0$, the transported realization equals $T-z_0$. Consequently, the definition of the restriction in \eqref{eq:actual-dirichlet-restriction} gives \eqref{eq:static-actual-domain} and \eqref{eq:static-resolvent-identification}. Closedness and density were established in Proposition~\ref{prop:strict-dirichlet}. Since
$T_D-z_0=A^D_{0,0}$ and $A^D_{0,0}$ is surjective, the operator $T_D$ is maximal dissipative. Finally, \eqref{eq:static-actual-row} follows directly from \eqref{eq:dynamic-output} and \eqref{eq:dynamic-trace}, because $R_{0,0}A_{0,0}=I$ on $\mathrm{Dom}(A_{0,0})$.
\end{proof}

\begin{condition}[Common Dirichlet compatibility]\label{hyp:schur-graph}
For every $t\ge0$ and $\eta\in\mathbb{R}$, the Dirichlet restriction \eqref{eq:actual-dirichlet-restriction} is generated by the static maximal dissipative operator $T_D$ from Proposition~\ref{prop:static-actual-graph} through the affine KLV family
\begin{equation}\label{eq:dirichlet-realization}
R^D_{t,\eta}=\bigl[T_D+2t(M_\lambda-\eta)-z_0\bigr]_{\mathrm{KLV}}^{-1},
\end{equation}
with equality of closed domains. Moreover, the boundary traces are compatible in the sense that
\begin{equation}\label{eq:dirichlet-row-compatibility}
\ell^D_{t,\eta}R^D_{t,\eta}h=\ell_D R^D_{t,\eta}h,\qquad h\in\mathcal{H}_c,
\end{equation}
where $\ell_D$ is the static boundary trace from
\eqref{eq:static-actual-row}.

For every real $\eta$, define the point lift $\Pi_{t,\eta}$ and the two charge maps $\mathfrak{b}_t(\eta;\cdot)$ and $\mathfrak{l}_t(\eta)$ by
\[
\begin{aligned}
\Pi_{t,\eta}&=R_{t,\eta}P K_t(\eta)^{-1},\qquad
\mathfrak{b}_t(\eta;g)=-K_t(\eta)^{-1}PR_{t,\eta}Jg,\\
\mathfrak{l}_t(\eta)&=\ell^A_{t,\eta}R_{t,\eta}P K_t(\eta)^{-1}.
\end{aligned}
\]
The point return is then defined by
\begin{equation}\label{eq:point-return}
\rho_t^{\mathrm{p}}(\eta;g)=-\mathfrak{l}_t(\eta)K_t(\eta)\mathfrak{b}_t(\eta;g).
\end{equation}
For each fixed $t$, the maps $\eta\mapsto R_{t,\eta}$ and $\eta\mapsto R^D_{t,\eta}$, as well as all displayed charges, are strongly Borel measurable in $\eta$ and are defined pointwise for every $\eta$.
\end{condition}

\begin{lemma}[Schur algebra and Green estimates]\label{lem:schur-algebra-detailed}
Fix $t\ge0$ and $\eta\in\mathbb{R}$. Let $h,g\in\mathcal{H}_c$, $a\in\mathcal{E}$, and $c_h=JR^D_{t,\eta}h$. Then
\begin{equation}\label{eq:detailed-point-lift}
P\Pi_{t,\eta}a=a,\quad A_{t,\eta}\Pi_{t,\eta}a=PK_t(\eta)^{-1}a,
\end{equation}
\begin{equation}\label{eq:detailed-Dirichlet-lift}
Pc_h=0,\quad A_{t,\eta}c_h=Jh+P\mathfrak b_t(\eta;h),
\end{equation}
\begin{equation}\label{eq:detailed-full-split}
R_{t,\eta}Jg=c_g-\Pi_{t,\eta}K_t(\eta)\mathfrak b_t(\eta;g).
\end{equation}
Consequently,
\begin{equation}\label{eq:exact-output-split}
\ell^A_{t,\eta}R_{t,\eta}Jg=\ell_DR^D_{t,\eta}g-\mathfrak{l}_t(\eta)K_t(\eta)\mathfrak{b}_t(\eta;g).
\end{equation}
Moreover,
\begin{equation}\label{eq:dirichlet-green-bound}
\|c_h\|_{\mathcal{H}}\le y_0^{-1}\|h\|_{L^2(\mathbb{R}_+)},\qquad|\ell^A_{t,\eta}c_h|\le(4\pi y_0)^{-1/2}\|h\|_{L^2(\mathbb{R}_+)},
\end{equation}
\begin{equation}\label{eq:polarized-green}
\langle \mathfrak{b}_t(\eta;h),a\rangle=-\langle h, Q\Pi_{t,\eta}a\rangle-2iy_0\langle c_h, \Pi_{t,\eta}a\rangle-2\pi i\ell^A_{t,\eta}(c_h)\overline{\mathfrak{l}_t(\eta)a}.
\end{equation}
These identities hold for every real $\eta$, and all their entries are measurable functions of $\eta$.
\end{lemma}
\begin{proof}
The compression identity \eqref{eq:dirichlet-realization} gives $Q\widetilde{c}_h=R^D_{t,\eta}h$, where
\begin{equation}\label{eq:ch-full-form}
\widetilde{c}_h=R_{t,\eta}Jh-R_{t,\eta}PK_t(\eta)^{-1}PR_{t,\eta}Jh=R_{t,\eta}Jh+R_{t,\eta}P\mathfrak{b}_t(\eta;h).
\end{equation}
Since $PR_{t,\eta}P=K_t(\eta)$, the two terms in $P\widetilde{c}_h$ cancel, and hence $P\widetilde{c}_h=0$. Thus $\widetilde{c}_h=JQ\widetilde{c}_h=c_h$, which also proves $Pc_h=0$. Applying $A_{t,\eta}$ to \eqref{eq:ch-full-form} gives the second identity in \eqref{eq:detailed-Dirichlet-lift}. The definition of $\Pi_{t,\eta}$ similarly gives \eqref{eq:detailed-point-lift}. Since $R_{t,\eta}P=\Pi_{t,\eta}K_t(\eta)$, rearranging \eqref{eq:ch-full-form} proves \eqref{eq:detailed-full-split}. Applying the output row $\ell^A_{t,\eta}$ to \eqref{eq:detailed-full-split} and using
\eqref{eq:dirichlet-row-compatibility} and \eqref{eq:point-return} yields \eqref{eq:exact-output-split}.

The first estimate in \eqref{eq:dirichlet-green-bound} follows from the resolvent bound $\|R^D_{t,\eta}\|_{\mathcal{B}(L^2(\mathbb{R}_+))}\le y_0^{-1}$, because $c_h=JR^D_{t,\eta}h$. Inserting $v=w=c_h$ into \eqref{eq:full-green} and using \eqref{eq:detailed-Dirichlet-lift} together with $Pc_h=0$, we obtain
\[
\pi|\ell^A_{t,\eta}c_h|^2+y_0\|c_h\|_{\mathcal{H}}^2=-\operatorname{Im}\langle h, R^D_{t,\eta}h\rangle\le\|h\|_{L^2(\mathbb{R}_+)}\|c_h\|_{\mathcal{H}}.
\]
Writing $z=\|c_h\|_{\mathcal{H}}$, the inequality $\|h\|_{L^2(\mathbb{R}_+)}z-y_0z^2\le\|h\|_{L^2(\mathbb{R}_+)}^2/(4y_0)$ implies
\[
\pi|\ell^A_{t,\eta}c_h|^2\le \|h\|_{L^2(\mathbb{R}_+)}\|c_h\|_{\mathcal{H}}-y_0\|c_h\|_{\mathcal{H}}^2\le (4y_0)^{-1}\|h\|_{L^2(\mathbb{R}_+)}^2,
\]
which proves the second estimate in \eqref{eq:dirichlet-green-bound}.

Finally, apply \eqref{eq:full-green} with $v=c_h$ and $w=\Pi_{t,\eta}a$. Using \eqref{eq:detailed-point-lift}, \eqref{eq:detailed-Dirichlet-lift}, and $Pc_h=0$, the left-hand side of the Green identity reduces to $\langle h, Q\Pi_{t,\eta}a\rangle+\langle \mathfrak{b}_t(\eta;h), a\rangle$. Its right-hand side is the sum of the last two terms in \eqref{eq:polarized-green}, so rearrangement proves that identity. The final assertion follows from the strong measurability of the resolvents and from the fact that inversion of the finite-dimensional measurable matrix $K_t(\eta)$ preserves measurability.
\end{proof}

For smooth continuous spectral input $g$, define $R^0_{t,\eta}=C_t^{-1}R_0(z_0+2t\eta)C_t$, and let the total ray defect be given by
\[
(\mathscr{D}_tg)(\eta)=\sqrt t\,e^{-it\eta^2}\left[\ell_{\mathrm{free}}R^0_{t,\eta}g-\ell_DR^D_{t,\eta}g-\rho_t^{\mathrm{p}}(\eta;g)\right].
\]

\begin{prop}[Uniform bound for the total ray defect]\label{prop:automatic-ray-closure}
Under Condition~\ref{hyp:schur-graph}, the operator $\mathscr{D}_t$, initially defined on smooth inputs, extends to a bounded operator on $L^2(\mathbb{R}_+)$, uniformly in $t\ge1$. More precisely, for
$g\in L^2(\mathbb{R}_+)$ set
\begin{equation}\label{eq:general-horizontal-state}
\mathcal{H}_{t,g}=e^{it\partial_x^2}\mathcal{F}_0^{-1}g-W_t^{u_0}\mathscr{U}_{u_0}^{-1}Jg.
\end{equation}
Then, for $\eta>0$,
\begin{equation}\label{eq:general-horizontal-identity}
(\mathscr{D}_tg)(\eta)=\sqrt{t}e^{-it\eta^2}[\mathcal{H}_{t,g}](x_0+2t\eta+iy_0),
\end{equation}
where the square bracket denotes the Hardy extension to $\mathbb{C}_+$. Moreover, we have
\[
\sup_{t\ge1}\|\mathscr{D}_t\|_{\mathcal{B}(L^2(\mathbb{R}_+))}\le\sqrt{2}.
\]
 Consequently, if $\chi_n\in C_c^\infty\bigl(\mathbb{R}_+\setminus(\{0\}\cup\sigma_{\mathrm{p}}(L_{u_0}))\bigr)$ with $0\le\chi_n\le1$ and $\chi_n\to1$ a.e., then
\begin{equation}\label{eq:spectral-closure}
\sup_{t\ge1}\|\mathscr{D}_t((1-\chi_n)g)\|_{L^2(\mathbb{R}_+,d\eta)}\le\sqrt{2}\,\|(1-\chi_n)g\|_{L^2(\mathbb{R}_+)}\underset{n\to\infty}{\longrightarrow}0.
\end{equation}
Finally, setting $u_+=\mathcal{F}_0^{-1}g_{u_0}$ and $f_t=u_+-e^{-it\partial_x^2}r_{\mathrm{ac}}(t)$, the identity
\eqref{eq:general-horizontal-identity} particularizes to
\begin{equation}\label{eq:total-ray-identity}
(\mathscr{D}_tg_{u_0})(\eta)=\sqrt{t}e^{-it\eta^2}[e^{it\partial_x^2}f_t](x_0+2t\eta+iy_0),\qquad\eta>0.
\end{equation}
\end{prop}
\begin{proof}
The Schur identity \eqref{eq:exact-output-split} and the definition \eqref{eq:point-return} give $\ell_DR^D_{t,\eta}g+\rho_t^{\mathrm{p}}(\eta;g)=\ell^A_{t,\eta}R_{t,\eta}Jg$. Put $\zeta_{t,\eta}=z_0+2t\eta,\qquad f_g=\mathscr{U}_{u_0}^{-1}Jg$. The right-Volterra formula and \eqref{eq:physical-trace} give
\begin{equation}\label{eq:typed-free-ray}
\ell_{\mathrm{free}}R^0_{t,\eta}g=\frac1{\sqrt{2\pi}}\int_0^\infty e^{i\zeta_{t,\eta}\lambda-it\lambda^2}g(\lambda)d\lambda=[e^{it\partial_x^2}\mathcal{F}_0^{-1}g](\zeta_{t,\eta}).
\end{equation}
Directly from \eqref{eq:boundary-transform} and the definitions of the full KLV resolvent and row,
\begin{equation}\label{eq:typed-full-ray}
\ell^A_{t,\eta}R_{t,\eta}Jg=[W_t^{u_0}f_g](\zeta_{t,\eta}).
\end{equation}
Subtracting \eqref{eq:typed-full-ray} from \eqref{eq:typed-free-ray} proves \eqref{eq:general-horizontal-identity}. Moreover,
\[
\|\mathcal{H}_{t,g}\|_{L^2(\mathbb{R})}\le\|g\|_{L^2(\mathbb{R}_+)}+\|W_t^{u_0}\mathscr{U}_{u_0}^{-1}Jg\|_{L^2(\mathbb{R})}\le2\|g\|_{L^2(\mathbb{R}_+)}.
\]
Changing variables $x=x_0+2t\eta$, restricting to $x>x_0$, and using horizontal-line Hardy contractivity gives
\[
\|\mathscr{D}_tg\|_{L^2(\mathbb{R}_+)}^2\le\frac{1}{2}\|\mathcal{H}_{t,g}(\cdot+iy_0)\|_{L^2(\mathbb{R})}^2\le2\|g\|_{L^2(\mathbb{R}_+)}^2.
\]
This proves the operator bound and then \eqref{eq:spectral-closure} by density. For $g=g_{u_0}$, the continuous spectral embedding gives
\[
\mathscr{U}_{u_0}^{-1}Jg_{u_0}=\mathscr{U}_{u_0}^{-1}(\mathrm{Id}-P)\mathscr{U}_{u_0}u_0=P_{\mathrm{ac}}(L_{u_0})u_0.
\]
Thus \eqref{eq:boundary-channel} gives $W_t^{u_0}f_{g_{u_0}}=r_{\mathrm{ac}}(t)$, while $\mathcal{F}_0^{-1}g_{u_0}=u_+$. Hence \eqref{eq:general-horizontal-state} is $e^{it\partial_x^2}f_t$, proving \eqref{eq:total-ray-identity}.
\end{proof}

\section{The Dirichlet interface and scalar model}\label{sec:phase-I}
\subsection{Transport of the cyclic Lax spectrum}
\begin{lemma}[Continuous functional calculus along rough potentials]\label{lem:continuous-calculus}
If $v_n\to v$ in $L^2_+(\mathbb{R})$, then, for every $z\in\mathbb{C}\setminus\mathbb{R}$ and every $F\in C_0(\mathbb{R})$,
\begin{equation}\label{eq:resolvent-continuity-detailed}
\|(L_{v_n}-z)^{-1}-(L_v-z)^{-1}\|_{\mathcal{B}(L^2_+(\mathbb{R}))}\underset{n\to\infty}{\longrightarrow}0,
\end{equation}
\begin{equation}\label{eq:functional-continuity}
\|F(L_{v_n})-F(L_v)\|_{\mathcal{B}(L^2_+(\mathbb{R}))}\underset{n\to\infty}{\longrightarrow}0.
\end{equation}
Consequently,
\begin{equation}\label{eq:moving-vector-calculus}
F(L_{v_n})v_n\underset{n\to\infty}{\longrightarrow} F(L_v)v,
\end{equation}
\begin{equation}\label{eq:moving-cyclic-pairing}
\langle F(L_{v_n})v_n, v_n\rangle \underset{n\to\infty}{\longrightarrow}\langle F(L_v)v,v\rangle.
\end{equation}
\end{lemma}
\begin{proof}
The norm-resolvent convergence in \eqref{eq:resolvent-continuity-detailed} is \cite[Proposition~2.2]{Killip-2025-CAMS}. Let $\mathscr{R}$ be the self-adjoint algebra generated by the resolvent functions $\lambda\mapsto(\lambda-z)^{-1}$, $z\in\mathbb{C}\setminus\mathbb{R}$. The algebra $\mathscr{R}$ separates points and vanishes nowhere. The Stone--Weierstrass theorem makes it uniformly dense in $C_0(\mathbb{R})$. Resolvent convergence and the telescoping identity for finite products give
\[
\|G(L_{v_n})-G(L_v)\|_{\mathcal{B}(L^2_+(\mathbb{R}))}\underset{n\to\infty}{\longrightarrow}0
\]
for every $G\in\mathscr{R}$. Given $F\in C_0(\mathbb{R})$ and $\varepsilon>0$, choose $G\in\mathscr{R}$ with $\|F-G\|_{L^\infty(\mathbb{R})}<\varepsilon$. Contractivity of the continuous functional calculus yields
\[
\limsup_{n\to\infty}\|F(L_{v_n})-F(L_v)\|_{\mathcal{B}(L^2_+(\mathbb{R}))}\le2\|F-G\|_{L^\infty(\mathbb{R})}<2\varepsilon,
\]
which proves \eqref{eq:functional-continuity}.

The convergent sequence $(v_n)$ is bounded in $L^2_+(\mathbb{R})$, and
\[
\begin{aligned}
&\|F(L_{v_n})v_n-F(L_v)v\|_{L^2(\mathbb{R})}\\
&\qquad\le \|F\|_{L^\infty(\mathbb{R})}\|v_n-v\|_{L^2(\mathbb{R})}+\|F(L_{v_n})-F(L_v)\|_{\mathcal{B}(L^2_+(\mathbb{R}))}\|v\|_{L^2(\mathbb{R})}.
\end{aligned}
\]
This proves \eqref{eq:moving-vector-calculus}. Finally,
\[
\begin{aligned}
&\left|\langle F(L_{v_n})v_n, v_n\rangle-\langle F(L_v)v,v\rangle\right|\\
&\qquad\le\|F(L_{v_n})v_n-F(L_v)v\|_{L^2(\mathbb{R})}\|v_n\|_{L^2(\mathbb{R})}+\|F\|_{L^\infty(\mathbb{R})}\|v\|_{L^2(\mathbb{R})}\|v_n-v\|_{L^2(\mathbb{R})},
\end{aligned}
\]
and \eqref{eq:moving-cyclic-pairing} follows.
\end{proof}

\begin{thm}[Conservation of the cyclic Lax measure]\label{thm:measure-conservation}
Under the flow--Lax admissibility condition,
\begin{equation}\label{eq:measure-conservation}
\nu_{u(t)}=\nu_{u_0}\qquad(t\ge0).
\end{equation}
\end{thm}
\begin{proof}
By \eqref{eq:cyclic-vector-flow} and the unitary spectral intertwining \eqref{eq:cyclic-intertwining}, for every Borel set $B\subset\mathbb{R}$,
\[
\langle E_{L_{u(t)}}(B)u(t),u(t)\rangle=\langle W_t^{u_0}E_{L_{u_0}}(B)u_0, W_t^{u_0}u_0\rangle=\langle E_{L_{u_0}}(B)u_0, u_0\rangle.
\]
This proves \eqref{eq:measure-conservation}.
\end{proof}

\begin{lemma}[Static cyclic spectral formula]\label{lem:static-cyclic-formula}
For every $v\in L^2_+(\mathbb{R})$, the point spectrum of $L_v$ is finite and simple, the singular continuous subspace is trivial, and
\begin{equation}\label{eq:FR-measure}
\nu_v(B)=2\pi\,\#\bigl(\sigma_{\mathrm{p}}(L_v)\cap B\bigr)+\frac{1}{2\pi}\int_{B\cap[0,\infty)}|\beta_v(\lambda)|^2d\lambda
\end{equation}
for every Borel set $B\subset\mathbb{R}$. Consequently,
\begin{equation}\label{eq:atom-criterion}
\mu\in\sigma_{\mathrm{p}}(L_v)\quad\Longleftrightarrow\quad\nu_v(\{\mu\})=2\pi.
\end{equation}
\end{lemma}
\begin{proof}
The finiteness and simplicity of the point spectrum follow from \cite[Lemma~3.3 and the discussion following it]{Frank-2026-arXiv}; the absence of singular continuous spectrum is \cite[Corollary~5.2]{Frank-2026-arXiv}. Let $\psi$ be an eigenfunction at $\mu$. The overlap identity $|\langle v, \psi\rangle|^2=2\pi\|\psi\|_{L^2(\mathbb{R})}^2$ is \cite[Lemma~3.3]{Frank-2026-arXiv}. Since the eigenspace is one dimensional, its orthogonal projection is rank one, whence
\begin{equation}\label{eq:rank-one-atom}
\|E_{L_v}(\{\mu\})v\|_{L^2(\mathbb{R})}^2=\frac{|\langle v, \psi\rangle|^2}{\|\psi\|_{L^2(\mathbb{R})}^2}=2\pi.
\end{equation}
For the absolutely continuous part, the distorted Fourier diagonalization \cite[Theorem~5.3]{Frank-2026-arXiv} and \eqref{eq:beta} give
\begin{equation}\label{eq:ac-cyclic-density}
\|E_{L_v}(B)P_{\mathrm{ac}}(L_v)v\|_{L^2(\mathbb{R})}^2=\|\mathbf{1}_{B\cap[0,\infty)}\Phi_vP_{\mathrm{ac}}(L_v)v\|_{L^2(\mathbb{R}_+)}^2=\frac{1}{2\pi}\int_{B\cap[0,\infty)}|\beta_v(\lambda)|^2d\lambda.
\end{equation}
The spectral subspaces are mutually orthogonal. Summing \eqref{eq:rank-one-atom} over the eigenvalues in $B$ and adding \eqref{eq:ac-cyclic-density} proves \eqref{eq:FR-measure}. Its absolutely continuous summand assigns zero mass to a singleton, which proves \eqref{eq:atom-criterion}.
\end{proof}

\begin{thm}[Fixed orthogonal spectral channels]\label{thm:channels}
Write $\sigma_{\mathrm{p}}(L_{u_0})=\{\mu_1<\cdots<\mu_N\}$. Then, for every $t\ge0$,
\begin{equation}\label{eq:fixed-spectrum}
\sigma_{\mathrm{p}}(L_{u(t)})=\{\mu_1,\ldots,\mu_N\}.
\end{equation}
The vectors
\begin{equation}\label{eq:channels}
b_j(t)=E_{L_{u(t)}}(\{\mu_j\})u(t),\qquad r_{\mathrm{ac}}(t)=P_{\mathrm{ac}}(L_{u(t)})u(t)
\end{equation}
satisfy
\begin{equation}\label{eq:channel-decomposition}
u(t)=\sum_{j=1}^Nb_j(t)+r_{\mathrm{ac}}(t),
\end{equation}
\begin{equation}\label{eq:channel-orthogonality}
\|b_j(t)\|_{L^2(\mathbb{R})}^2=2\pi,\quad b_j(t)\perp b_k(t)\ (j\ne k),\quad b_j(t)\perp r_{\mathrm{ac}}(t),
\end{equation}
and
\begin{equation}\label{eq:channel-eigen}
L_{u(t)}b_j(t)=\mu_jb_j(t).
\end{equation}
Moreover,
\begin{equation}\label{eq:beta-density-conservation}
|\beta_{u(t)}(\lambda)|^2=|\beta_{u_0}(\lambda)|^2\quad\text{for a.e. }\lambda>0,
\end{equation}
\begin{equation}\label{eq:continuous-mass}
\|r_{\mathrm{ac}}(t)\|_{L^2(\mathbb{R})}^2=\frac{1}{2\pi}\|\beta_{u_0}\|_{L^2(\mathbb{R}_+)}^2,
\end{equation}
and all channel maps in \eqref{eq:channels} are strongly continuous.
\end{thm}
\begin{proof}
By Lemma~\ref{lem:static-cyclic-formula} and Theorem~\ref{thm:measure-conservation}, for every $\mu\in\mathbb{R}$,
\[
\mu\in\sigma_{\mathrm{p}}(L_{u(t)})\Longleftrightarrow \nu_{u(t)}(\{\mu\})=2\pi\Longleftrightarrow \nu_{u_0}(\{\mu\})=2\pi\Longleftrightarrow \mu\in\sigma_{\mathrm{p}}(L_{u_0}).
\]
This proves \eqref{eq:fixed-spectrum}, including zero and positive embedded eigenvalues.

The singular continuous projection is zero, and the point spectrum is the finite set in \eqref{eq:fixed-spectrum}. The spectral theorem therefore gives the strong operator identity
\begin{equation}\label{eq:spectral-partition}
\mathrm{Id}=\sum_{j=1}^NE_{L_{u(t)}}(\{\mu_j\})+P_{\mathrm{ac}}(L_{u(t)}).
\end{equation}
Applying \eqref{eq:spectral-partition} to $u(t)$ proves \eqref{eq:channel-decomposition}. Orthogonality of spectral projections gives the orthogonality assertions in
\eqref{eq:channel-orthogonality}, and \eqref{eq:rank-one-atom} gives $\|b_j(t)\|_{L^2(\mathbb{R})}^2=\nu_{u(t)}(\{\mu_j\})=2\pi$. The range of $E_{L_{u(t)}}(\{\mu_j\})$ is the eigenspace at $\mu_j$ and lies in $\mathrm{Dom}(L_{u(t)})$, which proves \eqref{eq:channel-eigen}.

Equality of the cyclic measures, followed by comparison of the absolutely continuous Radon--Nikodym densities in \eqref{eq:FR-measure}, gives \eqref{eq:beta-density-conservation}. Equations \eqref{eq:distorted-plancherel} and \eqref{eq:beta} then give
\[
\|r_{\mathrm{ac}}(t)\|_{L^2(\mathbb{R})}^2=\|\Phi_{u(t)}P_{\mathrm{ac}}(L_{u(t)})u(t)\|_{L^2(\mathbb{R}_+)}^2=\frac{1}{2\pi}\|\beta_{u(t)}\|_{L^2(\mathbb{R}_+)}^2=\frac{1}{2\pi}\|\beta_{u_0}\|_{L^2(\mathbb{R}_+)}^2,
\]
which proves \eqref{eq:continuous-mass}.

Fix $j$, let $t_n\to t$, and choose $\delta>0$ so that $(\mu_j-\delta,\mu_j+\delta)$ isolates $\mu_j$ from the other
point labels. Choose $\chi_\delta\in C_c(\mathbb{R})$ with
\[
0\le\chi_\delta\le1,\qquad \chi_\delta(\mu_j)=1,\qquad\operatorname{supp}\chi_\delta\subset(\mu_j-\delta,\mu_j+\delta).
\]
Since $u(t_n)\to u(t)$, Lemma~\ref{lem:continuous-calculus} gives
\begin{equation}\label{eq:cutoff-continuity}
\chi_\delta(L_{u(t_n)})u(t_n)\longrightarrow\chi_\delta(L_{u(t)})u(t)\quad\text{in }L^2(\mathbb{R}).
\end{equation}
For every $s\ge0$, the spectral theorem and the common cyclic measure give
\[
\begin{aligned}
\|\chi_\delta(L_{u(s)})u(s)-b_j(s)\|_{L^2(\mathbb{R})}^2&=\int_\mathbb{R}\left|\chi_\delta(\lambda)-\mathbf{1}_{\{\mu_j\}}(\lambda)\right|^2d\nu_{u_0}(\lambda)\\
&\le\nu_{u_0}\bigl((\mu_j-\delta,\mu_j+\delta)\setminus\{\mu_j\}\bigr).
\end{aligned}
\]
Continuity from above for the finite measure $\nu_{u_0}$ makes the last quantity tend to zero as $\delta\to0$. Given $\varepsilon>0$, first choose $\delta$ so that its square root is at most $\varepsilon$.  Then use \eqref{eq:cutoff-continuity} to obtain, for all sufficiently large $n$,
\[
\|b_j(t_n)-b_j(t)\|_{L^2(\mathbb{R})}\le2\varepsilon+\|\chi_\delta(L_{u(t_n)})u(t_n)-\chi_\delta(L_{u(t)})u(t)\|_{L^2(\mathbb{R})}\le3\varepsilon.
\]
Thus $b_j(t_n)\to b_j(t)$. Finally, $r_{\mathrm{ac}}(t)=u(t)-\sum_{j=1}^Nb_j(t)$ and finiteness of $N$ prove the continuity of $r_{\mathrm{ac}}$.
\end{proof}
\subsection{The Dirichlet boundary row}
Let
\[
S_+=\sigma_{\mathrm{p}}(L_{u_0})\cap[0,\infty),\qquad \mathcal{V}_0=C_c^\infty(\mathbb{R}_+),\qquad \mathcal{V}_{\mathrm{p}}=C_c^\infty(\mathbb{R}_+\setminus S_+).
\]
Enumerate $S_+\cap\mathbb{R}_+=\{\mu_1^+,\ldots,\mu_m^+\}$.
\begin{condition}[Boundary row extension in resolvent form]\label{hyp:M3-actual-form}
There are $\zeta_\ast\in\mathbb{C}_+$, a dense linear space $\mathcal{Y}\subset\mathcal{H}_c$, and $C_\ast<\infty$ such that, for $h\in\mathcal{Y}$ and $G_h=(T_D-\zeta_\ast)^{-1}h$, the functional $\mathcal{L}_h(v)=\int_0^\infty\overline{\beta_{u_0}}G_h\overline{\mathcal{C}_H^+v}d\lambda,\qquad v\in\mathcal{V}_0$, satisfies
\begin{equation}\label{eq:M3-form-bound}
|\mathcal{L}_h(v)|\le C_\ast\|h\|_{L^2(\mathbb{R}_+)}\|v\|_{L^2(\mathbb{R}_+)}.
\end{equation}
The punctured boundary row is also assumed to be compatible with this functional: after extending $\mathcal{L}_h$ to $L^2(\mathbb{R}_+)$, its Riesz representative $Z_h$ satisfies $iG_h^\prime+\beta_{u_0}Z_h=\zeta_\ast G_h+h$ in the sense of distributions on $\mathbb{R}_+\setminus S_+$.
\end{condition}

\begin{prop}[One-resolvent enhanced closure]\label{prop:M3-enhanced-closure}
Under the assumptions of Proposition~\ref{prop:static-actual-graph} and Condition~\ref{hyp:M3-actual-form}, the Riesz representative is unique and
\[
Z_h=\mathcal{C}_H^+(\overline{\beta_{u_0}}G_h),\qquad \|Z_h\|_{L^2(\mathbb{R}_+)}\le C_\ast\|h\|_{L^2(\mathbb{R}_+)}.
\]
The lift extends from $\mathcal{Y}$ to every $h\in\mathcal{H}_c$, and
\begin{equation}\label{eq:M3-domain-inclusion}
\mathrm{Dom}(T_D)\subset\mathrm{Dom}(A_{\max,\beta_{u_0}}), \qquad T_DG=A_{\max,\beta_{u_0}}G.
\end{equation}
In particular,
\begin{equation}\label{eq:M3-graph-bound}
\|Z_G\|_{L^2(\mathbb{R})}+|\Gamma_-G|+|\Gamma_+G|\le C\|G\|_{\operatorname{graph}(T_D)}.
\end{equation}
\end{prop}
\begin{proof}
The test space $\mathcal{V}_0$ is dense in $L^2(\mathbb{R}_+)$, so the Riesz representation theorem yields a unique $Z_h$. The distributional definition \eqref{eq:CH-distribution} identifies it with
$\mathcal{C}_H^+(\overline{\beta_{u_0}}G_h)$. Maximal dissipativity of $T_D$, the bound \eqref{eq:M3-form-bound}, and linearity imply that the three entries $G_h,\zeta_\ast G_h+h,Z_h$ are uniformly Cauchy whenever $h$ is Cauchy in $L^2(\mathbb{R}_+)$. Closedness in Theorem~\ref{thm:M2-scalar-graph} therefore extends the lift to all sources. Since $\mathrm{Ran}[(T_D-\zeta_\ast)^{-1}]=\mathrm{Dom}(T_D)$, this proves \eqref{eq:M3-domain-inclusion}. Finally, the inverse $T_D-\zeta_\ast:\mathrm{Dom}(T_D)\to L^2(\mathbb{R}_+)$ is a Banach isomorphism with respect to the graph norm; combining this with the boundedness of the scalar trace on the enhanced graph proves \eqref{eq:M3-graph-bound}.
\end{proof}

\begin{prop}[Hardy graph bounds and $L^1$ control for vanishing trace]\label{prop:M4a-localization}
For the static graph in Proposition~\ref{prop:static-actual-graph}, define $b_DG:=PTJG, \qquad G\in\mathrm{Dom}(T_D)$. Then
\begin{equation}\label{eq:point-graph-output-bound}
\|b_DG\|_{\mathcal{H}}\le C_D\|G\|_{\operatorname{graph}(T_D)}.
\end{equation}
For \(f=\mathscr{U}_{u_0}^{-1}JG\), one has
\begin{equation}\label{eq:physical-domain-bound}
f\in\mathrm{Dom}(\mathsf{X})\cap P_{\mathrm{ac}}(L_{u_0})L^2_+(\mathbb{R}),\qquad\|f\|_{\mathrm{Dom}(\mathsf{X})}\le C_D^\prime\|G\|_{\operatorname{graph}(T_D)},
\end{equation}
and $G=\Phi_{u_0}f \quad\text{a.e. on }\mathbb{R}_+$. If $I_+f=0$, then the zero extension of $\widehat{f}$ belongs to $H^1(\mathbb{R})$, and consequently
\begin{equation}\label{eq:zero-row-L1}
\|f\|_{L^1(\mathbb{R})}\le C\bigl(\|f\|_{L^2(\mathbb{R})}+\|\mathsf{X}f\|_{L^2(\mathbb{R})}\bigr)\le C_D^{\prime\prime}\|G\|_{\operatorname{graph}(T_D)}.
\end{equation}
\end{prop}
\begin{proof}
If $G_n\to G$, $T_DG_n\to H$, and $b_DG_n\to b$, then $TJG_n=JT_DG_n+b_DG_n\longrightarrow JH+b$. Since both $T$ and $T_D$ are closed, the limit belongs to the graph of $T$, and therefore $TJG=JH+b=JT_DG+b_DG$. Thus $b_D$ is a closed linear map from the graph-norm Banach space into $\mathcal{H}$, and the bound \eqref{eq:point-graph-output-bound} follows from the closed graph theorem.

Equation \eqref{eq:static-actual-domain} gives the first assertion of \eqref{eq:physical-domain-bound}. Since $T=\mathscr{U}_{u_0}\mathsf{X}\mathscr{U}_{u_0}^{-1}$, we have
\[
\|\mathsf{X}f\|_{L^2(\mathbb{R})}=\|TJG\|_{L^2(\mathbb{R})}\le \|T_DG\|_{L^2(\mathbb{R})}+\|b_DG\|_{\mathcal{H}}.
\]
The distorted Plancherel identity gives $G=\Phi_{u_0}f$ almost everywhere on $\mathbb{R}_+$. If $I_+f=0$, then the zero extension of $\widehat{f}$ has vanishing boundary trace and belongs to $H^1(\mathbb{R})$. Plancherel's theorem then yields $xf\in L^2(\mathbb{R})$, and the Cauchy--Schwarz inequality combined with $\langle x\rangle^{-1}\in L^2(\mathbb{R})$ proves \eqref{eq:zero-row-L1}.
\end{proof}

\begin{condition}[Anchoring normalization]\label{hyp:M4ab-anchor}
There exist an element $E_0\in\mathrm{Dom}(T_D)\cap\mathrm{Dom}(A_{\max,\beta_{u_0}})$ and an enhanced maximal-graph representative such that its inverse spectral vector $e_0=\mathscr{U}_{u_0}^{-1}JE_0$ satisfies $I_+e_0=1$, and
\begin{equation}\label{eq:M4ab-anchor}
J_+E_0=(2\pi)^{-1/2},\qquad\ell_DE_0=(2\pi i)^{-1},\qquad\Gamma_{\mu,-}E_0=\Gamma_{\mu,+}E_0\quad(\mu\in S_+\cap\mathbb{R}_+).
\end{equation}
The incoming physical normalization fixes all phases.
\end{condition}

\subsection{The scalar graph}
\begin{thm}[Rough enhanced scalar graph]\label{thm:M2-scalar-graph}
Let $\beta\in L^2(\mathbb{R}_+)$, and let $S\subset\mathbb{R}_+$ be finite. On the components of $\mathbb{R}_+\setminus S$, consider triples $(G,K,Z)\in L^2(\mathbb{R}_+)^3$ satisfying
\begin{equation}\label{eq:M2-factor-graph}
Z=\mathcal{C}_H^+(\overline{\beta} G),\qquad i G^\prime+\beta Z=K
\end{equation}
in the distributional factor sense of \eqref{eq:CH-distribution}. Then the following assertions hold.
\begin{enumerate}[(i)]
  \item The resulting enhanced maximal graph $\mathfrak{G}_{\max}^S(\beta)$, equipped with the norm $\|G\|_{L^2(\mathbb{R}_+)}+\|K\|_{L^2(\mathbb{R}_+)}+\|Z\|_{L^2(\mathbb{R}_+)},$ is closed. Every $G$ has ordinary one-sided traces at the points of $S$ and at zero, and the aggregate trace is bounded and onto; it has a bounded finite-dimensional right inverse.

  \item For any two enhanced triples $(G,K,Z)$ and $(F,L,W)$, the polarized Green identity is
  \begin{equation}\label{eq:M2-Green}
  \langle K,F\rangle-\langle G,L\rangle+iJ_+G\,\overline{J_+F}=i\langle \Gamma_-G, \Gamma_-F\rangle-i\langle \Gamma_+G, \Gamma_+F\rangle.
  \end{equation}

  \item Sewing every interior pair with coefficient one defines a closed maximal dissipative operator $A_\beta$. Its resolvent is obtained, for $z\in\mathbb{C}_+$ with sufficiently large imaginary
  part, from
   $\mathcal{K}_\beta(z)=\mathcal{C}_H^+M_{\overline{\beta}}R_0(z)M_\beta$,    $Z_h=[\mathrm{Id}+\mathcal{K}_\beta(z)]^{-1}\mathcal{C}_H^+M_{\overline{\beta}}R_0(z)h$,   \begin{equation}\label{eq:M2-Woodbury}
  (A_\beta-z)^{-1}h=R_0(z)(h-\beta Z_h).
  \end{equation}
  The endpoint row is defined by $\ell_\beta=(i\sqrt{2\pi})^{-1}J_+$, and it is continuous with respect to the graph norm.
\end{enumerate}
\end{thm}
\begin{proof}
First note the global estimate
\begin{equation}\label{eq:M2-global-Linfty}
\|G\|_{L^\infty(\mathbb{R}_+)}\le C_S\bigl(\|G\|_{L^2(\mathbb{R}_+)}+\|K\|_{L^2(\mathbb{R}_+)}+\|\beta\|_{L^2(\mathbb{R}_+)}\|Z\|_{L^2(\mathbb{R})}\bigr).
\end{equation}
Indeed, on every bounded component this follows from the one-dimensional trace inequality and $G^\prime=-iK+i\beta Z\in L^1_{\mathrm{loc}}(\mathbb{R}_+)$. On the last component, for each $x$ choose $y\in(x,x+1)$ with $|G(y)|\le\|G\|_{L^2(x,x+1)}$, and integrate the same identity between $x$ and $y$. This proves \eqref{eq:M2-global-Linfty}; it
also bounds all finitely many one-sided traces by the enhanced norm.

Now suppose $(G_n,K_n,Z_n)$ converges in the three displayed $L^2$ entries. Since
\[
\|\overline{\beta}(G_n-G)\|_{L^1(\mathbb{R}_+)}\le\|\beta\|_{L^2(\mathbb{R}_+)}\|G_n-G\|_{L^2(\mathbb{R}_+)},
\]
and
\[
\|\beta(Z_n-Z)\|_{L^1(\mathbb{R}_+)}\le\|\beta\|_{L^2(\mathbb{R}_+)}\|Z_n-Z\|_{L^2(\mathbb{R}_+)},
\]
both identities in \eqref{eq:M2-factor-graph} pass to distributions. Applying \eqref{eq:M2-global-Linfty} to the differences gives $G_n\to G$ in $L^\infty(\mathbb{R}_+)$, hence
$\overline{\beta}G_n\to\overline{\beta}G$ in $L^2(\mathbb{R}_+)$. Thus the factor identity passes as a strong $L^2$ identity, proving closedness of the enhanced graph.

For the construction and maximality, take $z=iY$. The integral kernel of $M_{\overline{\beta}}R_0(i Y)M_\beta$ gives
\[
\|M_{\overline{\beta}}R_0(iY)M_\beta\|_{\mathfrak{S}_2}^2=\int_{0<\lambda<s}e^{-2Y(s-\lambda)}|\beta(\lambda)|^2|\beta(s)|^2dsd\lambda\underset{Y\to\infty}{\longrightarrow}0.
\]
Thus $\mathrm{Id}+\mathcal{K}_\beta(iY)$ is invertible for large $Y$. Moreover, $R_0(iY):L^2(\mathbb{R}_+)\to L^\infty(\mathbb{R}_+)$ and $R_0(iY):L^1(\mathbb{R}_+)\to L^2(\mathbb{R}_+)$ are bounded. Consequently every entry in \eqref{eq:M2-Woodbury} lies in the asserted space, and direct substitution proves \eqref{eq:M2-factor-graph} with $K=zG+h$. The factor formulation makes the interaction term symmetric at $L^2$ regularity. If $(F,L,W)$ is a second enhanced triple, then \eqref{eq:M2-global-Linfty} gives $\overline{\beta} G,\overline{\beta} F\in L^2(\mathbb{R}_+)$, and self-adjointness of $\mathcal{C}_H^+$ gives
\[
\langle \beta Z,F\rangle=\langle \mathcal{C}_H^+(\overline{\beta}G), \overline{\beta}F\rangle=\langle\overline{\beta}G, \mathcal{C}_H^+(\overline{\beta}F)\rangle=\langle G,\beta W\rangle.
\]
Integrating the derivative term on every component proves \eqref{eq:M2-Green}. On the unbounded component, $G\overline{F}\in W^{1,1}(\mathbb{R}_+)\cap L^1(\mathbb{R}_+)$, so its boundary value at infinity and hence its terminal contribution vanish. On the coefficient-one sewn domain the identity gives $\langle A_\beta G, F\rangle-\langle G, A_\beta F\rangle=-iJ_+G\overline{J_+F}$. Hence $A_\beta$ is dissipative; surjectivity at $iY$ from \eqref{eq:M2-Woodbury} makes it maximal once density is checked. Put $B=(A_\beta-iY)^{-1}$. The Green identity gives $-\operatorname{Im}\langle h, Bh\rangle=\tfrac{1}{2}|J_+Bh|^2+Y\|Bh\|_{L^2(\mathbb{R})}^2$. If $B^\ast k=0$, then $\langle k, Bk\rangle=0$, so this identity with $h=k$ gives $Bk=0$; injectivity gives $k=0$. Hence $\overline{\mathrm{Ran}(B)}=L^2(\mathbb{R}_+)$, and $A_\beta$ is densely defined and maximal dissipative.

Trace surjectivity follows from an explicit lift. Let $\tau$ be the aggregate map consisting of the zero trace and the two one-sided traces at every point of $S$. Choose a
finite-dimensional piecewise-$H^1$, compactly supported reference lift $\mathcal{R}:\mathbb{C}^{2|S|+1}\longrightarrow L^2(\mathbb{R}_+),\qquad \tau\mathcal{R}=\mathrm{Id}$. For $r=\mathcal{R}c$ and large $Y$, set
\begin{align*}
Z_{Y,r}&=\bigl[\mathrm{Id}+\mathcal{K}_\beta(i Y)\bigr]^{-1}\mathcal{C}_H^+(\overline{\beta} r),\\
G_{Y,r}&=r-R_0(iY)(\beta Z_{Y,r}),\qquad
K_{Y,r}=ir^\prime-iYR_0(iY)(\beta Z_{Y,r}),
\end{align*}
with $r^\prime$ taken componentwise. Since $(i\partial-iY)R_0(iY)=\mathrm{Id}$, direct substitution gives $\mathcal{C}_H^+(\overline{\beta} G_{Y,r})=Z_{Y,r},\qquad iG_{Y,r}^\prime+\beta Z_{Y,r}=K_{Y,r}$. The correction is continuous at every interior trace point. Moreover, $\mathcal{K}_\beta(iY)\to0$, hence $\beta Z_{Y,r}$ converges in $L^1(\mathbb{R}_+)$, uniformly for $c$ in the unit
sphere of this finite-dimensional space. For each trace point $s$,
\[
[R_0(iY)h](s)=i\int_s^\infty e^{-Y(\lambda-s)}h(\lambda)d\lambda\underset{Y\to\infty}{\longrightarrow}0,
\]
uniformly for that compact family of $L^1$ inputs. Consequently $\tau G_{Y,\mathcal{R}c}\to c$ in operator norm. For large $Y$ this trace matrix is invertible, and composing with its inverse gives the asserted bounded finite-dimensional right inverse.

The endpoint row is graph continuous by \eqref{eq:M2-global-Linfty}.
\end{proof}

\subsection{Paired Plancherel--Feshbach crossing}
\begin{lemma}[Weighted trace-zero criterion]\label{lem:weighted-trace-zero}
Let $S$ have a finite ordinary or essential one-sided limit at zero. If $S(s)/s\in L^1(0,\delta)$, then that limit is zero.
\end{lemma}
\begin{proof}
If the limit were $L\ne0$, then $|S(s)|\ge|L|/2$ for all sufficiently small $s$, contradicting
\[
\int_0^\delta|S(s)|s^{-1}ds<\infty.
\]

\end{proof}

\begin{thm}[Paired Feshbach crossing]\label{thm:paired-Feshbach-crossing}
Let $\mu\in S_+\cap\mathbb{R}_+$ be a simple eigenvalue, let $\phi_\mu$ be a corresponding eigenfunction, and put
\begin{equation}\label{eq:eigencharge-source}
g_\mu=\Pi_+(\overline{u_0}\phi_\mu),\qquad F_\mu=u_0g_\mu=(D-\mu)\phi_\mu.
\end{equation}
Then $F_\mu\in L^1_+(\mathbb{R})\cap L^2_+(\mathbb{R})$, where the subscript denotes positive Fourier support. Fix the incoming orientation used in the definition of $\Phi_{u_0}$, on both sides of $\mu$, and set $R_0^{\mathrm{in}}(\lambda):=R_0(\lambda-0i)$. Thus the sign agrees with the fixed column $m_{u_0}^-(\lambda)=m_{\mathrm{e}}(\lambda-0i)$ defining $\Phi_{u_0}$. Write
\[
\mathcal{A}_\lambda=\mathrm{Id}-\Pi_+M_{\overline{u_0}}R_0^{\mathrm{in}}(\lambda)M_{u_0}\Pi_+,\qquad b_\lambda=\Pi_+(\overline{u_0}e^{i\lambda x}).
\]
Choose $h_\mu$ spanning $\ker(\mathcal{A}_\mu^\ast)$. Let $P_r,P_l$ be the orthogonal projections onto $\operatorname{span}\{g_\mu\}$ and $\operatorname{span}\{h_\mu\}$, and put $Q_r=\mathrm{Id}-P_r$, $Q_l=\mathrm{Id}-P_l$. The reduced operator $B_\lambda=Q_l\mathcal{A}_\lambda Q_r:Q_rL^2_+(\mathbb{R})\longrightarrow Q_l L^2_+(\mathbb{R})$ is uniformly invertible on a small two-sided neighborhood $I_\mu=(\mu-\delta,\mu+\delta)$. Define
\[
z_\lambda=B_\lambda^{-1}Q_l b_\lambda,\qquad
r_\lambda=g_\mu-B_\lambda^{-1}Q_l\mathcal{A}_\lambda g_\mu.
\]
For $\lambda\ne\mu$, the Jost charge and column have decompositions
\begin{equation}\label{eq:Feshbach-column-decomposition}
a_\lambda:=\Pi_+(\overline{u_0}m_\lambda)=\mathcal{A}_\lambda^{-1}b_\lambda=z_\lambda+c_\lambda r_\lambda,\qquad m_{\lambda,\mathrm{reg}}=e^{i\lambda x}+R_0^{\mathrm{in}}(\lambda)u_0z_\lambda,
\end{equation}
and
\begin{equation}\label{eq:Feshbach-column-decomposition-2}
m_\lambda=m_{\lambda,\mathrm{reg}}+c_\lambda p_\lambda,\qquad p_\lambda=R_0^{\mathrm{in}}(\lambda)u_0r_\lambda.
\end{equation}
The full pole coefficient obeys
\begin{equation}\label{eq:Feshbach-c-L2}
c_{\lambda}\in L^2(I_\mu),\qquad\|c_{\lambda}\|_{L^2(I_{\mu})}\le\frac{\sqrt{2\pi}}{\|g_\mu\|_{L^2(\mathbb{R})}^2}\|F_\mu\|_{L^2(\mathbb{R})}.
\end{equation}
For every $f\in\mathrm{Dom}(\mathsf{X})\cap P_{\mathrm{ac}}(L_{u_0})L^2_+(\mathbb{R})\cap L^1(\mathbb{R})$, the following holds for a.e. $\lambda\in I_\mu\setminus\{\mu\}$:
\begin{equation}\label{eq:paired-divisibility}
\langle f,p_\lambda-\phi_\mu\rangle_{\mathrm{dual}}=(\lambda-\mu)r_{2,f}(\lambda),\qquad r_{2,f}\in L^2(I_\mu),
\end{equation}
where the subscript means the $L^1$--$L^\infty$ pairing for $p_\lambda$ and the $L^2$ pairing for $\phi_\mu$. Consequently, with the convention that pairings are linear in the
first entry, the singular paired term $S_f(\lambda)=\overline{c_\lambda}\langle f,p_\lambda\rangle_{L^1,L^\infty}$ satisfies
\begin{equation}\label{eq:paired-L1}
\frac{S_f(\lambda)}{\lambda-\mu}=\overline{c_\lambda}r_{2,f}(\lambda)\in L^1(I_{\mu}).
\end{equation}

If, in addition, $G=\Phi_{u_0}f$ belongs to the enhanced maximal graph, then its canonical ordinary traces are the raw fixed-incoming traces and
\begin{equation}\label{eq:paired-crossing}
\Gamma_{\mu,-}G=\Gamma_{\mu,+}G.
\end{equation}
\end{thm}
\begin{proof}
The eigenvalue equation gives \eqref{eq:eigencharge-source}; the two positive-frequency $L^2$ factors have product in $L^1_+(\mathbb{R})$, while $F_\mu=(D-\mu)\phi_\mu\in L^2_+(\mathbb{R})$. By \cite[Lemma~4.2 and Corollary~4.3]{Frank-2026-arXiv} and simplicity of the eigenspace of $L_{u_0}$, $\ker(\mathcal{A}_\mu^{\mathrm{in}})=\operatorname{span}\{g_\mu\}$. Since $(\mathcal{A}_\mu^{\mathrm{in}})^\ast=\mathcal{A}_\mu^{\mathrm{out}}$, the adjoint kernel is
one-dimensional as well. Thus $B_\mu:Q_rL^2_+(\mathbb{R})\to Q_lL^2_+(\mathbb{R})$ is bijective. Operator-norm continuity and the bounded inverse theorem give uniform invertibility of $B_\lambda$, and elementary block elimination gives \eqref{eq:Feshbach-column-decomposition}--\eqref{eq:Feshbach-column-decomposition-2}.

Since $z_\lambda\perp g_\mu$ and $r_\lambda-g_\mu\perp g_\mu$,
\[
c_\lambda\|g_\mu\|_{L^2(\mathbb{R})}^2=\langle a_\lambda, g_\mu\rangle=\langle m_\lambda, F_\mu\rangle_{L^\infty,L^1}=\overline{\langle F_\mu, m_\lambda\rangle_{L^1,L^\infty}}=\sqrt{2\pi}\overline{(\Phi_{u_0}F_\mu)(\lambda)}
\]
for a.e. $\lambda\in I_\mu$. The Jost-column expressions use $L^1$--$L^\infty$ duality. Distorted Plancherel \cite[Theorem~5.3]{Frank-2026-arXiv} applied to $F_\mu$ proves \eqref{eq:Feshbach-c-L2}.

Furthermore, $\mathcal{A}_\lambda g_\mu=-(\lambda-\mu)\Pi_+\left(\overline{u_0}R_0^{\mathrm{in}}(\lambda)\phi_\mu\right)$. This identity is understood for the a.e. $L^2_\lambda(L^2_x)$ boundary representative. The free resolvent square-function estimate yields
\[
\int_{I_\mu}\|u_0R_0^{\mathrm{in}}(\lambda)\phi_\mu\|_{L^2(\mathbb{R})}^2d\lambda\le2\pi\|u_0\|_{L^2(\mathbb{R})}^2\|\phi_\mu\|_{L^2(\mathbb{R})}^2,
\]
and hence $v_\lambda:=\frac{r_\lambda-g_\mu}{\lambda-\mu}\in L^2(I_\mu;L^2_x)$. The eigenvalue equation has the boundary interpretation $R_0^{\mathrm{in}}(\mu)F_\mu=\phi_\mu\quad\text{in }\mathcal{S}^\prime$, so the anchored singular column has $p_\mu=\phi_\mu$. After pairing, the resolvent identity gives, for a.e. $\lambda\in I_\mu\setminus\{\mu\}$,
\[
\frac{\langle f,p_\lambda-\phi_\mu\rangle_{\mathrm{dual}}}{\lambda-\mu}=\mathcal{C}^{\mathrm{in}}\bigl(\widehat{f}\overline{\widehat{\phi}_\mu}\bigr)(\lambda)\quad+\langle f, R_0^{\mathrm{in}}(\lambda)u_0v_\lambda\rangle_{L^1,L^\infty}=:r_{2,f}(\lambda).
\]
Here $\mathcal{C}^{\mathrm{in}}$ is the corresponding Cauchy boundary operator. Its first term belongs to $L^2(I_\mu)$ because $\widehat{f}\in L^\infty(\mathbb{R})$ and $\widehat{\phi}_\mu\in L^2(\mathbb{R})$; the second is controlled by $f\in L^1(\mathbb{R})$, $u_0v_\lambda\in L^1(\mathbb{R})$, and the uniform $L^1\to L^\infty$ bound for the free boundary resolvent. This proves \eqref{eq:paired-divisibility}. Since $f\perp\phi_\mu$, $\langle f,p_\lambda\rangle_{L^1,L^\infty}=\langle f,p_\lambda-\phi_\mu\rangle_{\mathrm{dual}}$, and Cauchy--Schwarz proves \eqref{eq:paired-L1}.

On each punctured bank, the raw Jost pairing is continuous and equals the distorted transform almost everywhere. The canonical enhanced representative is also continuous there, so the two representatives agree pointwise, with $G(\lambda)=\frac{1}{\sqrt{2\pi}}\langle f, m_\lambda\rangle_{L^1,L^\infty}$. The maps $\lambda\mapsto b_\lambda$ and $\lambda\mapsto\mathcal{A}_\lambda$ are continuous in $L^2(\mathbb{R})$ and
operator norm, respectively. Hence $z_\lambda$ is $L^2$-continuous and $u_0z_\lambda$ is $L^1$-continuous. \cite[Lemma~4.1(b)]{Frank-2026-arXiv}, together with the uniform
$L^1\to L^\infty$ bound for $R_0^{\mathrm{in}}(\lambda)$, shows that the scalar regular pairing $\lambda\longmapsto\langle f, e^{i\lambda x}+R_0^{\mathrm{in}}(\lambda)u_0z_\lambda\rangle$ has the same finite limit from the two sides. The singular term has a finite canonical one-sided limit and satisfies \eqref{eq:paired-L1}; Lemma~\ref{lem:weighted-trace-zero} makes both limits zero. This proves \eqref{eq:paired-crossing}.
\end{proof}
\subsection{Identification of the Dirichlet model and point-return control}\label{sec:return}
\begin{cor}[Identification under the structural conditions]\label{cor:phase-I-closure}
Assume Conditions~\ref{hyp:schur-graph}, \ref{hyp:M3-actual-form}, and \ref{hyp:M4ab-anchor}. Then $T_D=A_{\beta_{u_0}},\qquad U=\mathrm{Id}$. More explicitly, for every $t\ge0$, $\eta\in\mathbb{R}$, and $g\in C_c^\infty(\mathbb{R}_+)$,
\begin{equation}\label{eq:phase-I-full-signed-A1}
\ell_{\mathrm{free}}R^0_{t,\eta}g-\ell^A_{t,\eta}R_{t,\eta}Jg=\frac1{i\sqrt{2\pi}}D_t^{\beta_{u_0},g}(\eta)-\rho_t^{\mathrm{p}}(\eta;g).
\end{equation}
For fixed $(t,\eta)$, both sides are bounded rows in $g$, so this identity extends uniquely from smooth inputs to the full physical $L^2$ row.
\end{cor}
\begin{proof}
By Proposition~\ref{prop:M3-enhanced-closure}, every $G\in\mathrm{Dom}(T_D)$ lies in the scalar maximal graph with the same action. By \eqref{eq:static-actual-domain} and
Proposition~\ref{prop:M4a-localization}, put $f=\mathscr{U}_{u_0}^{-1}JG$; then $f\in\mathrm{Dom}(\mathsf{X})\cap P_{\mathrm{ac}}(L_{u_0})L^2_+(\mathbb{R})$. Set
\[
c=I_+f,\qquad G_0=G-cE_0,\qquad f_0=f-ce_0 .
\]
Then $I_+f_0=0$, and \eqref{eq:zero-row-L1} gives $f_0\in L^1(\mathbb{R})$. By linearity of \eqref{eq:physical-domain-bound}, $G_0=\Phi_{u_0}f_0$ almost everywhere. Both $G_0$ and $E_0$ belong to the enhanced maximal graph. Applying Theorem~\ref{thm:paired-Feshbach-crossing} to $(f_0,G_0)$ at every positive embedded label and then using \eqref{eq:M4ab-anchor} gives $\Gamma_{\mu,-}G=\Gamma_{\mu,+}G\qquad(\mu\in S_+\cap\mathbb{R}_+)$. Therefore $T_D\subset A_{\beta_{u_0}}$. Both operators are maximal dissipative, so the inclusion is equality; equivalently the aggregate interface relation is $U=\mathrm{Id}$.

At $t=0$, cancel the common $-2iy_0$ shift in the polarized actual Green identity \eqref{eq:full-green} and compare it with \eqref{eq:M2-Green}. Since the domains and actions now agree, for all $G,F\in\mathrm{Dom}(T_D)$, $\ell_DG\overline{\ell_DF}=\ell_{\beta_{u_0}}G\overline{\ell_{\beta_{u_0}}F}$. Thus $\ell_D=\omega\ell_{\beta_{u_0}}$ for one $|\omega|=1$. The nonzero anchor values in \eqref{eq:M4ab-anchor} give $\ell_DE_0=\ell_{\beta_{u_0}}E_0=(2\pi i)^{-1}$, hence
$\omega=1$. Thus the endpoint rows agree.

Chirp conjugation gives $R^D_{t,\eta}=C_t^{-1}R_{\beta_{u_0},t}(z_0+2t\eta)C_t$. The common endpoint row then yields $\ell_{\mathrm{free}}R^0_{t,\eta}g-\ell_DR^D_{t,\eta}g=\frac{1}{i\sqrt{2\pi}}D_t^{\beta_{u_0},g}(\eta)$. Finally, $\ell^A_{t,\eta}R_{t,\eta}Jg=\ell_DR^D_{t,\eta}g+\rho_t^{\mathrm{p}}(\eta;g)$ by \eqref{eq:exact-output-split}. Subtraction proves
\eqref{eq:phase-I-full-signed-A1}.
\end{proof}

\begin{condition}[Point-return control]\label{hyp:return}
For every compact interval $\Lambda\Subset\mathbb{R}_+\setminus (\{0\}\cup\sigma_{\mathrm{p}}(L_{u_0}))$, one of the following conditions holds.
\begin{enumerate}[(i)]
 \item There exist $0\leq\rho<\frac{1}{2}$ and $C_\Lambda<\infty$ such that
 \begin{equation}\label{eq:schur-bound}
  \mathrm{ess}\sup_{\eta\in\Lambda}\|K_t(\eta)^{-1}-2t(M_{\mathrm{p}}-\eta)\|_{\mathcal{B}(\mathbb{C}^N)}\le C_\Lambda t^\rho,\qquad t\ge1.
 \end{equation}
 \item For every $g\in C_c^\infty(\mathbb{R}_+)$ and $\varphi\in C_c^\infty(\Lambda)$,
 \begin{equation}\label{eq:weak-return}
  \sqrt{t}\int_\Lambda e^{-it\eta^2}\varphi(\eta)\rho_t^{\mathrm{p}}(\eta;g)d\eta\underset{t\to\infty}{\longrightarrow}0.
 \end{equation}
\end{enumerate}
\end{condition}
Proposition~\ref{prop:schur-return} shows that (i) implies (ii), so either condition gives \eqref{eq:weak-return}.
\section{Continuous radiation: scattering and dispersion}\label{sec:radiation}
\subsection{Scalar cancellation and return estimates}
\begin{lemma}[Rough scalar cancellation]\label{lem:scalar-cancellation}
For $\beta\in L^2(\mathbb{R}_+)$, $g\in C_c^\infty(\mathbb{R}_+)$, and $\Lambda\Subset(0,\infty)$,
\begin{equation}\label{eq:scalar-cancellation}
\sup_{\eta\in\Lambda}\sqrt{t}\,|D_t^{\beta,g}(\eta)|\underset{t\to\infty}{\longrightarrow}0.
\end{equation}
\end{lemma}
\begin{proof}
Put $\zeta=z_0+2t\eta$ and $r_{t,\eta}=C_t^{-1}R_0(\zeta)C_tg$.
The free Volterra formula and van der Corput's lemma \cite[Lemma~A.2]{Chen-2026-arXiv} give
\begin{equation}\label{eq:vdc}
\sup_{\eta\in\Lambda}\|r_{t,\eta}\|_{L^\infty(\mathbb{R}_+)}\le C_{z_0,g,\Lambda}t^{-1/2}.
\end{equation}
For $w_1,w_2\in L^2(\mathbb{R}_+)$ and $r\in L^\infty(\mathbb{R}_+)$, the definition \eqref{eq:rough-toeplitz}, Cauchy--Schwarz, and boundedness of the half-line Hilbert transform on $L^2(\mathbb{R})$ give
\begin{equation}\label{eq:B-lipschitz}
\|[B_+(w_1)-B_+(w_2)]r\|_{L^1(\mathbb{R}_+)}\le C_H\|w_1-w_2\|_{L^2(\mathbb{R}_+)}(\|w_1\|_{L^2(\mathbb{R}_+)}+\|w_2\|_{L^2(\mathbb{R}_+)})\|r\|_{L^\infty(\mathbb{R}_+)}.
\end{equation}
Expand the difference of $M_w\mathcal C_H^+M_{\bar w}$ at $w_1,w_2$ and apply Cauchy--Schwarz to each term.

Define the rough endpoint functional
\[
\mathfrak{J}_{\beta,t,\zeta}(f)=J_+R_{\beta,t}(\zeta)C_tf,\qquad f\in L^1(\mathbb{R}_+)\cap L^2(\mathbb{R}_+).
\]
Write $\beta_t=C_t\beta$ and $\Gamma_{\beta,t}^+(\zeta)=\mathcal{C}_H^+M_{\overline{\beta_t}}R_0(\zeta)M_{\beta_t}$. The Volterra kernel gives
\[
\|M_{\overline{\beta_t}}R_0(x+iy)M_{\beta_t}\|_{\mathrm{HS}}^2=\int_0^\infty\int_\lambda^\infty|\beta(\lambda)|^2e^{-2y(s-\lambda)}|\beta(s)|^2dsd\lambda\underset{y\to\infty}\longrightarrow0
\]
uniformly in $x$ and $t$, by dominated convergence. Hence there exists $Y_\beta$ such that $\|\Gamma_{\beta,t}^+(\zeta)\|_{\mathcal{B}(L^2(\mathbb{R}_+))}\leq\frac{1}{2}$ whenever $\operatorname{Im}\zeta\geq Y_\beta$. In this region the Woodbury identity is
\begin{equation}\label{eq:plus-Woodbury}
R_{\beta,t}(\zeta)=R_0(\zeta)-R_0(\zeta)M_{\beta_t}[\mathrm{Id}+\Gamma_{\beta,t}^+(\zeta)]^{-1}\times\mathcal{C}_H^+M_{\overline{\beta_t}}R_0(\zeta).
\end{equation}

To identify this formula with the enhanced scalar graph, let $G=R_{\beta,t}(\zeta)f$ and let $Z$ be its factor profile. Equality of the coefficient-one traces removes the distributional jumps at every puncture. Applying the global right-Volterra inverse to the two factor equations gives
\[
G=R_0(\zeta)f-R_0(\zeta)M_{\beta_t}Z,\qquad Z=\mathcal{C}_H^+M_{\overline{\beta_t}}G.
\]
 Eliminating $Z$ gives \eqref{eq:plus-Woodbury}. Conversely, the right-hand side of \eqref{eq:plus-Woodbury} satisfies the enhanced graph equations and every coefficient-one interface condition. Maximal-dissipative uniqueness therefore identifies it with $R_{\beta,t}(\zeta)$. The free Volterra bounds $\|R_0(\zeta)\|_{L^1(\mathbb{R}_+)\to L^\infty(\mathbb{R}_+)}\le1$,  $\|R_0(x+iY)f\|_{L^2(\mathbb{R}_+)}\le(2Y)^{-1/2}\|f\|_{L^1(\mathbb{R}_+)}$,  $|J_+R_0(\zeta)f|\le\|f\|_{L^1(\mathbb{R}_+)}$, and Cauchy--Schwarz imply, for $f\in L^1(\mathbb{R}_+)\cap L^2(\mathbb{R}_+)$,
\[
|J_+R_{\beta,t}(\zeta)C_tf|\le C_\beta\|f\|_{L^1(\mathbb{R}_+)},\qquad\operatorname{Im}\zeta\ge Y_\beta.
\]
Set
\[
h=[\mathrm{Id}+\Gamma_{\beta,t}^+(\zeta)]^{-1}\mathcal{C}_H^+M_{\overline{\beta_t}}R_0(\zeta)C_tf.
\]
 The preceding bounds give $\|h\|_{L^2(\mathbb{R}_+)}\le C\|\beta\|_{L^2(\mathbb{R}_+)}\|f\|_{L^1(\mathbb{R}_+)}$. Moreover,
\[
|J_+R_0(\zeta)M_{\beta_t}h|\le\|\beta_t h\|_{L^1(\mathbb{R}_+)}\le\|\beta\|_{L^2(\mathbb{R}_+)}\|h\|_{L^2(\mathbb{R}_+)}.
\]
The same $L^1(\mathbb{R}_+)\to L^2(\mathbb{R}_+)$ estimate gives
\[
\|R_0(x+i Y)M_{\beta_t}h\|_{L^2(\mathbb{R}_+)}\le(2Y)^{-1/2}\|\beta\|_{L^2(\mathbb{R}_+)}\|h\|_{L^2(\mathbb{R}_+)}.
\]
Consequently, \eqref{eq:plus-Woodbury} also yields the high-line bound
\begin{equation}\label{eq:plus-high-line-L2}
\|R_{\beta,t}(x+iY)C_tf\|_{L^2(\mathbb{R}_+)}\le C_{\beta,Y}\|f\|_{L^1(\mathbb{R}_+)},\qquad Y\ge Y_\beta.
\end{equation}

For $y_0\leq\operatorname{Im}\zeta<Y_\beta$, put $\zeta_Y=\operatorname{Re}\zeta+iY_\beta$. The resolvent identity gives $R_{\beta,t}(\zeta)=R_{\beta,t}(\zeta_Y)+(\zeta-\zeta_Y)R_{\beta,t}(\zeta)R_{\beta,t}(\zeta_Y)$. The endpoint bound \eqref{eq:general-scalar-row-bound} gives
\[
|J_+R_{\beta,t}(\zeta)g|\le\left(\frac{2}{y_0}\right)^{1/2}\|g\|_{L^2(\mathbb{R}_+)},\qquad g\in L^2(\mathbb{R}_+).
\]
Combining this estimate with \eqref{eq:plus-high-line-L2} proves
\begin{equation}\label{eq:rough-endpoint-bound}
|\mathfrak{J}_{\beta,t,\zeta}(f)|\le C_{\beta,y_0}\|f\|_{L^1(\mathbb{R}_+)},
\end{equation}
uniformly in $t$, $\eta$, and $\operatorname{Re}\zeta$. Density of $L^1(\mathbb{R}_+)\cap L^2(\mathbb{R}_+)$ in $L^1(\mathbb{R}_+)$ gives the unique bounded extension. This factorization is used in \cite[Lemma~A.5]{Chen-2026-arXiv}. The first resolvent identity gives $D_t^{\beta,g}(\eta)=\mathfrak{J}_{\beta,t,\zeta}\bigl(B_+(\beta)r_{t,\eta}\bigr)$.
Choose $\beta_\varepsilon\in C_c^\infty(\mathbb{R}_+)$ with $\beta_\varepsilon\to\beta$ in $L^2(\mathbb{R}_+)$, and split
\[
D_t^{\beta,g}(\eta)=\mathfrak{J}_{\beta,t,\zeta}\bigl([B_+(\beta)-B_+(\beta_\varepsilon)]r_{t,\eta}\bigr)+\mathfrak{J}_{\beta,t,\zeta}\bigl(B_+(\beta_\varepsilon)r_{t,\eta}\bigr).
\]
Let $\mathcal{K}f=\overline{f}$, define
\[
B_-(w)=\frac{1}{4\pi}M_w(\mathrm{Id}-iH)M_{\overline{w}},\qquad G_{t,\eta}^{z_0}=C_t^{-1}R_0(z_0+2t\eta)C_t,
\]
so that
\begin{equation}\label{eq:focusing-conjugation}
\mathcal{K}G_{t,\eta}^{z_0}=-G_{-t,\eta}^{-\overline{z_0}}\mathcal{K},\qquad\mathcal{K}B_+(w)=B_-(\overline{w})\mathcal{K}.
\end{equation}
Indeed, these follow from
$\mathcal{K}(i\partial_\lambda-\zeta)\mathcal{K}=-(i\partial_\lambda+\overline\zeta)$, $\mathcal{K}C_t\mathcal{K}=C_t^{-1}$, and $\mathcal{K}H\mathcal{K}=H$.
Apply \cite[Equation~(4.8) and Proposition~A.1]{Chen-2026-arXiv} at time $-t$ to the conjugated smooth data in \eqref{eq:focusing-conjugation}. For each fixed $\varepsilon$, this gives $\sup_{\eta\in\Lambda}\|B_+(\beta_\varepsilon)r_{t,\eta}\|_{L^1(\mathbb{R}_+)}\le C_{\varepsilon,z_0,g,\Lambda}t^{-1}$. The endpoint estimate \eqref{eq:rough-endpoint-bound} therefore gives
\[
\sup_{\eta\in\Lambda}\sqrt t\left|\mathfrak {J}_{\beta,t,\zeta}\bigl(B_+(\beta_\varepsilon)r_{t,\eta}\bigr)\right|\le C_{\varepsilon,\beta,z_0,g,\Lambda}t^{-1/2}\underset{t\to\infty}{\longrightarrow}0.
\]
Equations \eqref{eq:vdc}, \eqref{eq:B-lipschitz}, and \eqref{eq:rough-endpoint-bound} give
\[
\limsup_{t\to\infty}\sup_{\eta\in\Lambda}\sqrt{t}|D_t^{\beta,g}(\eta)|\le C_{\beta,y_0,z_0,g,\Lambda}\|\beta-\beta_\varepsilon\|_{L^2(\mathbb{R}_+)}\bigl(\|\beta\|_{L^2(\mathbb{R}_+)}+\|\beta_\varepsilon\|_{L^2(\mathbb{R}_+)}\bigr).
\]
Sending $\varepsilon\to0$ proves \eqref{eq:scalar-cancellation}.
\end{proof}
\begin{prop}[Subcritical Schur return]\label{prop:schur-return}
Under Condition~\ref{hyp:schur-graph}, if \eqref{eq:schur-bound} holds, then \eqref{eq:weak-return} holds. More precisely, for $g\in C_c^\infty(\mathbb{R}_+)$,
\begin{equation}\label{eq:return-rate}
\|\sqrt{t}\,\rho_t^{\mathrm{p}}(\cdot;g)\|_{L^2(\Lambda)}=O_{\Lambda,g}(t^{\rho-1/2}).
\end{equation}
\end{prop}
\begin{proof}
Set $\delta_\Lambda=\mathrm{dist}(\Lambda,\{\mu_1,\ldots,\mu_N\})>0$. For all sufficiently large $t$, the finite-dimensional Neumann estimate gives
\begin{equation}\label{eq:K-bound}
\operatorname*{ess\,sup}_{\eta\in\Lambda}\|K_t(\eta)\|_{\mathcal{B}(\mathbb{C}^N)}\le (2t\delta_\Lambda-C_\Lambda t^\rho)^{-1}.
\end{equation}
Let $\mathfrak{E}_t(\eta)=K_t(\eta)^{-1}-2t(M_{\mathrm{p}}-\eta)$. The Green identity \eqref{eq:full-green}, the lift identity \eqref{eq:detailed-point-lift}, and $P\Pi_{t,\eta}=\mathrm{Id}_{\mathcal{E}}$ give
\[
\pi|\mathfrak{l}_t(\eta)a|^2+y_0\|\Pi_{t,\eta}a\|_{\mathcal{H}}^2=-\operatorname{Im}\langle \mathfrak{E}_t(\eta)a, a\rangle_{\mathcal{E}}\le C_\Lambda t^\rho\|a\|_{\mathcal{E}}^2.
\]
Consequently,
\[
\|\mathfrak{l}_t(\eta)\|_{\mathcal{B}(\mathcal{E},\mathbb{C})}\le\left(\frac{C_\Lambda}{\pi}\right)^{1/2}t^{\rho/2},\qquad\|\Pi_{t,\eta}\|_{\mathcal{B}(\mathcal{E},\mathcal{H})}\le\left(\frac{C_\Lambda}{y_0}\right)^{1/2}t^{\rho/2}.
\]
For $\|a\|_{\mathcal{E}}=1$, \eqref{eq:polarized-green} and \eqref{eq:dirichlet-green-bound} yield
\[
\begin{aligned}
|\langle \mathfrak{b}_t(\eta;h), a\rangle_{\mathcal{E}}|&\le \|h\|_{\mathcal{H}_c}\|\Pi_{t,\eta}a\|_{\mathcal{H}}+2y_0\|c_h\|_{\mathcal{H}}\|\Pi_{t,\eta}a\|_{\mathcal{H}}
+2\pi|\ell^A_{t,\eta}c_h|\,|\mathfrak{l}_t(\eta)a|\\
&\le C_{\Lambda,z_0}t^{\rho/2}\|h\|_{\mathcal{H}_c}.
\end{aligned}
\]
Taking the supremum over $\|a\|_{\mathcal{E}}=1$ gives
\begin{equation}\label{eq:charge-bounds}
\operatorname*{ess\,sup}_{\eta\in\Lambda}\|\mathfrak{l}_t(\eta)\|_{\mathcal{B}(\mathcal{E},\mathbb{C})}+\operatorname*{ess\,sup}_{\eta\in\Lambda}\|\mathfrak{b}_t(\eta;\cdot)\|_{\mathcal{B}(\mathcal{H}_c,\mathcal{E})}\le C^\prime_{\Lambda,z_0}t^{\rho/2}.
\end{equation}
Equations \eqref{eq:point-return}, \eqref{eq:K-bound}, and \eqref{eq:charge-bounds} imply $\operatorname*{ess\,sup}_{\eta\in\Lambda}|\rho_t^{\mathrm p}(\eta;g)|\le C_{\Lambda,g}t^{\rho-1}$. Integration over $\Lambda$ proves \eqref{eq:return-rate}, and Cauchy--Schwarz proves \eqref{eq:weak-return}.
\end{proof}
\subsection{Ray convergence and strong scattering}
\begin{lemma}[Poisson--Fresnel ray criterion]\label{lem:ray-criterion}
Let $f_t$ be a bounded family in $L^2_+(\mathbb{R})$, and set
\[
d_t(\eta)=\sqrt{t}e^{-it\eta^2}[e^{it\partial_x^2}f_t](x_0+2t\eta+iy_0),\qquad\eta>0.
\]
Assume that $d_t$ is uniformly bounded in $L^2(\mathbb{R}_+)$ and that
\begin{equation}\label{eq:local-ray-weak}
\langle d_t, \varphi\rangle\underset{t\to\infty}{\longrightarrow}0
\end{equation}
for every $\varphi\in C_c^\infty(\mathbb{R}_+\setminus(\{0\}\cup\sigma_{\mathrm{p}}(L_{u_0})))$. Then
\begin{equation}\label{eq:ray-weak}
f_t\rightharpoonup0\quad\text{in }L^2(\mathbb{R}).
\end{equation}
\end{lemma}
\begin{proof}
Extend functions on $\mathbb{R}_+$ by zero to $\mathbb{R}$ and put $a_t(\xi)=e^{ix_0\xi-y_0\xi}\mathcal{F}_0f_t(\xi)$.
Fourier inversion for the Hardy extension and completion of the square give
\[
d_t(\eta)=\frac{1}{\sqrt{2\pi}}(\mathcal{F}_ta_t)(\eta),\qquad(\mathcal{F}_ta)(\eta)=\sqrt{t}\int_\mathbb{R}e^{-it(\eta-\xi)^2}a(\xi)d\xi.
\]
The Fourier multiplier of $\mathcal{F}_t$ is $\sqrt{\pi}e^{-i\pi/4}e^{i s^2/(4t)}$.
It follows from Plancherel and dominated convergence that
\begin{equation}\label{eq:fresnel-adjoint-limit}
\mathcal{F}_t^\ast\varphi\longrightarrow e^{i\pi/4}\sqrt{\pi}\varphi\quad\text{in }L^2(\mathbb{R}).
\end{equation}
For every test in \eqref{eq:local-ray-weak}, $\langle d_t, \varphi\rangle=\frac{1}{\sqrt{2\pi}}\langle a_t, \mathcal{F}_t^\ast\varphi\rangle$. The family $a_t$ is bounded in $L^2(\mathbb{R})$. Equations \eqref{eq:local-ray-weak} and \eqref{eq:fresnel-adjoint-limit} therefore imply $\langle a_t, \varphi\rangle\to0$.

The multiplier $e^{ix_0\xi-y_0\xi}$ is smooth and nonzero on each compact subset of $\mathbb{R}_+$. Given a smooth compactly supported $\psi$ away from the exceptional set, choose $\varphi(\xi)=e^{ix_0\xi+y_0\xi}\psi(\xi)$.
Then $\langle a_t, \varphi\rangle=\langle \mathcal{F}_0f_t, \psi\rangle$, so this pairing tends to zero. Such tests are dense in $L^2(\mathbb{R}_+)$, since the
exceptional set is finite.  Boundedness of $f_t$ and Plancherel extend the limit to every $L^2$ test, proving \eqref{eq:ray-weak}.
\end{proof}

\begin{thm}[Strong scattering of the transported continuous channel]\label{thm:P3}
Under Hypothesis~\ref{hyp:schur-graph}, \ref{hyp:M3-actual-form}, \ref{hyp:M4ab-anchor}, and \ref{hyp:return},
\begin{equation}\label{eq:P3}
\|r_{\mathrm{ac}}(t)-e^{it\partial_x^2}u_+\|_{L^2(\mathbb{R})}\underset{t\to\infty}{\longrightarrow}0,
\end{equation}
where $u_+$ is given by \eqref{eq:radiation-state}.
\end{thm}
\begin{proof}
The channel identity \eqref{eq:boundary-channel} and Proposition~\ref{prop:automatic-ray-closure} identify the cyclic vector $g_{u_0}$ with the physical continuous output $r_{\mathrm{ac}}(t)$. Hence the signed full--Dirichlet--point row is the free radiation ray in \eqref{eq:total-ray-identity}.

Fix a compact interval $\Lambda\Subset\mathbb{R}_+\setminus(\{0\}\cup\sigma_{\mathrm{p}}(L_{u_0}))$ and $\varphi\in C_c^\infty(\Lambda)$. For $g\in C_c^\infty(\mathbb{R}_+)$, Corollary~\ref{cor:phase-I-closure} gives
\[
(\mathscr{D}_tg)(\eta)=\frac{\sqrt{t}e^{it\eta^2}}{i\sqrt{2\pi}}D_t^{\beta_{u_0},g}(\eta)-\sqrt{t}e^{-it\eta^2}\rho_t^{\mathrm{p}}(\eta;g).
\]
Lemma~\ref{lem:scalar-cancellation} implies that the first term on the right-hand side converges strongly to zero in $L^2(\Lambda)$. If the operator-norm bound in Condition~\ref{hyp:return} holds, then Proposition~\ref{prop:schur-return} implies that the second term converges strongly to zero in $L^2(\Lambda)$ as well. If instead the weak return estimate in Condition~\ref{hyp:return} holds, then the pairing of the second term with $\varphi$ converges to zero by \eqref{eq:weak-return} with test function $\overline\varphi$. Therefore
\begin{equation}\label{eq:smooth-total-weak}
\langle \mathscr{D}_tg, \varphi\rangle_{L^2(\Lambda)}\underset{t\to\infty}{\longrightarrow}0
\end{equation}
for every smooth $g$.

Fix $n$. Since $\chi_ng_{u_0}$ is supported in a compact regular subset of the spectral axis, choose $g_m\in C_c^\infty(\mathbb{R}_+)$, supported away from the exceptional set, such that
\[
\|g_m-\chi_ng_{u_0}\|_{L^2(\mathbb{R}_+)}\underset{m\to\infty}{\longrightarrow}0.
\]
Uniform boundedness of $\mathscr{D}_t$ gives
\[
\limsup_{t\to\infty}\left|\langle \mathscr{D}_t(\chi_ng_{u_0}), \varphi \rangle\right|\le C\|\chi_ng_{u_0}-g_m\|_{L^2(\mathbb{R}_+)}\|\varphi\|_{L^2(\Lambda)}.
\]
First send $t\to\infty$, using \eqref{eq:smooth-total-weak}, and then send $m\to\infty$. This proves \eqref{eq:smooth-total-weak} with $g=\chi_ng_{u_0}$. The spectral closure estimate now yields
\[
\limsup_{t\to\infty}\left|\langle \mathscr{D}_tg_{u_0}, \varphi\rangle\right|\le C\|(1-\chi_n)g_{u_0}\|_{L^2(\mathbb{R}_+)}\|\varphi\|_{L^2(\Lambda)}\underset{n\to\infty}{\longrightarrow}0.
\]
If a test in $C_c^\infty(\mathbb{R}_+\setminus(\{0\}\cup\sigma_{\mathrm{p}}(L_{u_0})))$ meets several regular components, compactness of its support and
finiteness of the exceptional set give a finite smooth partition $\varphi=\sum_r\varphi_r$, with $\operatorname{supp}\varphi_r\Subset\Lambda_r$ for regular compact intervals $\Lambda_r$. Applying the preceding argument to each $\varphi_r$ and summing proves \eqref{eq:local-ray-weak} for every admissible test.

The family $f_t=u_+-e^{-it\partial_x^2}r_{\mathrm{ac}}(t)$ is bounded by \eqref{eq:continuous-mass}. Equations \eqref{eq:total-ray-identity} and the uniform $L^2$ bound for $\mathscr{D}_t$ show that $f_t$ satisfies Lemma~\ref{lem:ray-criterion}. Hence $f_t\rightharpoonup0$, which is
\begin{equation}\label{eq:interaction-weak}
e^{-it\partial_x^2}r_{\mathrm{ac}}(t)\rightharpoonup u_+.
\end{equation}

The identity $\mathcal{F}_0u_+=g_{u_0}=-i\beta_{u_0}/\sqrt{2\pi}$, Plancherel, and \eqref{eq:continuous-mass} give
\begin{equation}\label{eq:equal-norms}
\|e^{-it\partial_x^2}r_{\mathrm{ac}}(t)\|_{L^2(\mathbb{R})}^2=\frac{1}{2\pi}\|\beta_{u_0}\|_{L^2(\mathbb{R})}^2=\|u_+\|_{L^2(\mathbb{R})}^2.
\end{equation}
Set $x_t=e^{-it\partial_x^2}r_{\mathrm{ac}}(t)$. Equations \eqref{eq:interaction-weak}--\eqref{eq:equal-norms} give
\[
\|x_t-u_+\|_{L^2(\mathbb{R})}^2=\|x_t\|_{L^2(\mathbb{R})}^2+\|u_+\|_{L^2(\mathbb{R})}^2-2\operatorname{Re}\langle x_t, u_+\rangle\underset{t\to\infty}{\longrightarrow}0.
\]
Multiplication by the unitary $e^{it\partial_x^2}$ proves \eqref{eq:P3}.
\end{proof}

\section{Rigidity of the discrete channels}\label{sec:phase-III-point}
We give four sufficient criteria for the following sequential rigidity condition, equivalent to global soliton resolution in Theorem~\ref{thm:main}.
\begin{condition}[Sequential point-channel rigidity]\label{hyp:P4}
For every sequence $t_n\to\infty$, there exist a common subsequence $t_{n_k}$ and parameters
\[
\lambda_{j,k}>0,\qquad y_{j,k}\in\mathbb{R},\qquad\xi_{j,k}\ge0,\qquad \theta_{j,k}\in\mathbb{R}/2\pi\mathbb{Z},\qquad 1\le j\le N,
\]
such that
\begin{equation}\label{eq:P4}
\|b_j(t_{n_k})-\mathcal{M}_{\lambda_{j,k},y_{j,k},\xi_{j,k},\theta_{j,k}}R\|_{L^2(\mathbb{R})}\underset{k\to\infty}{\longrightarrow}0.
\end{equation}
\end{condition}

For a fixed label $j$ and a sequence $t_n\to\infty$, write
\begin{equation}\label{eq:M8-basic-sequence}
b_n=b_j(t_n),\qquad \mu=\mu_j,\qquad\|b_n\|_{L^2(\mathbb{R})}^2=2\pi,\qquad L_{u(t_n)}b_n=\mu b_n.
\end{equation}
The criteria below require additional dynamical estimates for the CM-DNLS point column.

\subsection{Geometric compactness of the point column}
Fix a nonzero test
\[
\psi\in\mathcal{S}(\mathbb{R})\cap L^2_+(\mathbb{R}),\ \mathcal{F}_0\psi\in C_c^\infty(\mathbb{R}_+),\ \mathcal{F}_0\psi\ge0,\ \mathcal{F}_0\psi>0\ \hbox{on some open }I\Subset\mathbb{R}_+.
\]
Its affine orbit with zero carrier is total in $L^2_+(\mathbb{R})$. Indeed, if $f\in L^2_+(\mathbb{R})$ is orthogonal to $\mathcal{M}_{\lambda,y,0,0}\psi$ for every $\lambda>0$ and $y\in\mathbb{R}$, then, for each positive rational $\lambda$, Fourier uniqueness in $y$ gives $(\mathcal{F}_0f)(\zeta)\overline{(\mathcal{F}_0\psi)(\lambda\zeta)}=0\quad\text{for a.e. }\zeta>0$. The exceptional sets can therefore be united countably. Since $I/\zeta$ contains a positive rational for every $\zeta>0$, varying $\lambda$ gives $f=0$. For $f\in L^2(\mathbb{R})$, set
\[
\mathfrak{T}_R(f)=\int_{|x|>R}|f(x)|^2dx+\int_{|\zeta|>R}|(\mathcal{F}_0f)(\zeta)|^2d\zeta.
\]

\begin{lemma}[Compactness of rank-one projections]\label{lem:M8-no-fraction}
Let $m>0$, let $g_n\rightharpoonup g\ne0$ in $L^2(\mathbb{R})$, and assume $\|g_n\|_{L^2(\mathbb{R})}^2=m$. For $f\in L^2(\mathbb{R})$, define $\Pi_nf=\frac{1}{m}\langle f, g_n\rangle g_n,\qquad\Pi f=\frac{1}{m}\langle f,g\rangle g$. Then $\Pi_n\to\Pi$ in the weak operator topology and
\begin{equation}\label{eq:M8-fraction-identity}
\Pi^2=\frac{\|g\|_{L^2(\mathbb{R})}^2}{m}\Pi.
\end{equation}
The following are equivalent:
\begin{enumerate}[(i)]
 \item $\Pi^2=\Pi$;
 \item $\|g\|_{L^2(\mathbb{R})}^2=m$;
 \item $g_n\to g$ strongly in $L^2(\mathbb{R})$;
 \item $\Pi_n\to\Pi$ in trace norm.
\end{enumerate}
\end{lemma}
\begin{proof}
Weak convergence gives the weak-operator limit after testing against two fixed vectors, and direct calculation gives \eqref{eq:M8-fraction-identity}. Since $\Pi\ne0$, idempotence is equivalent to $\|g\|_{L^2(\mathbb{R})}^2=m$, which in turn is equivalent to strong convergence by the Hilbert-space weak-plus-norm criterion. Strong convergence implies trace-norm convergence of the rank-one operators. Conversely, trace-norm convergence and continuity of the trace give $1=\operatorname{tr}\Pi=\|g\|_{L^2(\mathbb{R})}^2/m$.
\end{proof}

\begin{thm}[Compactness from physical and Fourier tails]\label{thm:M8-two-tail}
For the sequence \eqref{eq:M8-basic-sequence}, suppose that there are parameters
\[
p_n=(\lambda_n,y_n,\xi_n,\theta_n),\qquad\lambda_n>0,\quad \xi_n\geq0,\qquad g_n=\mathcal{M}_{p_n}^{-1}b_n,
\]
where the parameters with $\xi\ge0$ form a Hardy-preserving affine semigroup. The inverse is taken in the full $L^2$ affine group, and \eqref{eq:M8-negative-tail} measures the resulting Hardy defect. Assume in addition that
\begin{equation}\label{eq:M8-double-tail}
\lim_{R\to\infty}\sup_n\mathfrak{T}_R(g_n)=0,
\end{equation}
\begin{equation}\label{eq:M8-negative-tail}
\|\mathbf{1}_{(-\infty,0)}\mathcal{F}_0g_n\|_{L^2(\mathbb{R})}\underset{n\to\infty}{\longrightarrow}0.
\end{equation}
Then $\{g_n\}$ is relatively compact in $L^2(\mathbb{R})$. Every convergent subsequence has a strong limit $g$ satisfying
\begin{equation}\label{eq:M8-full-profile}
g\in L^2_+(\mathbb{R}),\qquad \|g\|_{L^2(\mathbb{R})}^2=2\pi.
\end{equation}

Moreover, after composing $p_n$ with one of finitely many bounded zero-carrier relative transforms and redefining $g_n$ accordingly, there is a constant $c>0$ such that
\begin{equation}\label{eq:M8-affine-nonvanishing}
\left|\langle b_n, \mathcal{M}_{p_n}\psi\rangle\right|=\left|\langle g_n,\psi\rangle\right|\ge c.
\end{equation}

The negative-frequency defect is
\begin{equation}\label{eq:M8-negative-original}
\|\mathbf{1}_{(-\infty,0)}\mathcal{F}_0 g_n\|_{L^2(\mathbb{R})}^2=\int_0^{\xi_n}|(\mathcal{F}_0 b_n)(\omega)|^2d\omega.
\end{equation}
In particular, it vanishes identically if the compactness frame can be chosen with $\xi_n=0$.
\end{thm}
\begin{proof}
The Fourier tail gives uniform translation continuity. For every $A>0$, Plancherel gives
\[
\|g_n(\cdot+h)-g_n\|_{L^2(\mathbb{R})}^2\le h^2A^2\|g_n\|_{L^2(\mathbb{R})}^2+4\int_{|\zeta|>A}|(\mathcal{F}_0g_n)(\zeta)|^2d\zeta.
\]
First choose $A$ using \eqref{eq:M8-double-tail}, and then choose $h$. The physical tail in \eqref{eq:M8-double-tail} and the Kolmogorov--Riesz theorem now give relative compactness.

If $g_n\to g$ strongly, then \eqref{eq:M8-negative-tail} gives $g\in L^2_+(\mathbb{R})$, while unitarity of the affine maps gives $\|g\|_{L^2(\mathbb{R})}^2=2\pi$. The associated line projections converge in trace norm by Lemma~\ref{lem:M8-no-fraction}.

To use a single test, observe that the closure $K$ of the positive-frequency parts of the tail of $\{g_n\}$ is compact and separated from zero. The function $f\longmapsto\sup_{\lambda>0,\,y\in\mathbb{R}}\left|\langle f,\mathcal{M}_{\lambda,y,0,0}\psi\rangle\right|$ is Lipschitz and, by totality of the affine orbit of $\psi$, is strictly positive on $K$. Its minimum on $K$ is therefore positive. Compactness permits the maximizing relative parameters to be chosen from a finite set up to a factor of two. Absorbing that finite relative transform into $p_n$ proves \eqref{eq:M8-affine-nonvanishing} while preserving both tail estimates.

Inverting \eqref{eq:modulation-fourier} gives
\[
(\mathcal{F}_0 g_n)(\zeta)=\lambda_n^{-1/2}e^{-i\theta_n+iy_n(\xi_n+\zeta/\lambda_n)}(\mathcal{F}_0b_n)(\xi_n+\zeta/\lambda_n).
\]
Since $b_n\in L^2_+(\mathbb{R})$, the change of variables $\omega=\xi_n+\zeta/\lambda_n$ proves \eqref{eq:M8-negative-original}.
\end{proof}

\begin{prop}[Compactness from spectral projections]\label{prop:M8-projector-certificate}
Assume the phase-fixed nonvanishing $\operatorname{Re}\langle g_n,\psi\rangle\ge c>0$. Extend the rank-one spectral projection onto \(\operatorname{span}\{b_n\}\) to $L^2(\mathbb{R})$ and conjugate it by $\mathcal{M}_{p_n}$.  If every weak-operator cluster point of $\Pi_n=\frac{1}{2\pi}\langle \cdot, g_n\rangle g_n$ is idempotent, then $\{g_n\}$ is relatively compact, and \eqref{eq:M8-double-tail} follows.
\end{prop}
\begin{proof}
Every subsequence has a weakly convergent further subsequence. Nonvanishing makes its limit nonzero, and idempotence plus Lemma~\ref{lem:M8-no-fraction} upgrades it to strong convergence. Hence the original set is relatively compact. Compact subsets of $L^2$ have uniform physical tails, and their Fourier images are compact by Plancherel and have uniform Fourier tails.
\end{proof}

\begin{prop}[Transport of a spectral atom]\label{prop:M8-exact-atom-output}
Put $a_j=E_{L_{u_0}}(\{\mu_j\})u_0,\qquad P_j=\frac{1}{2\pi}\langle \cdot, a_j\rangle a_j$. Then $b_j(t)=W_t^{u_0}a_j$.
If $P_j(t)=(2\pi)^{-1}\langle \cdot, b_j(t)\rangle b_j(t)$, then on the cyclic spaces
\begin{equation}\label{eq:M8-exact-line-transport}
P_j(t)=W_t^{u_0}P_j\bigl(W_t^{u_0}|_{\mathcal{K}_{u_0}}\bigr)^\ast.
\end{equation}
\end{prop}
\begin{proof}
By cyclic intertwining and \eqref{eq:cyclic-vector-flow},
\[
W_t^{u_0}a_j=W_t^{u_0}E_{L_{u_0}}(\{\mu_j\})u_0=E_{L_{u(t)}}(\{\mu_j\})W_t^{u_0}u_0=b_j(t).
\]
The restriction of $W_t^{u_0}$ to $\mathcal{K}_{u_0}$ is unitary, so conjugating the rank-one projection proves \eqref{eq:M8-exact-line-transport}.
\end{proof}

\begin{prop}[Anti-Wick phase-space leakage]\label{prop:M8-moving-leakage}
Let $u$ be a smooth CM-DNLS orbit; the same statement holds for a flow--Lax admissible orbit under the atom-preserving passage specified below. Let $b_j(t)=E_{L_{u(t)}}(\{\mu_j\})u(t)$, and choose a $C^1$ Hardy-admissible affine frame
\[
M(t)=\mathcal{M}_{\lambda(t),y(t),\xi(t),\theta(t)}, \qquad \lambda(t)>0,\quad\xi(t)\ge0.
\]
Put
\begin{equation}\label{eq:M8-actual-generator}
\mathbf{B}_uf=i\partial_x^2f+2u\partial_x\Pi_+(\overline{u}f),\qquad \partial_tu=\mathbf{B}_uu,\qquad \partial_tb_j=\mathbf{B}_ub_j,
\end{equation}
the transporting gauge of \cite[(2.6)]{Gerard-2024-CPAM}, and
\[
g(t)=M(t)^{-1}b_j(t),\qquad \mathcal{K}_j(t)=M(t)^{-1}\mathbf{B}_{u(t)}M(t)-M(t)^{-1}\dot M(t),\qquad \dot g=\mathcal{K}_jg,
\]
where the last equality is understood on the smooth orbit. Fix $\varphi\in\mathcal{S}(\mathbb{R})$, $\|\varphi\|_{L^2(\mathbb{R})}=1$, and normalize the short-time Fourier transform $V_\varphi$ as an isometry from $L^2(\mathbb{R})$ into $L^2(\mathbb{R}^2)$. For cutoffs $0\le\chi_R\le1$ in $C_c^\infty(\mathbb{R}^2)$, with $\chi_R\to1$ pointwise, set $K_R=V_\varphi^\ast M_{\chi_R}V_\varphi,\qquad A_R=\mathrm{Id}-K_R$. Then $K_R$ is a compact positive contraction \cite[Section~1]{Marceca-2023-Studia}, and
\begin{equation}\label{eq:M8-leakage-Ward}
\frac{d}{dt}\langle A_Rg(t), g(t)\rangle=2\operatorname{Re}\langle A_R\mathcal{K}_j(t)g(t), g(t)\rangle.
\end{equation}
For a smooth global orbit, assume
\begin{equation}\label{eq:M8-integrated-leakage}
\varepsilon_R:=\langle A_Rg(T_0), g(T_0)\rangle+\int_{T_0}^\infty\left|2\operatorname{Re}\langle A_R\mathcal{K}_j(t)g(t), g(t)\rangle\right|dt\underset{R\to\infty}{\longrightarrow}0,
\end{equation}
\begin{equation}\label{eq:M8-flux-negative-tail}
\|\mathbf{1}_{(-\infty,0)}\mathcal{F}_0g(t)\|_{L^2(\mathbb{R})}\underset{t\to\infty}{\longrightarrow}0.
\end{equation}
For a rough flow--Lax orbit, retain \eqref{eq:M8-flux-negative-tail}. On each $[T_0,T]$, let $g_{r,T}=M_{r,T}^{-1}b_{j,r,T}$ and $\mathcal{K}_{j,r,T}$ be the
identified smooth point profile and its moving generator, and set
\[
\varepsilon_{R,r}^{(T)}:=\langle A_Rg_{r,T}(T_0), g_{r,T}(T_0)\rangle+\int_{T_0}^T\left|2\operatorname{Re}\langle A_R\mathcal{K}_{j,r,T}(t)g_{r,T}(t), g_{r,T}(t)\rangle\right|dt.
\]
Assume the uniform truncated-flux bound
\begin{equation}\label{eq:M8-rough-leakage-uniform}
\lim_{R\to\infty}\sup_{T>T_0}\limsup_{r\to\infty}\varepsilon_{R,r}^{(T)}=0
\end{equation}
and the atom-preserving passage, for every $T<\infty$,
\begin{equation}\label{eq:M8-atom-preserving-passage}
\sup_{T_0\leq t\leq T}\|g_{r,T}(t)-g(t)\|_{L^2(\mathbb{R})}\underset{r\to\infty}{\longrightarrow}0.
\end{equation}
Under either set of assumptions, every sequence $t_n\to\infty$ satisfies \eqref{eq:M8-double-tail}--\eqref{eq:M8-negative-tail}.
\end{prop}

\begin{proof}
In the smooth case, integration of \eqref{eq:M8-leakage-Ward} gives $\sup_{t\geq T_0}\langle A_Rg(t), g(t)\rangle \le \varepsilon_R$. In the rough case, integrate the Ward identity on each approximant. For fixed $T$ and $t\le T$, let $r\to\infty$ in the resulting quadratic form, using \eqref{eq:M8-atom-preserving-passage}; then let $T\to\infty$. This gives
\[
\sup_{t\ge T_0}\langle A_Rg(t), g(t)\rangle\le\sup_{T>T_0}\limsup_{r\to\infty}\varepsilon_{R,r}^{(T)}.
\]
Since $0\le K_R\le\mathrm{Id}$,
\[
\|g(t)-K_Rg(t)\|_{L^2(\mathbb{R})}^2=\langle (\mathrm{Id}-K_R)^2g(t), g(t)\rangle\le\langle A_Rg(t), g(t)\rangle.
\]
The orbit has fixed $L^2$ norm and is therefore uniformly approximated by the relatively compact sets $K_R\{g(t):t\ge T_0\}$. Hence the orbit itself is relatively
compact in $L^2(\mathbb{R})$. Compactness and Plancherel give uniform physical and Fourier tails, while \eqref{eq:M8-flux-negative-tail} gives the Hardy defect.
\end{proof}
\begin{re}\label{rem:M8-leakage-ledger}
Applying Proposition~\ref{prop:M8-moving-leakage} requires \eqref{eq:M8-integrated-leakage} (for example, from a sign, an integrable generator term, or a telescoping law) and, for rough orbits, \eqref{eq:M8-rough-leakage-uniform}--\eqref{eq:M8-atom-preserving-passage}.
\end{re}
\subsection{Affine covariance and channel decoupling}
Fix a point label $\mu=\mu_j\in\sigma_{\mathrm{p}}(L_{u_0})$, a sequence $t_n\to\infty$, and abbreviate
\[
b_n=b_j(t_n),\qquad v_n=u(t_n)-b_n,\qquad M_n=\mathcal{M}_{\lambda_n,y_n,\xi_n,\theta_n}.
\]
Put
\[
g_n=M_n^{-1}b_n,\qquad w_n=M_n^{-1}v_n,\qquad \widetilde{u}_n=g_n+w_n,\qquad\kappa_n=\lambda_n(\mu-\xi_n).
\]
The inverse image of the laboratory Hardy space is
\[
\mathfrak{H}_n:=M_n^{-1}L^2_+(\mathbb{R})=\{f\in L^2(\mathbb{R}):\operatorname{supp}\widehat{f}\subset[-\lambda_n\xi_n,\infty)\}.
\]
At finite $n$, the profiles $g_n,w_n$ belong to the moving Hardy space $\mathfrak{H}_n$, and the conjugated formula below is interpreted on that space.

\begin{prop}[Affine covariance and the channel defect]\label{prop:M9-affine-covariance}
For $f\in\mathfrak{H}_n$ such that $M_nf\in H^1_+(\mathbb{R})$, define $\mathscr{L}_zf:=Df-z\Pi_+(\overline{z}f)$. Then
\begin{equation}\label{eq:M9-affine-covariance}
M_n^{-1}L_{u(t_n)}M_nf=\xi_nf+\lambda_n^{-1}\mathscr{L}_{\widetilde{u}_n}f.
\end{equation}
Consequently the eigenvalue equation is equivalent to
\begin{equation}\label{eq:M9-affine-defect}
(\mathscr{L}_{g_n}-\kappa_n)g_n=F_n,
\end{equation}
where
\begin{equation}\label{eq:M9-cross-defect}
F_n=g_n\Pi_+(\overline{w_n}g_n)+w_n\Pi_+(|g_n|^2)+w_n\Pi_+(\overline{w_n}g_n).
\end{equation}
Equivalently, $F_n$ is $\lambda_nM_n^{-1}$ applied to the right-hand side of the laboratory identity
\begin{equation}\label{eq:channel-defect-report}
(L_{b_n}-\mu)b_n=b_n\Pi_+(\overline{v_n}b_n)+v_n\Pi_+(|b_n|^2)+v_n\Pi_+(\overline{v_n}b_n).
\end{equation}
All three identities hold in $H^{-1}(\mathbb{R})$, equivalently in distributions.
\end{prop}
\begin{proof}
Writing $z=(x-y_n)/\lambda_n$, direct differentiation gives $M_n^{-1}DM_n=\lambda_n^{-1}D+\xi_n$. The common carrier phases cancel in $\overline{u(t_n,x)}[M_nf](x)=\lambda_n^{-1}\overline{\widetilde{u}_n(z)}f(z)$. The Cauchy--Szeg\H{o} projector commutes with positive dilations and translations. This proves \eqref{eq:M9-affine-covariance}. Substitute $f=g_n$ and expand $\widetilde{ u}_n=g_n+w_n$ to obtain \eqref{eq:M9-affine-defect}--\eqref{eq:M9-cross-defect}. Expanding $u=b_n+v_n$ directly in $L_ub_n=\mu b_n$ gives \eqref{eq:channel-defect-report}.
\end{proof}

\begin{prop}[Evolution of the point channel at $H^1$ regularity]
\label{prop:M9-actual-column-evolution}
Let $u\in C^0(I;H^1_+)$ be the canonical $H^1_+$ CM-DNLS flow and use the transporting generator $\mathbf{B}_u$ from \eqref{eq:M8-actual-generator}.
For the canonically phased point channel and its complement,
\begin{equation}\label{eq:M9-actual-column-evolution}
\partial_tb_j=\mathbf B_ub_j,\qquad\partial_tv_j=\mathbf B_uv_j,\qquad v_j=u-b_j.
\end{equation}
These equations are equalities in distributions and hence in the corresponding $H^{-1}$ evolution class.
\end{prop}
\begin{proof}
The standard eigenvector $\varphi_j^{u(t)}$ is normalized by $\|\varphi_j^{u(t)}\|_{L^2(\mathbb{R})}=1$ and $\langle u(t), \varphi_j^{u(t)}\rangle=\sqrt{2\pi}>0$. Hence the rank-one projection formula gives $b_j(t)=\sqrt{2\pi}\,\varphi_j^{u(t)}$. Theorem~1.6 of \cite{Sun-2026-arXiv} gives the first equation in \eqref{eq:M9-actual-column-evolution}. The CM-DNLS equation is $(\partial_t-\mathbf{B}_u)u=0$; subtraction gives the second.
\end{proof}

Since $L^1(\mathbb{R})\hookrightarrow H^{-1}(\mathbb{R})$ and $\Pi_+$ is an $L^2$-contraction, the defect satisfies
\begin{equation}\label{eq:M9-minimal-product-bound}
\|F_n\|_{H^{-1}(\mathbb{R})}\le C\left[(\|g_n\|_{L^2(\mathbb{R})}+\|w_n\|_{L^2(\mathbb{R})})\|\Pi_+(\overline{w_n}g_n)\|_{L^2(\mathbb{R})}+\|w_n\Pi_+(|g_n|^2)\|_{L^1(\mathbb{R})}\right].
\end{equation}
Thus it suffices to prove
\begin{equation}\label{eq:M9-column-gate}
\|\Pi_+(\overline{w_n}g_n)\|_{L^2(\mathbb{R})}\underset{n\to\infty}{\longrightarrow}0,\qquad\|w_n\Pi_+(|g_n|^2)\|_{L^1(\mathbb{R})}\underset{n\to\infty}{\longrightarrow}0.
\end{equation}

\begin{lemma}[Hankel decomposition in an affine frame]\label{lem:M9-carrier-Hankel}
Put $c_n=\lambda_n\xi_n$, $G_n=\mathcal{F}_0g_n$, and $W_n=\mathcal{F}_0 w_n$, and write
\[
G_n^+=\mathbf{1}_{[0,\infty)}G_n,\qquad W_n^+=\mathbf{1}_{[0,\infty)}W_n,\qquad W_n^-=\mathbf{1}_{[-c_n,0)}W_n.
\]
Since $G_n,W_n$ are supported in $[-c_n,\infty)$, for $\zeta\ge0$,
\begin{equation}\label{eq:M9-Hankel-split}
\mathcal{F}_0\Pi_+(\overline{w_n}g_n)(\zeta)=\mathsf{H}_{G_n^+}W_n^+(\zeta)+\mathsf{C}_nW_n^-(\zeta),
\end{equation} $[\mathsf{H}_GW](\zeta)=\frac{1}{\sqrt{2\pi}}\int_0^\infty\overline{W(\eta)}G(\eta+\zeta)d\eta$,  $[\mathsf{C}_nW^-](\zeta)=\frac{1}{\sqrt{2\pi}}\int_{-c_n}^0\overline{W^-(\eta)}G_n(\eta+\zeta)d\eta$. Both maps are conjugate-linear in $W$; $\mathsf{H}_G$ is understood as a Hilbert--Schmidt map after composition with the canonical conjugation on $L^2(\mathbb{R}_+)$. With this convention,
\begin{equation}\label{eq:M9-Hankel-S2}
\|\mathsf{H}_G\|_{\mathfrak{S}_2}^2=\frac{1}{2\pi}\int_0^\infty r|G(r)|^2dr.
\end{equation}
Hence strong convergence of $G_n^+$ in $L^2(\mathbb{R}_+,rdr)$ and weak convergence of $W_n^+$ remove the first term. The remaining condition is
\begin{equation}\label{eq:M9-subcarrier-gate}
\|\mathsf{C}_nW_n^-\|_{L^2(\mathbb{R})}\underset{n\to\infty}{\longrightarrow}0.
\end{equation}
If the two conclusions hold, then
\begin{equation}\label{eq:M9-Hankel-output}
\mathfrak{e}_n=\|\Pi_+(\overline{w_n}g_n)\|_{L^2(\mathbb{R})}^2\underset{n\to\infty}{\longrightarrow}0.
\end{equation}

A stronger sufficient condition uses the common carrier shift
\[
\breve{g}_n=e^{ic_nx}g_n=\mathcal{M}_{\lambda_n,y_n,0,\theta_n}^{-1}b_n,\qquad\breve{w}_n=e^{ic_nx}w_n=\mathcal{M}_{\lambda_n,y_n,0,\theta_n}^{-1}v_n.
\]
It satisfies
\begin{equation}\label{eq:M9-column-scale-ledger}
\Pi_+(\overline{\breve{w}_n}\breve{g}_n)=\Pi_+(\overline{w_n}g_n),\qquad\mathfrak{e}_n=\lambda_n\|\Pi_+(\overline{v_n}b_n)\|_{L^2(\mathbb{R})}^2.
\end{equation}
Thus
\begin{equation}\label{eq:M9-lifted-compactness}
\breve{w}_n\rightharpoonup0\text{ in }L^2(\mathbb{R}),\qquad\breve{g}_n\to\breve{g}\text{ in }H^{1/2}_+(\mathbb{R})
\end{equation}
also implies \eqref{eq:M9-Hankel-output}; strong $L^2$ convergence plus a uniform $H^s$ bound, $s>\frac{1}{2}$, is sufficient by interpolation.
\end{lemma}
\begin{proof}
The Fourier product formula gives \eqref{eq:M9-Hankel-split} after splitting the integration variable at zero. Fubini gives $\|\mathsf{H}_G\|_{\mathfrak{S}_2}^2=\frac{1}{2\pi}\int_0^\infty\int_0^\infty|G(\eta+\zeta)|^2d\eta d\zeta$, which is \eqref{eq:M9-Hankel-S2}.  Therefore the positive-frequency operators converge in Hilbert--Schmidt norm and compactness converts weak convergence of $W_n^+$ into strong output convergence. Together with \eqref{eq:M9-subcarrier-gate} this proves \eqref{eq:M9-Hankel-output}. For the shifted profiles the common modulation cancels in $\overline{w}_ng_n$, and the same Hilbert--Schmidt argument applies on the fixed half-line. The affine scaling gives \eqref{eq:M9-column-scale-ledger}.
\end{proof}

\begin{lemma}[Weak convergence of radiation in an affine frame]\label{lem:M9-P3-two-weak-rows}
Let
\[
w_n^{\mathrm{ac}}=M_n^{-1}r_{\mathrm{ac}}(t_n),\qquad\breve{w}_n^{\mathrm{ac}}=\mathcal{M}_{\lambda_n,y_n,0,\theta_n}^{-1}r_{\mathrm{ac}}(t_n).
\]
Under \eqref{eq:P3}, $w_n^{\mathrm{ac}}\rightharpoonup0,\qquad \breve{w}_n^{\mathrm{ac}}\rightharpoonup0\quad\hbox{in }L^2(\mathbb{R})$. Hence the Hankel lemma removes the continuous-radiation contribution to the positive residual whenever the lifted point profile has \eqref{eq:M9-lifted-compactness}. The remaining point columns still require point--point separation in this frame.
\end{lemma}
\begin{proof}
Replace $r_{\mathrm{ac}}(t_n)$ by $e^{it_n\partial_x^2}u_+$, using strong scattering and unitarity of the affine maps. The two weak limits are then Lemma~\ref{lem:affine-weak}, first with the original carrier and then with carrier zero.
\end{proof}

\begin{lemma}[Local decay from smoothing and time regularity]\label{lem:M9-Rellich}
Let $w:[T_0,\infty)\to L^2(\mathbb{R})$ be bounded. Suppose that for every $\chi\in C_c^\infty(\mathbb{R})$ there are $\sigma>0$ and $C_\chi<\infty$ such that
\begin{equation}\label{eq:M9-local-smoothing}
\int_{T_0}^\infty\|\chi w(t)\|_{L^2(\mathbb{R})}^2dt<\infty,\quad\sup_{t\ge T_0}\|\chi w(t)\|_{H^\sigma(\mathbb{R})}<\infty,
\end{equation}
\begin{equation}\label{eq:M9-no-spike}
\|\chi(w(t+h)-w(t))\|_{H^{-1}(\mathbb{R})}\le C_\chi|h|
\end{equation}
whenever $t,t+h\ge T_0$. Then
\[
\|\chi w(t)\|_{L^2(\mathbb{R})}\underset{t\to\infty}{\longrightarrow}0.
\]
Consequently every sequence $t_n\to\infty$ satisfies \eqref{eq:M9-window-escape} with $w_n=w(t_n)$.
\end{lemma}
\begin{proof}
Interpolation between $H^{-1}(\mathbb{R})$ and $H^\sigma(\mathbb{R})$ gives, uniformly in $t$, $\|\chi(w(t+h)-w(t))\|_{L^2(\mathbb{R})}\le C|h|^{\sigma/(1+\sigma)}$. Thus $t\mapsto\|\chi w(t)\|_{L^2(\mathbb{R})}^2$ is a nonnegative uniformly continuous function. It is integrable by \eqref{eq:M9-local-smoothing}, so it tends to zero. Taking a cutoff equal to one on $[-R,R]$ proves the sequential conclusion.
\end{proof}

\begin{prop}[Local channel defect] \label{prop:M9-Lax-column}
Assume $g_n\to g$ in $L^2(\mathbb{R})$ and $\sup_n\|g_n\|_{H^s(\mathbb{R})}<\infty$ for some $s>1/2$. Suppose also that
\[
w_n\rightharpoonup0\quad\hbox{in }L^2(\mathbb{R}),\qquad\sup_n\|w_n\|_{L^2(\mathbb{R})}<\infty,\qquad\mathfrak{e}_n:=\|\Pi_+(\overline{w_n}g_n)\|_{L^2(\mathbb{R})}^2\underset{n\to\infty}{\longrightarrow}0.
\]
Then $F_n\to0$ in $H^{-1}_{\mathrm{loc}}(\mathbb{R})$. The residual has the nonnegative form
\begin{equation}\label{eq:M9-pair-energy-identity}
\mathfrak{e}_n=\int_\mathbb{R}w_n\Pi_+(\overline{w_n}g_n)\overline{g_n}dx=\|\Pi_+(\overline{w_n}g_n)\|_{L^2(\mathbb{R})}^2.
\end{equation}
The integral is the absolutely convergent $L^1$--$L^\infty$ duality pairing. The numbers $\{\kappa_n\}$ are bounded. If, in addition, the negative-frequency defect vanishes, then every subsequential limit $\kappa_n\to\kappa$ satisfies the same graph conclusion \eqref{eq:M9-limit-graph}.
\end{prop}
\begin{proof}
Interpolation gives $g_n\to g$ in $H^{s^\prime}$ for some $s^\prime>\frac{1}{2}$. Since $H^{s^\prime}(\mathbb{R})$ is an algebra and $\Pi_+$ is bounded on it,
\[
h_n:=\Pi_+(|g_n|^2)\longrightarrow h:=\Pi_+(|g|^2)\quad\hbox{in }H^{s^\prime}(\mathbb{R})\hookrightarrow L^\infty(\mathbb{R}).
\]
Thus $w_nh_n\rightharpoonup0$ in $L^2(\mathbb{R})$. For every compactly supported cutoff $\chi$, the embedding $L^2(\operatorname{supp}\chi)\hookrightarrow H^{-1}(\mathbb{R})$ is compact, and hence
\[
\chi w_n\Pi_+(|g_n|^2)\underset{n\to\infty}{\longrightarrow}0\quad\hbox{in }H^{-1}(\mathbb{R}).
\]
The other two terms are globally small:
\[
\|g_n\Pi_+(\overline{w_n}g_n)\|_{L^1(\mathbb{R})}\le\|g_n\|_{L^2(\mathbb{R})}\mathfrak{e}_n^{1/2},\qquad
\|w_n\Pi_+(\overline{w_n}g_n)\|_{L^1(\mathbb{R})}\le\|w_n\|_{L^2(\mathbb{R})}\mathfrak{e}_n^{1/2}.
\]
The embedding $L^1(\mathbb{R})\hookrightarrow H^{-1}(\mathbb{R})$ proves the assertion. Finally, with $z_n=\overline{w_n}g_n$ and $a_n=\Pi_+ z_n$, the duality integral in \eqref{eq:M9-pair-energy-identity} is $\langle a_n, z_n\rangle=\langle \Pi_+z_n,z_n\rangle=\|a_n\|_{L^2(\mathbb{R})}^2$, by orthogonality of $\Pi_+$. The convergence \eqref{eq:M9-self-form-close} follows as in the next proof.  Testing \eqref{eq:M9-affine-defect} against one compactly supported $\phi$ with $\langle g,\phi\rangle\ne0$ bounds
$\kappa_n$. Passing to the limit in distributions and then using $g\in H^s(\mathbb{R})\subset L^\infty(\mathbb{R})$ upgrades the limiting equation to \eqref{eq:M9-limit-graph}.
\end{proof}

\begin{prop}[Global convergence of the channel defect]\label{prop:M9-moving-window}
Assume, after the extraction, that for some $s>\frac{1}{2}$
\begin{equation}\label{eq:M9-graph-input}
g_n\underset{n\to\infty}{\longrightarrow}g\quad\hbox{in }L^2(\mathbb{R}),\quad\sup_n\|g_n\|_{H^s(\mathbb{R})}<\infty,\quad\sup_n\|w_n\|_{L^2(\mathbb{R})}<\infty,
\end{equation}
\begin{equation}\label{eq:M9-window-escape}
\|w_n\|_{L^2(-R,R)}\underset{n\to\infty}\longrightarrow0\quad\text{for every }R<\infty.
\end{equation}
Then \eqref{eq:M9-column-gate} holds and
\begin{equation}\label{eq:M9-defect-Hminus1}
F_n\underset{n\to\infty}{\longrightarrow}0\quad\hbox{in }H^{-1}(\mathbb{R}).
\end{equation}
If the negative-frequency defect also vanishes, then $g\in L^2_+(\mathbb{R})$. The numbers $\kappa_n$ are bounded, and along every subsequence on which $\kappa_n\to\kappa$,
\begin{equation}\label{eq:M9-limit-graph}
g\in H^1_+(\mathbb{R}),\qquad L_gg=\kappa g\quad\hbox{in }L^2(\mathbb{R}).
\end{equation}
\end{prop}
\begin{proof}
Interpolation gives $g_n\to g$ in $H^{s^\prime}(\mathbb{R})$ for every $\frac{1}{2}<s^\prime<s$, hence uniformly on $\mathbb{R}$. The limit is continuous and vanishes at infinity. Splitting into a fixed window and its complement gives
\[
\|w_ng_n\|_{L^2(\mathbb{R})}\le\|w_n\|_{L^2(-R,R)}\|g_n\|_{L^\infty(\mathbb{R})}+\|w_n\|_{L^2(\mathbb{R})}\|g_n\|_{L^\infty(|x|>R)}\longrightarrow0
\]
by first sending $n\to\infty$ and then $R\to\infty$.

Set $h_n=\Pi_+(|g_n|^2)$ and $h=\Pi_+(|g|^2)$. The estimate
\[
\||g_n|^2-|g|^2\|_{L^2(\mathbb{R})}\le(\|g_n\|_{L^\infty(\mathbb{R})}+\|g\|_{L^\infty(\mathbb{R})})\|g_n-g\|_{L^2(\mathbb{R})}
\]
shows that $h_n\to h$ in $L^2(\mathbb{R})$. Hence $\{h_n\}$ is uniformly $L^2$-tight, and the same window decomposition gives $\|w_nh_n\|_{L^1(\mathbb{R})}\to0$. Formula
\eqref{eq:M9-minimal-product-bound} proves \eqref{eq:M9-defect-Hminus1}.

Moreover,
\begin{equation}\label{eq:M9-self-form-close}
g_n\Pi_+(|g_n|^2)\underset{n\to\infty}{\longrightarrow} g\Pi_+(|g|^2)\quad\hbox{in }L^1(\mathbb{R})\hookrightarrow H^{-1}(\mathbb{R}).
\end{equation}
Choose $\phi\in C_c^\infty(\mathbb{R})$ with $\langle g,\phi\rangle\ne0$. Testing \eqref{eq:M9-affine-defect} against $\phi$, using $\|Df\|_{H^{-1}(\mathbb{R})}\le\|f\|_{L^2(\mathbb{R})}$, shows that $\{\kappa_n\}$ is bounded. After $\kappa_n\to\kappa$, all terms pass in $H^{-1}(\mathbb{R})$: $Dg-g\Pi_+(|g|^2)=\kappa g$. Since $g\in H^s(\mathbb{R})\subset L^\infty(\mathbb{R})$ and $\Pi_+(|g|^2)\in L^2(\mathbb{R})$, the right side is in $L^2(\mathbb{R})$. Therefore $Dg\in L^2(\mathbb{R})$, and the negative-frequency condition gives \eqref{eq:M9-limit-graph}.
\end{proof}

\begin{lemma}[A self-eigenvalue is the lower Fourier edge]\label{lem:M9-lower-edge}
If $0\ne g\in H^1_+(\mathbb{R})$ and $L_gg=\kappa g$, then $\kappa=\underline\omega(g):=\operatorname{ess\,inf}\operatorname{supp}\widehat{g}\ge0$. \end{lemma}
\begin{proof}
Put $f=\widehat{g}$, $H(\eta)=\mathbf{1}_{[0,\infty)}(\eta)\widehat{|g|^2}(\eta)$, and $a=\underline\omega(g)$. Since $g\in H^1(\mathbb{R})$, $H\in L^2(\mathbb{R})\cap L^\infty(\mathbb{R})$. Fourier transformation gives
\begin{equation}\label{eq:M9-Volterra}
(\zeta-\kappa)f(\zeta)=\frac{1}{2\pi}\int_0^\zeta f(\zeta-\eta)H(\eta)d\eta.
\end{equation}
Writing $F(r)=f(a+r)$, if $a\ne\kappa$, choose $T>0$ such that $|a+r-\kappa|\ge\frac{1}{2}|a-\kappa|$ on $(0,T)$. Young's inequality applies because at $\zeta=a+r$ the convolution in \eqref{eq:M9-Volterra} truncates to $\int_0^rF(r-\eta)H(\eta)d\eta$: the remaining part vanishes below the lower edge. It gives $\|F\|_{L^2(0,T)}\le \frac{1}{\pi|a-\kappa|}\|H\|_{L^1(0,T)}\|F\|_{L^2(0,T)}$. For small $T$, the coefficient is less than one, forcing $F=0$ on $(0,T)$, contrary to the definition of the lower edge. Thus $a=\kappa$, and $a\ge0$ by the Hardy support.
\end{proof}

\begin{cor}[Carrier recalibration and rigidity]\label{cor:M9-geometric-P4}
Assume Theorem~\ref{thm:M8-two-tail} and either the hypotheses of Proposition~\ref{prop:M9-moving-window} or those of Proposition~\ref{prop:M9-Lax-column}. Pass to a subsequence with $\kappa_n\to\kappa$. Define
\[
(\widetilde{\xi}_n,\widetilde{g}_n,\widetilde{w}_n)=
\begin{cases}
(\mu,e^{-i\kappa_nx}g_n,e^{-i\kappa_nx}w_n),&\mu>0,\\
(\xi_n,g_n,w_n),&\mu\le0.
 \end{cases}
\]
Then
\begin{equation}\label{eq:M9-same-vector}
\mathcal{M}_{\lambda_n,y_n,\widetilde\xi_n,\theta_n}\widetilde{g}_n=b_n,
\end{equation}
\[
\lambda_n(\mu-\widetilde{\xi}_n)\underset{n\to\infty}{\longrightarrow}0.
\]
For $\mu>0$, the second line is identically zero. For $\mu\le0$, it follows because $\kappa_n\le0$, whereas Lemma~\ref{lem:M9-lower-edge} gives $\kappa\ge0$; hence
$\kappa=0$. In particular, if $\mu<0$, then $\lambda_n\le-\kappa_n/|\mu|\to0$.
Moreover,
$\widetilde{g}_n\to\widetilde{g}$ in $L^2(\mathbb{R})$, where $\widetilde{g}=e^{-i\kappa x}g$ for $\mu>0$ and $\widetilde{g}=g$ for $\mu\le0$, $\widetilde{g}\in H^1_+(\mathbb{R})$, and
\begin{equation}\label{eq:M10-zero-profile}
L_{\widetilde{g}}\widetilde{g}=0,\qquad\|\widetilde{g}\|_{L^2(\mathbb{R})}^2=2\pi.
\end{equation}
The ground-state uniqueness lemma \cite[Lemma~4.1]{Gerard-2024-CPAM} therefore yields $\widetilde{g}=\mathcal{M}_{\alpha,\beta,0,\vartheta}R$ for some $\alpha>0$, $\beta\in\mathbb{R}$, and $\vartheta\in\mathbb{R}/2\pi\mathbb{Z}$. Undoing the frame proves \eqref{eq:P4} for this point label.
\end{cor}
\begin{proof}
The identity in \eqref{eq:M9-same-vector} follows from $\kappa_n=\lambda_n(\mu-\xi_n)$. For $\mu>0$, modulation of the profile cancels the change from $\xi_n$ to $\mu$. Strong convergence follows by dominated convergence. By Lemma~\ref{lem:M9-lower-edge}, the lower Fourier edge of $g$ is $\kappa$. If $\mu>0$, modulation shifts both that edge and the self-eigenvalue to zero. If $\mu\le0$, the sign argument above already gives $\kappa=0$.  Equation \eqref{eq:M8-full-profile} supplies the mass. Apply the classification to \eqref{eq:M10-zero-profile}. Finally use the affine
composition law as in \eqref{eq:M10prime-composition}.
\end{proof}

\subsection{Backward shifts and the Fano numerator}
For $a\ge0$, define the backward shift on $L^2_+(\mathbb{R})$ by $[\mathcal{F}_0(S_a^*f)](\zeta)=(\mathcal{F}_0f)(\zeta+a),\qquad \zeta\ge0$. It is the contraction semigroup obtained by left translation on the Fourier half-line.

\begin{thm}[Rigidity from the shift defect] \label{thm:M10prime-character}
Let $f_n\in L^2_+(\mathbb{R})$, $\|f_n\|_{L^2(\mathbb{R})}^2=2\pi$, and suppose that $f_n\to f$ strongly in $L^2(\mathbb{R})$. Assume that, for every $A<\infty$,
\begin{equation}\label{eq:M10prime-rank-one-defect}
\delta_n(A):=\sup_{0\le a\le A}
\left\{
\|S_a^\ast f_n\|_{L^2(\mathbb{R})}^2-\frac{\left|\langle S_a^\ast f_n, f_n\rangle\right|^2}{2\pi}\right\}\underset{n\to\infty}{\longrightarrow}0.
\end{equation}
Then there are $\alpha>0$, $\beta\in\mathbb{R}$, and $\vartheta\in\mathbb{R}/2\pi\mathbb{Z}$ such that
\begin{equation}\label{eq:M10prime-rigid-limit}
f=\mathcal{M}_{\alpha,\beta,0,\vartheta}R.
\end{equation}
\end{thm}
\begin{proof}
Put
\begin{equation}\label{eq:M10prime-mn}
m_n(a)=\frac{1}{2\pi}\langle S_a^\ast f_n, f_n\rangle,\qquad e_n(a)=S_a^\ast f_n-m_n(a)f_n.
\end{equation}
The expression in braces in \eqref{eq:M10prime-rank-one-defect} is $\|e_n(a)\|_{L^2(\mathbb{R})}^2$. Since $S_{a+b}^\ast=S_b^\ast S_a^\ast$,
\[
\sqrt{2\pi}|m_n(a+b)-m_n(a)m_n(b)|\le\|e_n(a+b)\|_{L^2(\mathbb{R})}+\|e_n(a)\|_{L^2(\mathbb{R})}+\|e_n(b)\|_{L^2(\mathbb{R})}.
\]
Hence the approximate character law holds.

Strong convergence gives $m_n(a)\to m(a)=(2\pi)^{-1}\langle S_a^\ast f,f\rangle$, locally uniformly after using the strong continuity and contractivity of the shift semigroup.
Passing to the limit in \eqref{eq:M10prime-mn} gives
\begin{equation}\label{eq:M10prime-exact-character}
S_a^\ast f=m(a)f,\qquad m(a+b)=m(a)m(b),\qquad m(0)=1.
\end{equation}
The continuous contractive characters of $\mathbb{R}_+$ are $m(a)=e^{-aw}$ with $\operatorname{Re}w\ge0$. Square integrability forces $w=\alpha+i\beta$ with $\alpha>0$.

Writing $F=\mathcal{F}_0f$, \eqref{eq:M10prime-exact-character} says $F(\zeta+a)=e^{-aw}F(\zeta)$. Therefore $e^{w\zeta}F(\zeta)$ is translation invariant and is constant
almost everywhere. Thus $(\mathcal{F}_0f)(\zeta)=ce^{-(\alpha+i\beta)\zeta}\mathbf{1}_{\mathbb{R}_+}(\zeta)$. The mass condition gives $|c|^2=4\pi\alpha$. A residue calculation gives $(\mathcal{F}_0R)(\zeta)=-2i\sqrt\pi e^{-\zeta}\mathbf{1}_{\mathbb{R}_+}(\zeta)$. Choosing $\vartheta$ to match the phase of $c$ proves \eqref{eq:M10prime-rigid-limit} by \eqref{eq:modulation-fourier}.
\end{proof}

\begin{prop}[Generator residual criterion]\label{prop:M10prime-generator}
Let
\[
\mathcal{A}_+=\mathcal{F}_0^{-1}\partial_\zeta\mathcal{F}_0,\qquad\mathrm{Dom}(\mathcal{A}_+)=\{f\in L^2_+(\mathbb{R}):(\mathcal{F}_0f)|_{\mathbb{R}_+}\in H^1(\mathbb{R}_+)\}.
\]
For $0\ne f\in\mathrm{Dom}(\mathcal{A}_+)$, set
\[
m=\|f\|_{L^2(\mathbb{R})}^2,\qquad P_fh=\frac{1}{m}\langle h, f\rangle f,\qquad\epsilon_+(f)=\|(\mathrm{Id}-P_f)\mathcal{A}_+f\|_2.
\]
Then, for every $a\ge0$,
\begin{equation}\label{eq:M10prime-generator-bound}
\delta_f(a):={\|S_a^\ast f\|_{L^2(\mathbb{R})}^2-\frac{|\langle S_a^\ast f,f\rangle|^2}{m}}=\|(\mathrm{Id}-P_f)S_a^\ast f\|_{L^2(\mathbb{R})}^2\leq a^2\epsilon_+(f)^2,                    \end{equation}
\begin{equation}\label{eq:M10prime-generator-limit}
\lim_{a\to0}\frac{\delta_f(a)}{a^2}=\epsilon_+(f)^2.
\end{equation}
Consequently, if $\|f_n\|_{L^2(\mathbb{R})}^2=2\pi$ and $\epsilon_+(f_n)\to0$, then
\[
\sup_{0\le a\le A}\delta_{f_n}(a)\le A^2\epsilon_+(f_n)^2\underset{n\to\infty}{\longrightarrow}0\qquad(A<\infty),
\]
which proves \eqref{eq:M10prime-rank-one-defect}.
\end{prop}
\begin{proof}
Write $F=\mathcal{F}_0f$ on $\mathbb{R}_+$, and let $z=\frac{\langle F^\prime, F\rangle}{m},\qquad r=F^\prime-zF$. Thus $r\perp F$ and $\|r\|_{L^2(\mathbb{R})}=\epsilon_+(f)$. Integration by parts gives
\begin{equation}\label{eq:M10prime-generator-dissipation}
2\operatorname{Re}z=\frac{2\operatorname{Re}\int_0^\infty F^\prime\overline{F}d\zeta}{m}=-\frac{|F(0)|^2}{m}\le0.
\end{equation}
For the left-translation semigroup $T_aF=F(\,\cdot+a)$, variation of constants gives $T_aF=e^{az}F+\int_0^ae^{(a-s)z}T_srds$. Contractivity of $T_s$ and \eqref{eq:M10prime-generator-dissipation} imply
\[
\operatorname{dist}_{L^2}(S_a^\ast f,\mathbb{C}f)\le\|T_aF-e^{az}F\|_{L^2(\mathbb{R})}\le a\|r\|_{L^2(\mathbb{R})},
\]
which proves \eqref{eq:M10prime-generator-bound}. Finally, $T_aF=F+aF^\prime+o(a)$ strongly in $L^2(\mathbb{R}_+)$. Apply $\mathrm{Id}-P_f$, divide by $a$, and take norms to obtain \eqref{eq:M10prime-generator-limit}.
\end{proof}
\begin{re}\label{rem:M10prime-generator-boundary}
For $F=\mathcal{F}_0f\in H^1(\mathbb{R}_+)$, inverse Fourier integration by parts gives the distributional identity
\begin{equation}\label{eq:M10prime-physical-generator}
\mathcal{A}_+f=-ixf-\frac{F(0)}{\sqrt{2\pi}}.
\end{equation}
The constant term cancels the nondecaying part of $-ixf$ for an affine $R$ profile. If
$u\in H^1_+(\mathbb{R})$, $u=b+v$, $L_ub=\mu b$, and $b=E_{L_u}(\{\mu\})u$, the eigenline lies in $H^1_+(\mathbb{R})$, and the backward-shift commutator identity of
\cite[Lemma~5.1]{Gerard-2024-CPAM} gives
\begin{equation}\label{eq:M10prime-shift-radiation-bridge}
\lim_{a\to0}\frac{1}{a}(L_u-\mu)S_a^\ast b=v\quad\hbox{in }L^2(\mathbb{R}).
\end{equation}
A reduced-resolvent or Fano estimate for the transverse inverse in \eqref{eq:M10prime-shift-radiation-bridge} controls the generator residual through the complementary channel.
\end{re}

\begin{criterion}[Canonical Fano numerator for the shift generator]\label{crit:M10prime-Fano}
After the extraction, let $f_n$ be the positive normalization \eqref{eq:M10prime-positive-normalization} and put $f_n\in\mathrm{Dom}(\mathcal{A}_+),\qquad q_n=(\mathrm{Id}-P_{f_n})\mathcal{A}_+f_n$. Assume there are fixed Banach spaces $\mathcal{X}_j,\mathcal{Y}_j$, a reduced column $\mathscr{L}_{j,n}:\mathrm{Dom}(\mathscr{L}_{j,n})\subset L^2(\mathbb{R})\to\mathcal{Y}_j$, an injection $\iota_j\in\mathcal{B}(L^2(\mathbb{R}),\mathcal{Y}_j)$, and a moving-boundary correction $\mathcal{J}_{j,n}\in\mathcal{Y}_j$ such that
\[
q_n\in\mathrm{Dom}(\mathscr{L}_{j,n}),\qquad\mathscr{L}_{j,n}q_n=\iota_jw_n+\mathcal{J}_{j,n},\qquad \|\mathcal{J}_{j,n}\|_{\mathcal{Y}_j}\underset{n\to\infty}\longrightarrow0.
\]
The same fixed left/right Grushin gauge produces $\mathcal{G}^{\mathrm{reg}}_{j,n}\in\mathcal{B}(\mathcal{X}_j,L^2(\mathbb{R})),\qquad\mathcal{C}_{j,n}\in\mathcal{B}(\mathcal{Y}_j,\mathcal{X}_j)$, and the factorization
\begin{equation}\label{eq:M10prime-Fano-factorization}
q_n=\mathcal{G}^{\mathrm{reg}}_{j,n}\mathcal{C}_{j,n}\bigl(\iota_jw_n+\mathcal{J}_{j,n}\bigr),
\end{equation}
with
\begin{equation}\label{eq:M10prime-Fano-smallness}
\sup_n\|\mathcal{G}^{\mathrm{reg}}_{j,n}\|<\infty,\qquad\|\mathcal{C}_{j,n}(\iota_jw_n+\mathcal{J}_{j,n})\|_{\mathcal{X}_j}\underset{n\to\infty}{\longrightarrow}0.
\end{equation}
The correction vanishes for a zero-carrier frame and otherwise equals the inner Fourier-boundary jump. The pencil fixes all spaces and maps independently of the vectors $q_n,w_n$. Then \eqref{eq:M10prime-generator-certificate} holds and, together with the two-tail compactness criterion, implies sequential point-channel rigidity.
\end{criterion}
\begin{proof}
Equations \eqref{eq:M10prime-Fano-factorization}--\eqref{eq:M10prime-Fano-smallness} give $\|q_n\|_{L^2(\mathbb{R})}\to0$, which is the residual in
\eqref{eq:M10prime-generator-certificate}. Apply Proposition~\ref{prop:M10prime-generator} and Theorem~\ref{thm:M10prime-character}.
\end{proof}

\begin{cor}[Sequential rigidity from rank-one defect]\label{cor:M8-M10prime-P4}
Fix $j$ and $t_n\to\infty$, and assume \eqref{eq:M8-double-tail}--\eqref{eq:M8-negative-tail}. Put
\begin{equation}\label{eq:M10prime-positive-normalization}
f_n=\frac{\sqrt{2\pi}\Pi_+ g_n}{\|\Pi_+ g_n\|_{L^2(\mathbb{R})}},\qquad g_n=\mathcal{M}_{p_n}^{-1}b_j(t_n).
\end{equation}
Assume either that \eqref{eq:M10prime-rank-one-defect} holds for $f_n$, or that
\begin{equation}\label{eq:M10prime-generator-certificate}
f_n\in\mathrm{Dom}(\mathcal{A}_+),\qquad\|(\mathrm{Id}-P_{f_n})\mathcal{A}_+f_n\|_{L^2(\mathbb{R})}\underset{n\to\infty}{\longrightarrow}0.
\end{equation}
Then, after passing to a subsequence, there are admissible affine parameters $\widetilde{p}_n$ such that
\begin{equation}\label{eq:M10prime-P4-output}
\|b_j(t_n)-\mathcal{M}_{\widetilde{p}_n}R\|_{L^2(\mathbb{R})}\underset{n\to\infty}\longrightarrow0.
\end{equation}
\end{cor}
\begin{proof}
By \eqref{eq:M8-negative-tail}, $\|f_n-g_n\|_{L^2(\mathbb{R})}\to0$. The two-tail compactness theorem supplies a strongly convergent subsequence. Proposition~\ref{prop:M10prime-generator} reduces \eqref{eq:M10prime-generator-certificate}, when used, to \eqref{eq:M10prime-rank-one-defect}. Then Theorem~\ref{thm:M10prime-character} identifies its limit as $\mathcal{M}_{\alpha,\beta,0,\vartheta}R$. The composition law is
\begin{equation}\label{eq:M10prime-composition}
\mathcal{M}_{\lambda_n,y_n,\xi_n,\theta_n}\mathcal{M}_{\alpha,\beta,0,\vartheta}=\mathcal{M}_{\lambda_n\alpha,y_n+\lambda_n\beta,\xi_n,\theta_n+\vartheta+\lambda_n\beta\xi_n}.
\end{equation}
The resulting carrier is still $\xi_n\ge0$, and unitarity gives \eqref{eq:M10prime-P4-output}.
\end{proof}

\subsection{Resolvent residues and reconstruction}
The point Schur matrix $K_t(\eta)$ is a position--Lax resolvent block at fixed height. Inverse-scattering reconstruction uses the Jost Laurent block below.
\begin{prop}[Uniform vertical cyclic residue in $L^2$]\label{prop:M11-vertical-cyclic-residue}
Fix a point label $\mu_j$, and put
\[
R_t(z)=(L_{u(t)}-z)^{-1},\qquad m_t(z)=R_t(z)u(t),\qquad b_j(t)=E_{L_{u(t)}}(\{\mu_j\})u(t).
\]
Then, for every $\varepsilon>0$,
\begin{equation}\label{eq:M11-vertical-residue-error}
\left\|-i\varepsilon m_t(\mu_j+i\varepsilon)-b_j(t)\right\|_{L^2(\mathbb{R})}^2=\int_{\mathbb{R}\setminus\{\mu_j\}}\frac{\varepsilon^2}{(\lambda-\mu_j)^2+\varepsilon^2}d\nu_{u_0}(\lambda).
\end{equation}
In particular,
\begin{equation}\label{eq:M11-vertical-residue-uniform}
\lim_{\varepsilon\to0}\sup_{t\ge0}\left\|-i\varepsilon(L_{u(t)}-\mu_j-i\varepsilon)^{-1}u(t)-b_j(t)\right\|_{L^2(\mathbb{R})}=0.
\end{equation}
Equivalently, if $J_tc=cu(t)$ and $U_{j,t}c=cb_j(t)$, then $-i\varepsilon R_t(\mu_j+i\varepsilon)J_t\longrightarrow U_{j,t}\quad\hbox{in }\mathcal B(\mathbb{C},L^2_+(\mathbb{R}))$, uniformly in $t$.
\end{prop}
\begin{proof}
The spectral multiplier of $-i\varepsilon R_t(\mu_j+i\varepsilon)u(t)$ is $-i\varepsilon/(\lambda-\mu_j-i\varepsilon)$, which equals one at $\lambda=\mu_j$. Removing that atom and using orthogonality of spectral subspaces gives \eqref{eq:M11-vertical-residue-error}, because $\nu_{u(t)}=\nu_{u_0}$. The integrand is bounded by one and converges pointwise to zero away from $\mu_j$. Dominated convergence proves \eqref{eq:M11-vertical-residue-uniform}; the source space in the last formulation is one dimensional.
\end{proof}

\begin{prop}[Frank--Read cyclic phase]\label{prop:M11-FR-cyclic-phase}
Assume $u_0\in L^1(\mathbb{R})\cap L^2_+(\mathbb{R})$, let $b_j=E_{L_{u_0}}(\{\mu_j\})u_0$, and choose the Frank--Read eigenfunction by $\langle u_0, \varphi_j\rangle=2\pi i$.  Then $\|\varphi_j\|_{L^2(\mathbb{R})}^2=2\pi,\qquad b_j=i\varphi_j,\qquad \varphi_j=-i b_j$. If $m_0(k)=(L_{u_0}-k)^{-1}u_0$ off the spectrum and $H_j(k)=m_0(k)+\frac{b_j}{k-\mu_j}$, then the local boundary expansion is
\begin{equation}\label{eq:M11-FR-b-residue}
m_0(k)=-\frac{b_j}{k-\mu_j}+H_j(k),
\end{equation}
\begin{equation}\label{eq:M11-FR-b-regular}
H_j(\mu_j\pm0i)=-i(\gamma_j+x)b_j\quad\hbox{in }(L^\infty(\mathbb{R}),w^\ast).
\end{equation}
If, in addition, $b_j\in L^1_+(\mathbb{R})\cap\mathrm{Dom}(\mathcal{A}_+)$, then the boundary vector in \eqref{eq:M11-FR-b-regular} belongs to $L^2(\mathbb{R})$. For
$Q_j=\mathrm{Id}-(2\pi)^{-1}\langle \cdot,b_j\rangle b_j$, one has
\begin{equation}\label{eq:M11-gamma-annihilated}
Q_jH_j(\mu_j\pm0i)=-i Q_j(xb_j)=Q_j\mathcal{A}_+b_j.
\end{equation}
The projection $Q_j$ removes $\gamma_j$, so the generator residual equals the transverse Laurent column.
\end{prop}
\begin{proof}
The overlap identity gives $\langle u_0, \varphi_j\rangle/\|\varphi_j\|_{L^2(\mathbb{R})}^2=i$, whence $\|\varphi_j\|_{L^2(\mathbb{R})}^2=2\pi$. The rank-one projection formula gives
\[
b_j=\frac{\langle u_0, \varphi_j\rangle}{\|\varphi_j\|_{L^2(\mathbb{R})}^2}\varphi_j=i\varphi_j.
\]
Substitution into the local Laurent expansion of \cite[Theorem~8.1 and Lemma~8.5]{Frank-2026-arXiv} proves \eqref{eq:M11-FR-b-residue}--\eqref{eq:M11-FR-b-regular}. Frank--Read's eigenfunction decay lemma \cite[Lemma~3.2]{Frank-2026-arXiv} gives $b_j\in L^1_+(\mathbb{R})$ under the standing $u_0\in L^1_+(\mathbb{R})\cap L^2_+(\mathbb{R})$ assumption. If also $b_j\in\mathrm{Dom}(\mathcal{A}_+)$, continuity of its Fourier transform and vanishing on $(-\infty,0)$ give $(\mathcal{F}_0 b_j)(0)=0$. Equation \eqref{eq:M10prime-physical-generator} then proves \eqref{eq:M11-gamma-annihilated}. This enhanced regularity promotes the weak-star boundary to the displayed $L^2$ identity.
\end{proof}

\begin{prop}[Weighted norming cocycle]\label{prop:M11-weighted-gamma-cocycle}
Let $I$ be a time interval, $s>\frac{1}{2}$, and let $u$ be a classical CM-DNLS solution in the persistent class
\begin{equation}\label{eq:M11-weighted-flow-class}
u\in C^0\bigl(I;H_s^2(\mathbb{R})\cap H_+^3(\mathbb{R})\bigr),\qquad H_s^2(\mathbb{R}):=\{f:\langle x\rangle^s\partial_x^kf\in L^2(\mathbb{R}),\ 0\le k\le2\}.
\end{equation}
For example, the invariant class $u(0)\in H_1^2(\mathbb{R})\cap H_+^4(\mathbb{R})$ supplies \eqref{eq:M11-weighted-flow-class} with $s=1$. Let $H_j(t,k)$ be the local pole-subtracted column in \eqref{eq:M11-FR-b-residue}, and define $\gamma_j(t)$ by its Frank--Read boundary value. Assume the following boundary limits: for every compact $J\Subset I$ and $\chi\in C_c^\infty(\mathbb{R})$, as $k\to\mu_j\pm0i$,
\begin{equation}\label{eq:M11-boundary-exchange-a}
\langle H_j(t,k),\chi\rangle\longrightarrow\langle-i(\gamma_j(t)+x),b_j(t)\chi\rangle\quad\hbox{locally uniformly for }t\in J,
\end{equation}
\begin{equation}\label{eq:M11-boundary-exchange-b}
\langle H_j(t,k),\mathbf{B}_{u(t)}^\ast\chi\rangle\longrightarrow\langle-i(\gamma_j(t)+x)b_j(t),\mathbf{B}_{u(t)}^\ast\chi\rangle\quad\hbox{in }L^1(J,dt).
\end{equation}
Then $\gamma_j$ is locally absolutely continuous and
\begin{equation}\label{eq:M11-gamma-cocycle}
\dot\gamma_j(t)=-2\mu_j\quad\hbox{for a.e. }t,\qquad\gamma_j(t)=\gamma_j(t_0)-2\mu_j(t-t_0).
\end{equation}
\end{prop}
\begin{proof}
Since $s>\frac{1}{2}$, Cauchy--Schwarz gives
\[
\|u(t)\|_{L^1(\mathbb{R})}\le\left(\int_\mathbb{R}\langle x\rangle^{-2s}dx\right)^{1/2}\|\langle x\rangle^su(t)\|_{L^2(\mathbb{R})}.
\]
Thus the Frank--Read Laurent chart is defined at every time. In the class \eqref{eq:M11-weighted-flow-class}, the off-spectrum Jost evolution of \cite[Theorem~1.5]{Sun-2026-arXiv}, the point-line evolution \cite[Theorem~1.6]{Sun-2026-arXiv}, and invariance of $\mu_j$ give $(\partial_t-\mathbf{B}_u)m_0(k)=0,\qquad(\partial_t-\mathbf{B}_u)\varphi_j=0$. Therefore $\widetilde{H}_j(k)=m_0(k)+i\varphi_j/(k-\mu_j)$ obeys the same homogeneous equation for $k$ off the spectrum. Its weak Duhamel identity, followed by \eqref{eq:M11-boundary-exchange-a}--\eqref{eq:M11-boundary-exchange-b}, shows that the boundary column obeys the same equation in distributions in time. Since the point column obeys it as well, $\gamma_j$ is locally absolutely continuous and coefficient extraction gives
\begin{equation}\label{eq:M11-gamma-before-commutator}
0=(\partial_t\gamma_j)\varphi_j+[(\gamma_j+x),\mathbf{B}_u]\varphi_j.
\end{equation}
Distributionally,
\[
[x,\mathbf{B}_u]\varphi_j=-2i\partial_x\varphi_j-2u\Pi_+(\overline u\varphi_j)=2L_u\varphi_j=2\mu_j\varphi_j.
\]
Substitution into \eqref{eq:M11-gamma-before-commutator} proves \eqref{eq:M11-gamma-cocycle}. The concrete persistent class follows from \cite[Theorem~1.7 and Corollary~1.8]{Sun-2026-arXiv}.
\end{proof}

\begin{prop}[Primitive Grushin quotient criterion]\label{prop:M11-primitive-Grushin}
Let $\lambda_t>0$, $Y_t,\xi_t,\theta_t\in\mathbb{R}$, with $\xi_t\ge0$, and set
\[
D_t(x)=x-Y_t+i\lambda_t,\qquad A_t(x)=\sqrt{2\lambda_t}e^{i\theta_t+i\xi_t(x-Y_t)}.
\]
Suppose the point column $b_t=b_j(t)$ has the canonical scalar Grushin quotient
\[
b_t=\frac{N_t}{S_t},\qquad S_t=d_tD_t+r_t,\qquad N_t=d_tA_t+n_t,
\]
where, for some $c,d_0>0$,
\begin{equation}\label{eq:M11-primitive-denominator}
|d_t|\ge d_0, \qquad|S_t(x)|\ge c|d_tD_t(x)|,
\end{equation}
\begin{equation}\label{eq:M11-primitive-errors}
|d_t|^{-1}\|n_t/D_t\|_{L^2(\mathbb{R})}\underset{t\to\infty}\longrightarrow0,\quad\frac{\sqrt{\lambda_t}}{|d_t|}\|r_t/D_t^2\|_{L^2(\mathbb{R})}\underset{t\to\infty}\longrightarrow0.
\end{equation}
Then
\[
\|b_t-\mathcal{M}_{\lambda_t,Y_t,\xi_t,\theta_t}R\|_{L^2(\mathbb{R})}\underset{t\to\infty}{\longrightarrow}0.
\]
\end{prop}
\begin{proof}
The leading quotient is $\frac{A_t(x)}{D_t(x)}=\mathcal{M}_{\lambda_t,Y_t,\xi_t,\theta_t}R(x)$. The cancellation of the common leading product gives $\frac{N_t}{S_t}-\frac{A_t}{D_t}=\frac{n_t}{S_t}-\frac{A_tr_t}{S_tD_t}$. Using \eqref{eq:M11-primitive-denominator} and $|A_t|=\sqrt{2\lambda_t}$,
\[
\|\frac{N_t}{S_t}-\frac{A_t}{D_t}\|_{L^2(\mathbb{R})}\le \frac{1}{c|d_t|}\|n_t/D_t\|_{L^2(\mathbb{R})}+\frac{\sqrt{2\lambda_t}}{c|d_t|}\|r_t/D_t^2\|_{L^2(\mathbb{R})},
\]
which tends to zero by \eqref{eq:M11-primitive-errors}.
\end{proof}

\begin{condition}[Bounds for reconstruction by inverse blocks]\label{hyp:M11-inverse-block}
There is a fixed rigging $\mathcal{X}_+\hookrightarrow L^2_+(\mathbb{R})\hookrightarrow\mathcal{X}_-$ and a Jost/Grushin pencil with fixed incoming data for which the following hold.
\begin{enumerate}[(i)]
 \item At every point label the pencil has a simple Laurent residue. At an isolated negative label this is an ordinary Hilbert-space residue; at a nonnegative embedded label it is cutoff-independent in $\mathcal{B}(\mathcal{X}_+,\mathcal{X}_-)$. In both cases the right gauge is $r_j=b_j/\sqrt{2\pi}$ and the left derivative normalization is fixed.
 \item The left/right norming data obey a cocycle along the rough flow--Lax admissible orbit.
 \item On one fixed reconstruction space $\mathcal{Y}$, the blocks have the types
 \[
  \mathcal{G}_{c,t},\mathcal{G}_t\in\operatorname{GL}(\mathcal{Y}),\qquad \mathcal{U}_t\in\mathcal{B}(\mathbb{C}^N,\mathcal{Y}),\qquad\mathcal{V}_t\in\mathcal{B}(\mathcal{Y},\mathbb{C}^N),
 \]
  $\mathcal{N}_t,\Delta_t\in\operatorname{GL}(\mathbb{C}^N)$,  and the finite-rank identities hold:
  $\mathcal{G}_t=\mathcal{G}_{c,t}+\mathcal{U}_t\mathcal{N}_t\mathcal{V}_t$,   $\Delta_t=\mathcal{N}_t^{-1}+\mathcal{V}_t\mathcal{G}_{c,t}^{-1}\mathcal{U}_t$,  \begin{equation}\label{eq:M11-Woodbury}
 \mathcal{G}_t^{-1}=\mathcal{G}_{c,t}^{-1}-\mathcal{G}_{c,t}^{-1}\mathcal{U}_t\Delta_t^{-1}\mathcal{V}_t\mathcal{G}_{c,t}^{-1}.
 \end{equation}
 \item The inverse and coupling operators satisfy
 \[
 \sup_{t\gg1}\bigl(\|\mathcal{G}_{c,t}^{-1}\|_{\mathcal{B}(\mathcal{Y})}+\|\mathcal{U}_t\|_{\mathcal{B}(\mathbb{C}^N,\mathcal{Y})}+\|\mathcal{V}_t\|_{\mathcal{B}(\mathcal{Y},\mathbb{C}^N)}+\|\Delta_t^{-1}\|_{\mathcal{B}(\mathbb{C}^N)}\bigr)<\infty.
 \]
 \item Each discrete column has a one-pole affine $R$ model, while the continuous--discrete and distinct-discrete residuals tend to zero in the reconstruction norm before inversion.
\end{enumerate}
\end{condition}

\begin{prop}[Inverse-block rigidity]\label{prop:M11-P4}
Under Condition~\ref{hyp:M11-inverse-block}, suppose that, for every $j$, the reconstruction identity for the $j$-th column is
\begin{equation}\label{eq:M11-reconstruction-error}
b_j(t)-\mathcal{M}_{p_j(t)}R=-\mathfrak{R}_{j,t}\mathcal{G}_t^{-1}\varepsilon_{j,t},\qquad\varepsilon_{j,t}\in\mathcal{Y},\quad\|\varepsilon_{j,t}\|_{\mathcal{Y}}\longrightarrow0,
\end{equation}
where
\[
\mathfrak{R}_{j,t}\in\mathcal{B}(\mathcal{Y},L^2_+(\mathbb{R})),\qquad\sup_t\|\mathfrak{R}_{j,t}\|_{\mathcal{B}(\mathcal{Y},L^2_+(\mathbb{R}))}<\infty,
\]
\[
p_j(t)=(\lambda_j(t),y_j(t),\xi_j(t),\theta_j(t)),\quad\lambda_j(t)>0,\quad\xi_j(t)\ge0.
\]
Then $\|b_j(t)-\mathcal{M}_{p_j(t)}R\|_{L^2(\mathbb{R})}\to0$, which implies Hypothesis~\ref{hyp:P4}.
\end{prop}
\begin{proof}
Equation \eqref{eq:M11-Woodbury} and the bounds in Condition~\ref{hyp:M11-inverse-block} give $\sup_t\|\mathcal{G}_t^{-1}\|<\infty$. Apply this and the bound on $\mathfrak {R}_{j,t}$ to \eqref{eq:M11-reconstruction-error}. The conclusion holds for all sufficiently large times and hence implies Hypothesis~\ref{hyp:P4}.
\end{proof}

\begin{prop}[Four sufficient criteria for point-channel rigidity]\label{prop:phase-III-modular}
Assume that, for every sequence $t_n\to\infty$ and every point label, one of the following sets of assumptions holds, with the profile, complementary column, and residuals in the same affine frame:
\begin{enumerate}[(i)]
 \item there is one $C^1$ frame $M_j(t)$ on $[T_0,\infty)$, with $M_j(t_n)=M_n$, and
 \[
  g(t)=M_j(t)^{-1}b_j(t),\qquad w(t)=M_j(t)^{-1}(u(t)-b_j(t)),
 \]
 for which all hypotheses of Proposition~\ref{prop:M8-moving-leakage} (including \eqref{eq:M8-rough-leakage-uniform} and \eqref{eq:M8-atom-preserving-passage} when the orbit is rough), the graph bounds \eqref{eq:M9-graph-input}, and the local-smoothing and time-regularity bounds \eqref{eq:M9-local-smoothing}--\eqref{eq:M9-no-spike} for the full
 complementary column hold;
 \item the two-tail compactness assumptions and the remaining hypotheses of Proposition~\ref{prop:M9-Lax-column}, with its positive residual supplied either
 by the weighted positive-frequency convergence and the condition \eqref{eq:M9-subcarrier-gate} in Lemma~\ref{lem:M9-carrier-Hankel}, applied to the full complementary column, or by the stronger condition \eqref{eq:M9-lifted-compactness};
 \item the two-tail compactness assumptions and either Criterion~\ref{crit:M10prime-Fano}, the locally uniform shift defect \eqref{eq:M10prime-rank-one-defect} for \eqref{eq:M10prime-positive-normalization}, or the stronger generator condition \eqref{eq:M10prime-generator-certificate};
 \item all hypotheses of Proposition~\ref{prop:M11-primitive-Grushin} or Proposition~\ref{prop:M11-P4}, including Condition~\ref{hyp:M11-inverse-block} and its reconstruction identity in the latter case.
\end{enumerate}
Then Condition~\ref{hyp:P4} holds. Together with Theorem~\ref{thm:P3}, each route supplies the two asymptotic inputs used in Theorem~\ref{thm:main}.
\end{prop}
\begin{proof}
In (i), Proposition~\ref{prop:M8-moving-leakage} gives the two-tail bounds and Lemma~\ref{lem:M9-Rellich} gives local escape, so Proposition~\ref{prop:M9-moving-window} and Corollary~\ref{cor:M9-geometric-P4} apply. In (ii), Lemma~\ref{lem:M9-carrier-Hankel} gives the vanishing residual required by Proposition~\ref{prop:M9-Lax-column}; apply Corollary~\ref{cor:M9-geometric-P4}. In (iii), Criterion~\ref{crit:M10prime-Fano} and Corollary~\ref{cor:M8-M10prime-P4} give rigidity; the alternative assumptions enter the same corollary directly. In (iv), apply Proposition~\ref{prop:M11-primitive-Grushin} or \ref{prop:M11-P4}. Different labels may use different criteria. Each set of assumptions passes to subsequences, so successive extraction over the finitely many labels gives the common subsequence in Condition~\ref{hyp:P4}.
\end{proof}

\section{Parameter separation and asymptotic modulation}\label{sec:geometric assembly}
\begin{lemma}[Uniform affine weak dispersion]\label{lem:affine-weak}
For every $f,g\in L^2(\mathbb{R})$,
\begin{equation}\label{eq:affine-weak}
\lim_{|t|\to\infty}\sup_{\lambda>0,\,y\in\mathbb{R},\,\xi\ge0,\,\theta\in\mathbb{R}/2\pi\mathbb{Z}}\left|\langle \mathcal{M}_{\lambda,y,\xi,\theta}^{-1} e^{it\partial_x^2}f,g\rangle\right|=0.
\end{equation}
\end{lemma}
\begin{proof}
Write $p=(\lambda,y,\xi,\theta)$.  Since $\mathcal{M}_p$ is unitary, $\langle\mathcal{M}_p^{-1}e^{it\partial_x^2}f,g\rangle=\langle e^{it\partial_x^2}f,\mathcal{M}_pg\rangle$. Fix $\varepsilon>0$. Choose $f_0\in C_c^\infty(\mathbb{R})$ and a Schwartz function $g_0$ with $\widehat{g}_0\in C_c^\infty([-A,A])$ so that $\|f-f_0\|_{L^2(\mathbb{R})}<\varepsilon,\qquad \|g-g_0\|_{L^2(\mathbb{R})}<\varepsilon$. For every $p$ and $t$, two applications of Cauchy--Schwarz and unitarity give
\begin{equation}\label{eq:affine-density-error}
\left|\langle e^{it\partial_x^2}f,\mathcal{M}_pg\rangle-\langle e^{it\partial_x^2}f_0, \mathcal{M}_pg_0\rangle\right|\le\|f-f_0\|_{L^2(\mathbb{R})}\|g\|_{L^2(\mathbb{R})}+
\|f_0\|_{L^2(\mathbb{R})}\|g-g_0\|_{L^2(\mathbb{R})}=:\mathcal{E}_\varepsilon.
\end{equation}
The right-hand side is independent of $p,t$, and $\mathcal{E}_\varepsilon\to0$ as $\varepsilon\downarrow0$.

Fix $0<\delta<1$. If $0<\lambda\leq\delta|t|$, then the one-dimensional dispersive estimate and the scaling identity $\|\mathcal{M}_pg_0\|_{L^1(\mathbb{R})}=\lambda^{1/2}\|g_0\|_{L^1(\mathbb{R})}$ yield
\[
\begin{aligned}
|\langle e^{it\partial_x^2}f_0, \mathcal{M}_pg_0\rangle|&\le\|e^{it\partial_x^2}f_0\|_{L^\infty(\mathbb{R})}\|\mathcal{M}_pg_0\|_{L^1(\mathbb{R})}\\
&\le C|t|^{-1/2}\|f_0\|_{L^1(\mathbb{R})}\lambda^{1/2}\|g_0\|_{L^1(\mathbb{R})}
\le C\delta^{1/2}\|f_0\|_{L^1(\mathbb{R})}\|g_0\|_{L^1(\mathbb{R})}.
\end{aligned}
\]
For the complementary scales, \eqref{eq:modulation-fourier} shows that the Fourier transform is supported in $I_{\lambda,\xi}=[\xi-A/\lambda,\xi+A/\lambda]$.
If $\lambda>\delta|t|$, then
\begin{equation}\label{eq:short-frequency-window}
|I_{\lambda,\xi}|\leq\frac{2A}{\delta|t|}.
\end{equation}
Define
\[
\omega_{f_0}(s)=\sup_{\substack{E\subset\mathbb{R}\ \text{ measurable}\\ |E|\le s}}\left(\frac{1}{2\pi}\int_E|\widehat{f}_0(\zeta)|^2d\zeta\right)^{1/2}.
\]
Since $|\widehat{f}_0|^2\in L^1(\mathbb{R})$, absolute continuity of the Lebesgue integral gives $\omega_{f_0}(s)\to0$ as $s\to0$. Plancherel, Cauchy--Schwarz, and \eqref{eq:short-frequency-window} give
\[
|\langle e^{it\partial_x^2}f_0, \mathcal{M}_pg_0\rangle|\le\left(\frac{1}{2\pi}\int_{I_{\lambda,\xi}}|\widehat{f}_0(\zeta)|^2d\zeta\right)^{1/2}\|\mathcal{M}_pg_0\|_{L^2(\mathbb{R})}\le\omega_{f_0}\left(\frac{2A}{\delta|t|}\right)
\|g_0\|_{L^2(\mathbb{R})}.
\]

Combining the two scale ranges with \eqref{eq:affine-density-error} yields
\[
\limsup_{|t|\to\infty}\sup_p\left|\langle \mathcal{M}_p^{-1}e^{it\partial_x^2}f,g\rangle\right|\le\mathcal{E}_\varepsilon+C\delta^{1/2}\|f_0\|_{L^1(\mathbb{R})}\|g_0\|_{L^1(\mathbb{R})}.
\]
Letting $\delta\to0$ and then $\varepsilon\to0$ proves \eqref{eq:affine-weak}, uniformly for $\xi\in\mathbb{R}$.
\end{proof}

\begin{lemma}[Soliton overlap]\label{lem:overlap}
Let $S_\ell=\mathcal{M}_{\lambda_\ell,y_\ell,\xi_\ell,\theta_\ell}R$, $\ell=1,2$. Then
\begin{equation}\label{eq:overlap}
|\langle S_1,S_2\rangle|=\frac{4\pi\sqrt{\lambda_1\lambda_2}}{\sqrt{(\lambda_1+\lambda_2)^2+(y_1-y_2)^2}}\exp\left[-\lambda_1(\xi_2-\xi_1)_+-\lambda_2(\xi_1-\xi_2)_+\right].
\end{equation}
\end{lemma}
\begin{proof}
For $\zeta>0$, close the contour in the lower half-plane in $\widehat{R}(\zeta)=\sqrt{2}\int_\mathbb{R}\frac{e^{-ix\zeta}}{x+i}dx$. The clockwise contour contains the pole $-i$, and the residue theorem gives $-2\sqrt{2}\pi ie^{-\zeta}$. For $\zeta<0$, closing in the upper half-plane gives zero. Hence, up to the value at $\zeta=0$, $\widehat{R}(\zeta)=c_Re^{-\zeta}\mathbf{1}_{\mathbb{R}_+}(\zeta),\qquad c_R=-2\sqrt{2}\pi i$. The normalization is $\frac{1}{2\pi}\int_0^\infty|c_R|^2e^{-2\zeta}d\zeta=2\pi$.
By \eqref{eq:modulation-fourier},
\begin{equation}\label{eq:S-fourier-exact}
\widehat{S}_\ell(\zeta)=c_R\sqrt{\lambda_\ell}e^{i\theta_\ell-i y_\ell\zeta}e^{-\lambda_\ell(\zeta-\xi_\ell)}\mathbf{1}_{[\xi_\ell,\infty)}(\zeta).
\end{equation}
Put $m=\max\{\xi_1,\xi_2\}$ and $\Delta y=y_1-y_2$. Plancherel gives
\[
\begin{aligned}
\langle S_1,S_2\rangle&=4\pi\sqrt{\lambda_1\lambda_2}e^{i(\theta_1-\theta_2)}e^{\lambda_1\xi_1+\lambda_2\xi_2}\int_m^\infty e^{-(\lambda_1+\lambda_2+i\Delta y)\zeta}d\zeta\\
&=4\pi\sqrt{\lambda_1\lambda_2}e^{i(\theta_1-\theta_2)}e^{\lambda_1\xi_1+\lambda_2\xi_2}\frac{e^{-(\lambda_1+\lambda_2+i\Delta y)m}}{\lambda_1+\lambda_2+i\Delta y}.
\end{aligned}
\]
Taking absolute values and using
\[
\lambda_1\xi_1+\lambda_2\xi_2-(\lambda_1+\lambda_2)m=-\lambda_1(\xi_2-\xi_1)_+-\lambda_2(\xi_1-\xi_2)_+
\]
proves \eqref{eq:overlap}.
\end{proof}

\begin{lemma}[Parameter separation]\label{lem:parameter-separation}
Suppose $b_j(t)\perp b_k(t)$ and
\[
\|b_\ell(t)-S_\ell(t)\|_{L^2(\mathbb{R})}\underset{t\to\infty}{\longrightarrow}0,\qquad \ell\in\{j,k\},
\]
where each $S_\ell(t)$ is a modulated copy of $R$. Then the separation expression in \eqref{eq:separation} tends to infinity.
\end{lemma}
\begin{proof}
Put $e_\ell(t)=\|b_\ell(t)-S_\ell(t)\|_{L^2(\mathbb{R})}$ for $\ell\in\{j,k\}$.
Orthogonality gives $\langle S_j, S_k\rangle=\langle S_j-b_j, S_k\rangle+\langle b_j,S_k-b_k\rangle$. Since $\|S_\ell\|_{L^2(\mathbb{R})}=\sqrt{2\pi}$ and $\|b_j\|_{L^2(\mathbb{R})}\le\sqrt{2\pi}+e_j$, Cauchy--Schwarz gives
\begin{equation}\label{eq:orthogonal-overlap}
|\langle S_j(t), S_k(t)\rangle|\le\sqrt{2\pi}\bigl(e_j(t)+e_k(t)\bigr)+e_j(t)e_k(t)\underset{t\to\infty}{\longrightarrow}0.
\end{equation}
Let $Q_{jk}(t)$ denote the expression in \eqref{eq:separation}. If $Q_{jk}(t)$ remains bounded by $C$ along a sequence $t_n\to\infty$, then, on this sequence,
\[
r_n=\frac{\lambda_j}{\lambda_k},\qquad Y_n=\frac{|y_j-y_k|}{\sqrt{\lambda_j\lambda_k}},\qquad Z_n=\sqrt{\lambda_j\lambda_k}|\xi_j-\xi_k|
\]
satisfy $r_n+r_n^{-1}\le C$, $Y_n^2\le C$, and $Z_n^2\le C$. The algebraic factor in \eqref{eq:overlap} becomes
\[
\frac{4\pi}{\sqrt{r_n+r_n^{-1}+2+Y_n^2}}\ge\frac{4\pi}{\sqrt{2C+2}}.
\]
If $\xi_k\ge\xi_j$, its exponential exponent is $\lambda_j|\xi_j-\xi_k|=\sqrt{r_n}Z_n$; if $\xi_j\geq\xi_k$, it is $\lambda_k|\xi_j-\xi_k|=r_n^{-1/2}Z_n$. Both are bounded in terms of $C$. Formula \eqref{eq:overlap} therefore gives a positive lower bound determined by $C$, contradicting \eqref{eq:orthogonal-overlap}. Thus $Q_{jk}(t)\to\infty$.
\end{proof}

\begin{prop}[Asymptotic modulation parameters]\label{prop:sequential-upgrade}
Under Condition~\ref{hyp:P4}, there exist Borel measurable parameters in the parameter domain stated in Theorem~\ref{thm:main} satisfying
\begin{equation}\label{eq:full-channel-approx}
\|b_j(t)-\mathcal{M}_{\lambda_j(t),y_j(t),\xi_j(t),\theta_j(t)}R\|_{L^2(\mathbb{R})}\underset{t\to\infty}{\longrightarrow}0
\end{equation}
for every $j$.
\end{prop}
\begin{proof}
For each channel define
\[
d_j(t)=\inf_{\substack{\lambda>0,\,y\in\mathbb{R}\\
                         \xi\ge0,\,\theta\in\mathbb{R}/2\pi\mathbb{Z}}}
\|b_j(t)-\mathcal{M}_{\lambda,y,\xi,\theta}R\|_{L^2(\mathbb{R})}.
\]
If $d_j(t)$ fails to converge to zero, then some $\varepsilon>0$ and sequence $t_n\to\infty$ satisfy $d_j(t_n)\geq\varepsilon$. The common subsequence in Condition~\ref{hyp:P4} supplies parameters $p_{j,k}$ such that
\[
0\le d_j(t_{n_k})\le\|b_j(t_{n_k})-\mathcal{M}_{p_{j,k}}R\|_{L^2(\mathbb{R})}\underset{k\to\infty}{\longrightarrow}0.
\]
This contradicts the lower bound, so $d_j(t)\to0$ as $t\to\infty$.

Let
$\mathcal{P}=\mathbb{R}_+\times\mathbb{R}\times[0,\infty)\times\mathbb{R}/2\pi\mathbb{Z}$, endowed with its product metric, and choose a countable dense set
$\{p_m\}_{m\ge1}\subset\mathcal{P}$. Translation, modulation, and unitary dilation are strongly continuous on $L^2$; this follows on $C_c^\infty$ by dominated convergence and on $L^2$ by density and unitarity. Consequently $p\mapsto\mathcal{M}_pR$ is continuous and
\begin{equation}\label{eq:countable-distance}
d_j(t)=\inf_{m\ge1}\|b_j(t)-\mathcal{M}_{p_m}R\|_{L^2(\mathbb{R})}.
\end{equation}
The strong continuity of \(b_j\), proved in Theorem~\ref{thm:channels}, also gives $|d_j(t)-d_j(s)|\le\|b_j(t)-b_j(s)\|_{L^2(\mathbb{R})}$, so $d_j$ is continuous.

For $t\ge1$, set
\[
A_{j,m}=\left\{t:\|b_j(t)-\mathcal{M}_{p_m}R\|_{L^2(\mathbb{R})}<d_j(t)+t^{-1}\right\}.
\]
These sets are Borel and cover $[1,\infty)$ by \eqref{eq:countable-distance}. Define $m(j,t)=\min\{m:t\in A_{j,m}\}$.
The identity $\{t:m(j,t)=m\}=A_{j,m}\setminus\bigcup_{\ell<m}A_{j,\ell}$ shows that $m(j,\cdot)$ is Borel measurable. Use the coordinates of $p_{m(j,t)}$ for $t\ge1$, and use the coordinates of $p_1$ for
$0\le t<1$. The resulting parameter functions satisfy
\[
\|b_j(t)-\mathcal{M}_{p_{m(j,t)}}R\|_{L^2(\mathbb{R})}<d_j(t)+t^{-1}\underset{t\to\infty}{\longrightarrow}0,
\]
which proves \eqref{eq:full-channel-approx}.
\end{proof}

\section{Proof of the global branch of Theorem~\ref{thm:main}}\label{sec:7}
The zero extension of $\beta_{u_0}$ belongs to $L^2(\mathbb{R})$ and is supported in $[0,\infty)$. Unitarity of $\mathcal{F}_0$ defines a unique $u_+\in L^2_+(\mathbb{R})$ by $\mathcal{F}_0u_+=-i\beta_{u_0}/\sqrt{2\pi}$.
Thus $\widehat{u}_+=-i\beta_{u_0}$, as in \eqref{eq:radiation-state}.

Theorem~\ref{thm:channels} gives, for every $t\ge0$, the orthogonal spectral decomposition $u(t)=\sum_{j=1}^Nb_j(t)+r_{\mathrm{ac}}(t)$. Theorem~\ref{thm:P3} gives $r_{\mathrm{ac}}(t)=e^{it\partial_x^2}u_++o_{L^2}(1)$.
Proposition~\ref{prop:sequential-upgrade} supplies Borel measurable parameter paths. Set $S_j(t)=\mathcal{M}_{\lambda_j(t),y_j(t),\xi_j(t),\theta_j(t)}R$.
For every point channel it gives
\begin{equation}\label{eq:main-points}
b_j(t)=S_j(t)+o_{L^2}(1).
\end{equation}
Since $N$ is finite, the channel decomposition and the triangle inequality give
\[
\|u(t)-e^{it\partial_x^2}u_+-\sum_{j=1}^NS_j(t)\|_{L^2(\mathbb{R})}\le\|r_{\mathrm{ac}}(t)-e^{it\partial_x^2}u_+\|_{L^2(\mathbb{R})}+\sum_{j=1}^N\|b_j(t)-S_j(t)\|_{L^2(\mathbb{R})}
\underset{t\to\infty}\longrightarrow0.
\]
This proves \eqref{eq:resolution}. Formula \eqref{eq:S-fourier-exact} and $\xi_j(t)\ge0$ give $S_j(t)\in L^2_+(\mathbb{R})$.

For $j\neq k$, orthogonality of $b_j(t)$ and $b_k(t)$, \eqref{eq:main-points}, and Lemma~\ref{lem:parameter-separation} imply \eqref{eq:separation}.

Fix $j$ and $g\in L^2(\mathbb{R})$. The uniformity in Lemma~\ref{lem:affine-weak} gives
\[
\left|\langle \mathcal{M}_{\lambda_j(t),y_j(t),\xi_j(t),\theta_j(t)}^{-1}e^{it\partial_x^2}u_+,g\rangle\right|\le\sup_{p\in\mathcal{P}}\left|\langle\mathcal{M}_p^{-1}e^{it\partial_x^2}u_+,g\rangle\right|
\underset{t\to\infty}{\longrightarrow}0.
\]
This is \eqref{eq:frame-weak}.

The orthogonal spectral partition gives
\begin{equation}\label{eq:terminal-mass}
\|u_0\|_{L^2(\mathbb{R})}^2=\|u(t)\|_{L^2(\mathbb{R})}^2=\sum_{j=1}^N\|b_j(t)\|_{L^2(\mathbb{R})}^2+\|r_{\mathrm{ac}}(t)\|_{L^2(\mathbb{R})}^2=2\pi N+\frac{1}{2\pi}\|\beta_{u_0}\|_{L^2(\mathbb{R})}^2.
\end{equation}
Plancherel and \eqref{eq:radiation-state} give $\|u_+\|_{L^2(\mathbb{R})}^2=\frac{1}{2\pi}\|\beta_{u_0}\|_{L^2(\mathbb{R}_+)}^2$.
Together with \eqref{eq:terminal-mass}, this proves \eqref{eq:mass-resolution}. The labels in Theorem~\ref{thm:channels} enumerate the entire point spectrum, so $N=N_{\mathrm{pp}}(L_{u_0})$.

\section{Large-mass well-posedness and maximal lifespan}\label{sec:phase-IV-flow}
To construct a flow--Lax admissible orbit by strong $L^2$ closure, let
\begin{equation}\label{eq:M0-KLV-core}
\mathcal{D}_{\mathrm{KLV}}=\{a\in H_+^\infty(\mathbb{R}):\langle x\rangle a\in L^2(\mathbb{R})\}.
\end{equation}
For $a\in\mathcal{D}_{\mathrm{KLV}}$, write $u_a^{\mathrm{cl}}:I_a^{\mathrm{cl}}(\mathbb{R})\to H_+^\infty(\mathbb{R})$ for its maximal classical orbit. The set \eqref{eq:M0-KLV-core} is a dense classical approximation core in $L^2_+$. For a bounded family $\mathcal{Q}\subset L^2_+(\mathbb{R})$, define its positive frequency tail by
\[
\mathfrak{h}_K(\mathcal{Q})=\sup_{f\in\mathcal{Q}}\|\mathbf{1}_{[K,\infty)}\mathcal{F}_0f\|_{L^2_\xi}^2.
\]
The family $\mathcal{Q}$ is Fourier equicontinuous if and only if $\mathfrak{h}_K(\mathcal{Q})\to0$.

\begin{prop}[KLV compactness]\label{prop:M0-KLV-engine}
Let $a_n\in\mathcal{D}_{\mathrm{KLV}}$, let $a_n\to a$ in $L^2(\mathbb{R})$, and suppose that the classical solutions are defined on one
compact interval $J\ni0$. If
\begin{equation}\label{eq:M0-engine-tail}
\sup_n\sup_{t\in J}\|\mathbf{1}_{[K,\infty)}\mathcal{F}_0u_{a_n}^{\mathrm{cl}}(t)\|_{L^2(\mathbb{R})}\underset{K\to\infty}{\longrightarrow}0,
\end{equation}
then $u_{a_n}^{\mathrm{cl}}$ converges in $C(J;L^2_+(\mathbb{R}))$. Its limit is independent among approximation sequences which satisfy the same common-lifespan and uniform-tail hypotheses, and is identified by the KLV boundary formula
\begin{equation}\label{eq:M0-engine-formula}
u(t,z)=\frac{1}{2\pi i}I_+\bigl([\mathsf{X}+2tL_a-z]_{\mathrm{KLV}}^{-1}a\bigr),\qquad \operatorname{Im}z>0.
\end{equation}
The same conclusion holds uniformly for a precompact family of initial data when \eqref{eq:M0-engine-tail} is uniform over that family.
\end{prop}
\begin{proof}
For smooth well-decaying data, the KLV formula holds throughout the classical lifespan. Operator-norm continuity of the KLV resolvent is \cite[Proposition~4.10]{Killip-2025-CAMS}, while \cite[Proposition~5.1]{Killip-2025-CAMS} controls the unbounded boundary row $I_+$ and identifies every pointwise weak cluster limit in terms of $a$ alone. Under \eqref{eq:M0-engine-tail}, the proofs of \cite[Lemma~5.4 and Propositions~5.5--5.6]{Killip-2025-CAMS} supply uniform time equicontinuity and physical-space tightness. Kolmogorov--Riesz and Arzel\`a--Ascoli therefore upgrade weak convergence to convergence in $C(J;L^2(\mathbb{R}))$. If two sequences satisfy the stated hypotheses, their interleaving does as well, proving independence in this class. On the common lifespan $J$, the cited estimates are independent of a numerical mass threshold.
\end{proof}

\begin{cor}[Focusing equicontinuity threshold]\label{cor:M0-exact-threshold}
The constant $M^\ast$ in \cite[Definition~1.5]{Killip-2025-CAMS} equals $2\pi$.
\end{cor}
\begin{proof}
The lower bound $M^\ast\ge2\pi$ is \cite[Theorem~1.4]{Killip-2025-CAMS}. If $M^\ast>2\pi$, choose $0<\varepsilon<M^\ast-2\pi$. By \cite[Theorem~1.1]{Kim-2026-MAMS}, there is a Schwartz chiral datum with mass in $(2\pi,2\pi+\varepsilon)$ whose classical solution blows up at a finite time. The singleton containing that datum satisfies the initial-set conditions in the definition of $M^\ast$. Fourier equicontinuity up to the endpoint would trigger the smooth continuation criterion \cite[Lemma~2.4]{Killip-2025-CAMS}, contradicting finite-time blow-up. Hence $M^\ast\le2\pi$. The identity $M^\ast=2\pi$ is also recorded in \cite[(1.4)]{Killip-2026-arXiv}.
\end{proof}

The local construction at arbitrary mass requires the following condition.
\begin{condition}[Local Fourier compactness]\label{hyp:M0-local-compactness}
For every $L^2$-precompact set $\mathcal{Q}\subset\mathcal{D}_{\mathrm{KLV}}$, there is $\tau(\mathcal{Q})>0$ such that $[-\tau(\mathcal{Q}),\tau(\mathcal{Q})]\subset I_a^{\mathrm{cl}}$ for every $a\in\mathcal{Q}$, and
\begin{equation}\label{eq:M0-local-tail}
\lim_{K\to\infty}\sup_{a\in\mathcal{Q}}\sup_{|t|\le\tau(\mathcal{Q})}\|\mathbf{1}_{[K,\infty)}\mathcal{F}_0u_a^{\mathrm{cl}}(t)\|_{L^2(\mathbb{R})}=0.
\end{equation}
Equivalently, whenever $a_n\in\mathcal{D}_{\mathrm{KLV}}$ converges in $L^2$ to an arbitrary $a\in L^2_+(\mathbb{R})$, there are $\tau>0$ and $n_0$ for which the tail $\{a_n:n\geq n_0\}$ has a common classical lifespan $[-\tau,\tau]$ and satisfies the corresponding uniform version of \eqref{eq:M0-local-tail}.
\end{condition}

\begin{thm}[Maximal strong $L^2_+$ closure]\label{thm:M0-maximal-closure}
Assume Condition~\ref{hyp:M0-local-compactness}. Then the classical flow has a unique maximal strong closure
\begin{equation}\label{eq:M0-maximal-flow}
 \mathsf S(\mathord\cdot)a\in C(I(a);\Hplus),\qquad
 I(a)=(-T_-(a),T_+(a)),\qquad 0<T_\pm(a)\leq\infty,
\end{equation}
with the following properties.
\begin{enumerate}[(i)]
 \item The local flow law holds, $\mathsf{S}(0)a=a$, and
 \begin{equation}\label{eq:M0-mass-Hardy}
  \|\mathsf{S}(t)a\|_{L^2(\mathbb{R})}=\|a\|_{L^2(\mathbb{R})},\qquad \mathsf{S}(t)a\in L^2_+(\mathbb{R}).
 \end{equation}
 If $a\in\mathcal{D}_{\mathrm{KLV}}$, this flow agrees with $u_a^{\mathrm{cl}}$ throughout the common classical lifespan.
 \item If $J\Subset I(a)$, there are $a_n\in\mathcal{D}_{\mathrm{KLV}}$ whose classical solutions are defined on $J$ and satisfy
 \begin{equation}\label{eq:M0-flow-Lax-approximation}
  \sup_{t\in J}\|u_{a_n}^{\mathrm{cl}}(t)-\mathsf{S}(t)a\|_{L^2(\mathbb{R})}\underset{n\to\infty}{\longrightarrow}0.
 \end{equation}
 In particular, a forward global branch is flow--Lax admissible in the sense of Definition~\ref{def:flow-admissible}.
 \item The lifespan is lower semicontinuous and the flow is continuously dependent on its datum on common compact lifespan intervals. Uniqueness is taken in the strong KLV-closure class \eqref{eq:M0-flow-Lax-approximation}.
 \item Formula \eqref{eq:M0-engine-formula} holds on $I(a)$. On every compact subinterval, the local form of Theorem~\ref{thm:P0min-import} applies: the cyclic measure is conserved and all cyclic intertwining and boundary identities are valid.
 \item If $T_+(a)<\infty$, then the endpoint has a nonzero critical cascade defect:
 \begin{equation}\label{eq:M0-blowup-criterion}
  \mathfrak{c}_+(a):=\lim_{K\to\infty}\lim_{\tau\to T_+(a)}\sup_{\tau<t<T_+(a)}\|\mathbf{1}_{[K,\infty)}\mathcal{F}_0\mathsf{S}(t)a\|_{L^2(\mathbb{R})}^2>0.
 \end{equation}
 An analogous statement holds at $-T_-(a)$.
\end{enumerate}
\end{thm}
\begin{proof}
Fix $a\in L^2_+(\mathbb{R})$ and choose a core sequence converging to $a$. For two choices use their interleaving; the sequential form of Condition~\ref{hyp:M0-local-compactness} and Proposition~\ref{prop:M0-KLV-engine} give the same local limit.  Uniformity and continuous dependence on compact sets follow from one fixed regularization scheme. Choose an orthonormal basis $(e_k)_{k\ge1}$ of $L^2_+(\mathbb{R})$ with $e_k\in\mathcal{D}_{\mathrm{KLV}}$, and let $\mathcal{R}_mf=\sum_{k=1}^m\langle f,e_k\rangle e_k$.
Then $\mathcal{R}_m\to\mathrm{Id}$ strongly and uniformly on every compact subset $\mathcal{K}\subset L^2_+(\mathbb{R})$. Consequently
\begin{equation}\label{eq:M0-precompact-core-family}
\bigcup_{m\ge1}\mathcal{R}_m\mathcal{K}\quad\hbox{is an $L^2$-precompact subset of }\mathcal{D}_{\mathrm{KLV}}.
\end{equation}
Apply Condition~\ref{hyp:M0-local-compactness} and Proposition~\ref{prop:M0-KLV-engine} to this single family. It gives a local strong flow uniformly on $\mathcal{K}$, as
well as continuous dependence and independence of the core approximation. The weighted smooth core is preserved on its classical lifespan by the local form of the estimates in
\cite[Proposition~2.3]{Killip-2025-CAMS}. At a time inside the local interval, the corresponding classical images of \eqref{eq:M0-precompact-core-family} are again precompact. Reapplying the same argument supplies the next common time slice. Finite iteration on a compact time interval gives the flow law, lower semicontinuity of lifespan, the maximal interval \eqref{eq:M0-maximal-flow}, and \eqref{eq:M0-flow-Lax-approximation}. Mass conservation and Hardy support pass through the strong limit.

Suppose $T:=T_+(a)<\infty$ but $\mathfrak{c}_+(a)=0$. Passing the classical low-frequency estimate through \eqref{eq:M0-flow-Lax-approximation} gives, for every fixed $N$,
\begin{equation}\label{eq:M0-low-frequency-time-modulus}
\|\mathbf{1}_{[0,N]}(D)[\mathsf{S}(t)a-\mathsf{S}(s)a]\|_{L^2(\mathbb{R})}\le C_{N,\|a\|_{L^2(\mathbb{R})}}|t-s|, \qquad 0\le s,t<T.
\end{equation}
This is the low-frequency part of the time-equicontinuity proof in \cite[Lemma~5.4]{Killip-2025-CAMS}. Vanishing of $\mathfrak{c}_+(a)$ makes the high-frequency tail uniformly small for $t$ sufficiently close to $T$. On the earlier compact part of the orbit the same uniform tail follows because a continuous image of a compact time interval is compact in $L^2(\mathbb{R})$. Thus the whole set $\{\mathsf{S}(t)a:0\le t<T\}$ is Fourier equicontinuous. Combining its tail bound with \eqref{eq:M0-low-frequency-time-modulus} shows that $\mathsf{S}(t)a$ is uniformly continuous on $[0,T)$, hence has a strong limit $a_\ast\in L^2_+(\mathbb{R})$ as $t\to T$.

Let $\mathsf{S}(h)a_\ast$ be the two-sided local flow furnished above. For fixed small $h<0$, continuous dependence and the old flow law give
\[
\mathsf{S}(h)a_\ast=\lim_{t\uparrow T}\mathsf{S}(h)(\mathsf{S}(t)a)=\lim_{t\uparrow T}\mathsf{S}(t+h)a=\mathsf{S}(T+h)a.
\]
Therefore the nonnegative-time half of the local orbit $s\mapsto\mathsf{S}(s)a_\ast$ sews to the old orbit and extends it beyond $T$, contradicting maximality. This proves \eqref{eq:M0-blowup-criterion}. Since $\|\mathsf{S}(t)a\|_{L^2(\mathbb{R})}$ is constant, finite-time blow-up means loss of strong $L^2$ continuation through a high-frequency cascade.
\end{proof}

\begin{prop}[Subthreshold $L^2_+$ flow]\label{prop:M0-known-range}
For every $a\in L^2_+(\mathbb{R})$ with $\|a\|_{L^2(\mathbb{R})}^2<2\pi$, the focusing equation has a global, jointly continuous $L^2_+$ flow obtained as the strong $C_tL_x^2$ limit of KLV-core solutions. It satisfies \eqref{eq:M0-mass-Hardy}, \eqref{eq:M0-engine-formula}, and the local conclusions.
\end{prop}
\begin{proof}
This is \cite[Theorems~1.6 and~1.7 and Section~5]{Killip-2025-CAMS}; see also the independent Hamiltonian construction \cite[Theorem~1.3]{Killip-2026-arXiv}.
\end{proof}
\begin{re}\label{rem:M0-large-mass-obstruction}
For the mixed regime, \eqref{eq:FR-measure} and \eqref{eq:cyclic-total-mass} give $\|u_0\|_{L^2(\mathbb{R})}^2=2\pi N+\frac{1}{2\pi}\|\beta_{u_0}\|_{L^2(\mathbb{R}_+)}^2>2\pi$. Thus the mixed regime lies above the verified strict subthreshold range. For every $\varepsilon>0$, \cite[Theorem~1.1]{Kim-2026-MAMS} constructs a datum
$a\in\mathcal{S}(\mathbb{R})\cap L^2_+(\mathbb{R})$ with $2\pi<\|a\|_{L^2(\mathbb{R})}^2<2\pi+\varepsilon$ whose classical chiral solution blows up in finite time. The general large-mass target is therefore a maximal-lifespan theorem containing the cascade branch \eqref{eq:M0-blowup-criterion}.

The rough-potential construction \cite[Theorem~1.1]{Hadama-2026-arXiv} supplies an independent unweighted $L^2$ solution concept and global flow for sufficiently small data, and \cite[Theorem~1.3]{Hadama-2026-arXiv} gives scattering in its stated Strichartz/density class. The two branches of \cite[Theorem~1.1]{Kim-2026-JEMS} cover the stated $H^1$ and $H^{1,1}$ regimes. For general mass, Theorem~\ref{thm:M0-maximal-closure} constructs the unweighted KLV closure under Condition~\ref{hyp:M0-local-compactness}.
\end{re}

\section{Proof of the finite-time branch of Theorem~\ref{thm:main}}\label{sec:blowup-resolution}
Let $\mathsf S(\cdot)u_0$ be the maximal flow from Theorem~\ref{thm:M0-maximal-closure}. We prove the positive-time statement under the following endpoint KLV assumptions; time reversal gives the negative-time statement.
\begin{condition}[Endpoint KLV--Ward identities] \label{hyp:blowup-KLV-Ward}
The following hold.
\begin{enumerate}[(i)]
  \item For every $v\in L^2_+(\mathbb{R})$, $L_v=D-v\Pi_+(\overline{v}\,\cdot)$ has a lower-semibounded self-adjoint realization on $H^{1/2}_+(\mathbb{R})$. Strong $L^2$ convergence of the potential implies norm-resolvent convergence. Moreover, zero-carrier affine changes obey
   $\mathcal{M}_{\lambda,y,0,0}^*L_v\mathcal{M}_{\lambda,y,0,0}=\lambda^{-1}L_{\mathcal{M}_{\lambda,y,0,0}^\ast v}$,   and, on a smooth common core,
   $[\mathsf{X},L_v]=i\left(I-\frac{1}{2\pi}|v\rangle\langle v|\right)$,   stable under smooth $L^2$ approximation in the Yosida-commutator sense.

  \item For every finite $t\in\mathbb{R}$, let
   $\mathsf{A}_t^{u_0}:=[\mathsf{X}+2tL_{u_0}]_{\mathrm{KLV}}$.   This is a closed maximal dissipative KLV realization, and its resolvent and boundary row are locally uniformly continuous in $t$. The associated semigroups and their adjoints converge strongly on compact semigroup-time intervals. With $W_t^{u_0}$ as in \eqref{eq:boundary-transform}, the horizontal Green identity
  \begin{equation}\label{eq:blowup-horizontal-Green}
  \int_{\mathbb{R}}|[W_t^{u_0}f](x+iy)|^2dx+\frac y\pi\int_{\mathbb{R}}\| (\mathsf{A}_t^{u_0}-x-iy)^{-1}f\|_2^2dx=\|f\|_2^2
  \end{equation}
  holds for $y>0$ and $f\in L^2_+$.

  \item Strictly inside the lifespan,
  \begin{equation}\label{eq:blowup-interior-boundary-flow}
  W_t^{u_0}u_0=\mathsf{S}(t)u_0,
  \end{equation}
  and the resolvent intertwinings
  \[
  W_t^{u_0}(\mathsf{A}_t^{u_0}-\zeta)^{-1}=(\mathsf{X}-\zeta)^{-1}W_t^{u_0},\qquad \zeta\in\mathbb{C}_+,
  \]
  hold, as does the common-domain inclusion
  \begin{equation}\label{eq:blowup-common-domain}
  \operatorname {Dom}(\mathsf X)\cap\operatorname {Dom}(L_{u_0})\subset\operatorname {Dom}(\mathsf A_t^{u_0}),\qquad\mathsf A_t^{u_0}f=\mathsf Xf+2tL_{u_0}f.
  \end{equation}
\end{enumerate}
\end{condition}

Throughout this section, set $T=T_+(u_0)<\infty$ and $u(\tau)=\mathsf{S}(\tau)u_0$ for $0\le\tau<T$.

\subsection{Weak endpoint and mass loss}\label{subsec:blowup-endpoint}
Define
\[
\mathfrak{c}_+(u_0)=\lim_{K\to\infty}\lim_{r\to T}\sup_{r<\tau<T}\|P_{>K}u(\tau)\|_{L^2(\mathbb{R})}^2.
\]
The inner limit exists because the terminal envelope decreases in $r$; the resulting quantity decreases in $K$.

\begin{lemma}[Low-frequency time modulus]\label{lem:blowup-low-frequency-modulus}
For $0\leq s,\tau<T$ and $K>0$,
\begin{equation}\label{eq:blowup-low-frequency-modulus}
\|P_{\leq K}(u(\tau)-u(s))\|_{L^2(\mathbb{R})}\leq K^2\|u_0\|_{L^2(\mathbb{R})}\left(1+\frac{\|u_0\|_{L^2(\mathbb{R})}^2}{2\pi}\right)|\tau-s|.
\end{equation}
\end{lemma}
\begin{proof}
For a classical chiral solution $v$, set $A_v(\eta)=\int_0^\infty\widehat{v}(\alpha+\eta) \overline{\widehat v(\alpha)}d\alpha$. Plancherel and Cauchy--Schwarz give $|A_v(\eta)|\leq2\pi\|v\|_{L^2(\mathbb{R})}^2$. If $\mathcal{N}(v)=(D+|D|)(|v|^2)v$, then, for $\xi\geq0$, $\widehat{\mathcal{N}(v)}(\xi)=\frac{2}{(2\pi)^2}\int_0^\xi\eta A_v(\eta)\widehat{v}(\xi-\eta)d\eta$. Young's inequality therefore yields
\[
\|P_{\leq K}\mathcal{N}(v)\|_{L^2(\mathbb{R})}\leq\frac{K^2}{2\pi}\|v\|_{L^2(\mathbb{R})}^3,\qquad\|P_{\leq K}D^2v\|_{L^2(\mathbb{R})}\leq K^2\|v\|_{L^2(\mathbb{R})}.
\]
Integrating the equation in time proves \eqref{eq:blowup-low-frequency-modulus} for smooth solutions. The compact-time approximation in Theorem~\ref{thm:M0-maximal-closure} passes the estimate to $u$.
\end{proof}

\begin{thm}[Weak endpoint and mass loss]\label{thm:blowup-weak-endpoint}
There is a unique $z_T\in L^2_+(\mathbb R)$ such that
\[
u(\tau)\rightharpoonup z_T,\qquad P_{\leq K}u(\tau)\longrightarrow P_{\leq K}z_T\quad\text{strongly in }L^2(\mathbb{R})
\]
for every fixed $K$. In fact,
\begin{equation}\label{eq:blowup-low-endpoint-rate}
\|P_{\leq K}(u(\tau)-z_T)\|_{L^2(\mathbb{R})}\leq C_{K,u_0}(T-\tau).
\end{equation}
Moreover,
\begin{equation}\label{eq:blowup-defect-loss}
\mathfrak{c}_+(u_0)=\|u_0\|_{L^2(\mathbb{R})}^2-\|z_T\|_{L^2(\mathbb{R})}^2>0.
\end{equation}
\end{thm}
\begin{proof}
For each $K$, Lemma~\ref{lem:blowup-low-frequency-modulus} makes $P_{\leq K}u(\tau)$ Cauchy as $\tau\to T$; denote its limit by $z_K$. If $K_1<K_2$, then $P_{\leq K_1}z_{K_2}=z_{K_1}$. For any
$K_m\to\infty$, orthogonality gives
\[
\|z_{K_m}-z_{K_n}\|_{L^2(\mathbb{R})}^2=\|z_{K_m}\|_{L^2(\mathbb{R})}^2-\|z_{K_n}\|_{L^2(\mathbb{R})}^2\qquad(m>n).
\]
Thus $z_{K_m}$ converges strongly to a vector $z_T$ independent of the chosen sequence, and $P_{\leq K}z_T=z_K$. Letting one time tend to $T$ in \eqref{eq:blowup-low-frequency-modulus} proves
\eqref{eq:blowup-low-endpoint-rate}.

For $h\in L^2(\mathbb{R})$, first truncate $h$ to $P_{\leq K}h$ and then send $\tau\to T$; the uniform mass bound and $P_{>K}h\to0$ prove weak convergence. Conservation of mass and orthogonality give
\[
\|P_{>K}u(\tau)\|_{L^2(\mathbb{R})}^2=\|u_0\|_{L^2(\mathbb{R})}^2-\|P_{\leq K}u(\tau)\|_{L^2(\mathbb{R})}^2.
\]
Taking the terminal limit and then $K\to\infty$ proves the equality in \eqref{eq:blowup-defect-loss}. Its strict positivity follows from item~(v) of Theorem~\ref{thm:M0-maximal-closure}.
\end{proof}

\subsection{The endpoint Wold defect and its quantization}\label{subsec:blowup-Wold}
For finite $t$, abbreviate $\mathsf{A}_t=\mathsf{A}_t^{u_0}$ and set
\begin{equation}\label{eq:blowup-Wold-projection}
V_t(s)=e^{-is\mathsf{A}_t}\quad(s\geq0),\qquad\mathsf{P}_t^{\mathrm{W}}=\operatorname*{s-lim}_{S\to\infty}V_t(S)^\ast V_t(S).
\end{equation}

\begin{lemma}[Green--Wold identity]
\label{lem:blowup-Green-Wold}
For $s\geq0$, $V_t(s)V_t(s)^\ast=I$, the limit in \eqref{eq:blowup-Wold-projection} is an orthogonal projection, and
\begin{equation}\label{eq:blowup-Wold-ledger}
(W_t^{u_0})^\ast W_t^{u_0}=I-\mathsf{P}_t^{\mathrm{W}}.
\end{equation}
\end{lemma}
\begin{proof}
Choose smooth bounded $u_{0,m}\to u_0$ in $L^2_+(\mathbb{R})$ and put
\[
\mathsf{A}_{t,0}=e^{-it\partial_x^2}\mathsf{X}e^{it\partial_x^2},\qquad\mathsf{Q}_mf=u_{0,m}\Pi_+(\overline{u_{0,m}}f).
\]
Then $\mathsf{A}_{t,m}=\mathsf{A}_{t,0}-2t\mathsf{Q}_m,\qquad\mathsf{A}_{t,m}^\ast=\mathsf{A}_{t,0}^\ast-2t\mathsf{Q}_m$. On the Fourier half-line,
\[
\operatorname {Dom}(\mathsf{A}_{t,0}^\ast)=e^{-it\partial_x^2}H_0^1(0,\infty),\qquad\mathsf{A}_{t,0}^\ast=e^{-it\partial_x^2}(i\partial_\xi)e^{it\partial_x^2},
\]
so both $\mathsf{A}_{t,0}^\ast$ and $\mathsf{A}_{t,m}^\ast$ are symmetric. Consequently $V_{t,m}(s)^\ast$ is isometric and $V_{t,m}(s)V_{t,m}(s)^\ast=I$.
Norm-resolvent convergence and Trotter--Kato give strong convergence of both semigroups, uniformly for $s$ in compact intervals, so the identity passes to the rough limit. Hence $V_t(s)^\ast V_t(s)$ are decreasing range projections and their strong limit is the projection onto the unitary Wold component.

For $y>0$, the semigroup-resolvent formula and vector-valued Plancherel give
\[
\frac{y}{\pi}\int_{\mathbb{R}}\|(\mathsf{A}_t-x-iy)^{-1}f\|_{L^2(\mathbb{R}_+)}^2dx=2y\int_0^\infty e^{-2ys}\|V_t(s)f\|_{L^2(\mathbb{R}_+)}^2ds.
\]
The last expression tends to $\|\mathsf{P}_t^{\mathrm{W}}f\|_{L^2(\mathbb{R}_+)}^2$ by Abel's theorem. Letting $y\to0$ in \eqref{eq:blowup-horizontal-Green} and using the $H^2$ boundary theorem yields $\|W_t^{u_0}f\|_{L^2(\mathbb{R}_+)}^2=\|f\|_{L^2(\mathbb{R}_+)}^2-\|\mathsf{P}_t^{\mathrm{W}}f\|_{L^2(\mathbb{R}_+)}^2$. Polarization proves \eqref{eq:blowup-Wold-ledger}.
\end{proof}

\begin{lemma}[Identification of the endpoint row]\label{lem:blowup-endpoint-row}
One has
\begin{equation}\label{eq:blowup-endpoint-row}
W_T^{u_0}u_0=z_T,\qquad\|\mathsf{P}_T^{\mathrm{W}}u_0\|_{L^2(\mathbb{R})}^2=\mathfrak{c}_+(u_0).
\end{equation}
\end{lemma}
\begin{proof}
At every interior time $\tau<T$, \eqref{eq:blowup-interior-boundary-flow} gives $W_{\tau}^{u_0}u_0=u(\tau)$. Point evaluation in $\mathbb{C}_+$ is bounded on $L^2_+(\mathbb{R})$, so Theorem~\ref{thm:blowup-weak-endpoint} and continuity of the scalar row give $[W_T^{u_0}u_0](z)=z_T(z)$ for every $z\in\mathbb{C}_+$. Hardy uniqueness proves the first identity. Applying \eqref{eq:blowup-Wold-ledger} to $u_0$ and using \eqref{eq:blowup-defect-loss} gives
\[
\|\mathsf {P}_T^{\mathrm{W}}u_0\|_{L^2(\mathbb{R})}^2=\|u_0\|_{L^2(\mathbb{R})}^2-\|W_T^{u_0}u_0\|_{L^2(\mathbb{R})}^2=\|u_0\|_{L^2(\mathbb{R})}^2-\|z_T\|_{L^2(\mathbb{R})}^2=\mathfrak{c}_+(u_0),
\]
which proves the second identity.
\end{proof}

\begin{thm}[Yosida--Wiener quantization]\label{thm:blowup-dark-quantization}
The unitary part of $\mathsf{A}_t$ on $\mathsf{P}_t^{\mathrm{W}}L^2_+(\mathbb{R})$ has finitely many simple eigenvalues and no continuous spectral subspace. Every normalized dark eigenvector $e$ satisfies
\begin{equation}\label{eq:blowup-dark-overlap}
|\langle e,u_0\rangle_{L^2(\mathbb{R})}|^2=2\pi.
\end{equation}
In particular, if $N_\ast=\dim(\mathsf{P}_T^{\mathrm{W}}L^2_+(\mathbb{R}))$, then
\begin{equation}\label{eq:blowup-defect-quantization}
\mathfrak{c}_+(u_0)=2\pi N_\ast,\qquad1\leq N_\ast\leq\left\lfloor\frac{\|u_0\|_{L^2(\mathbb{R})}^2}{2\pi}\right\rfloor.
\end{equation}
\end{thm}
\begin{proof}
Choose $C=C(u_0)$ so that $\mathbb{H}=L_{u_0}+C\geq1$, and set
\[
B=I-\frac{1}{2\pi}|u_0\rangle\langle u_0|,\quad F_\rho=\rho(\rho+\mathbb{H})^{-1},\quad \mathbb{H}_\rho=\rho\mathbb{H}(\rho+\mathbb{H})^{-1}.
\]
The Ward identity, first on smooth cores and then through the bounded Duhamel identity in Condition~\ref{hyp:blowup-KLV-Ward}, gives
\begin{equation}\label{eq:blowup-Yosida-commutator}
[\mathsf{A}_t,\mathbb{H}_\rho]=iF_\rho BF_\rho
\end{equation}
in the semigroup $C^1$ sense. Let $\mathsf{A}_t^{\mathrm{u}}$ denote the self-adjoint generator on the Wold unitary part. Compression gives $[\mathsf{A}_t^{\mathrm{u}},\mathsf{P}_t^{\mathrm{W}}\mathbb{H}_\rho\mathsf{P}_t^{\mathrm{W}}]=i\mathsf{P}_t^{\mathrm{W}}F_\rho BF_\rho\mathsf{P}_t^{\mathrm{W}}$. If $\mathsf{A}_t^{\mathrm{u}}e=ae$ and $\|e\|_{L^2(\mathbb{R})}=1$, the bounded virial identity and $F_\rho\to I$ strongly yield
\[
0=\lim_{\rho\to\infty}\langle F_\rho e,BF_\rho e\rangle_{L^2(\mathbb{R})}=1-\frac{1}{2\pi}|\langle e,u_0\rangle_{L^2(\mathbb{R})}|^2,
\]
which proves \eqref{eq:blowup-dark-overlap}. An eigenspace of dimension at least two would contain a nonzero vector orthogonal to $u_0$, so every eigenvalue is simple. Bessel's inequality bounds their number by $\lfloor\|u_0\|_{L^2(\mathbb{R})}^2/(2\pi)\rfloor$.

It remains to exclude continuous spectrum. Suppose that a normalized $f$ lies in the continuous spectral subspace of $\mathsf{A}_t^{\mathrm{u}}$. Put
\[
G_\rho=I-F_\rho,\qquad\gamma_\rho=\langle f,G_\rho f\rangle_{L^2(\mathbb{R})},\qquad S_\rho=\rho\sqrt{\gamma_\rho}.
\]
The spectral theorem and $\mathbb{H}\geq1$ imply
\[
\gamma_\rho\to0,\qquad\gamma_\rho\geq(\rho+1)^{-1},\qquad S_\rho\to\infty,\qquad S_\rho/\rho\to0.
\]
For $U(s)=e^{-is\mathsf{A}_t^{\mathrm{u}}}$, define $Q_\rho(s)=\rho\langle U(s)f,G_\rho U(s)f\rangle_{L^2(\mathbb{R})}$.
The commutator identity gives
\begin{equation}\label{eq:blowup-Wiener-energy-derivative}
Q_\rho'(s)=\langle F_\rho U(s)f,BF_\rho U(s)f\rangle_{L^2(\mathbb{R})},\qquad |Q_\rho'(s)|\leq\|B\|_{\mathcal{B}(L^2_+(\mathbb{R}))}.
\end{equation}
Uniformly for $-S_\rho\leq s\leq0$,
\[
\|(I-F_\rho)U(s)f\|_{L^2(\mathbb{R})}^2\leq Q_\rho(s)/\rho\leq\gamma_\rho+\|B\|_{\mathcal{B}(L^2_+(\mathbb{R}))}S_\rho/\rho=o(1).
\]
Thus the time average of the identity part of $B$ tends to one, while
\[
\langle u_0,F_\rho U(s)f\rangle_{L^2(\mathbb{R})}=\langle\mathsf{P}_t^{\mathrm{W}}u_0,U(s)f\rangle_{L^2(\mathbb{R})}+o(1)
\]
uniformly on this interval. The one-sided Wiener theorem for the continuous spectral measure gives
\[
\frac{1}{S_\rho}\int_{-S_\rho}^0|\langle\mathsf{P}_t^{\mathrm{W}}u_0,U(s)f\rangle_{L^2(\mathbb{R})}|^2ds\underset{\rho\to\infty}{\longrightarrow}0.
\]
Integration of \eqref{eq:blowup-Wiener-energy-derivative} therefore gives
\[
\frac{Q_\rho(0)-Q_\rho(-S_\rho)}{S_\rho}\underset{\rho\to\infty}{\longrightarrow}1.
\]
This contradicts $Q_\rho(-S_\rho)\geq0$ and $Q_\rho(0)/S_\rho=\sqrt{\gamma_\rho}\to0$ as $\rho\to\infty$.
Hence the continuous dark subspace is trivial. Finally, Lemma~\ref{lem:blowup-endpoint-row} makes the endpoint dark space nonzero and $\|\mathsf{P}_T^{\mathrm{W}}u_0\|_{L^2(\mathbb{R})}^2=\sum_{j=1}^{N_\ast}|\langle e_j,u_0\rangle_{L^2(\mathbb{R})}|^2=2\pi N_\ast$, which proves \eqref{eq:blowup-defect-quantization}.
\end{proof}

\begin{lemma}[Interior isometry and endpoint channels]
\label{lem:blowup-dark-decomposition}
Let $(e_j)_{j=1}^{N_*}$ be the orthonormal eigenbasis of the endpoint unitary part,
\begin{equation}\label{eq:blowup-dark-eigen-data}
\mathsf{A}_T^{\mathrm{u}}e_j=\alpha_je_j,\qquad\chi_j=\langle u_0,e_j\rangle_{L^2(\mathbb{R})},\qquad |\chi_j|^2=2\pi.
\end{equation}
The $\alpha_j$ are pairwise distinct. Define
\[
\mathfrak{d}_j(\tau)=\chi_jW_{\tau}^{u_0}e_j,\qquad r_{\mathrm{end}}(\tau)=W_{\tau}^{u_0}(I-\mathsf{P}_T^{\mathrm{W}})u_0.
\]
Then, for every $\tau<T$,
\begin{equation}\label{eq:S9-dark-ledger}
u(\tau)=r_{\mathrm{end}}(\tau)+\sum_{j=1}^{N_*}\mathfrak{d}_j(\tau),
\end{equation}
and
\begin{equation}\label{eq:S9-dark-orthogonality}
r_{\mathrm{end}}(\tau)\perp \mathfrak{d}_j(\tau),\qquad\mathfrak{d}_j(\tau)\perp \mathfrak{d}_k(\tau)\ (j\ne k),\qquad\|\mathfrak{d}_j(\tau)\|_{L^2(\mathbb{R})}^2=2\pi.
\end{equation}
Moreover,
\begin{equation}\label{eq:blowup-dark-endpoint-limits}
r_{\mathrm{end}}(\tau)\longrightarrow z_T\quad\text{strongly in }L^2(\mathbb{R}),\qquad\mathfrak{d}_j(\tau)\rightharpoonup0.
\end{equation}
Consequently,
\begin{equation}\label{eq:S9-no-diffuse-interface}
u(\tau)-z_T=\sum_{j=1}^{N_\ast}\mathfrak{d}_j(\tau)+o_{L^2(\mathbb{R})}(1),
\end{equation}
and
\[
\|u(\tau)-r_{\mathrm{end}}(\tau)\|_{L^2(\mathbb{R})}^2=2\pi N_\ast,\qquad\|u(\tau)-z_T\|_{L^2(\mathbb R)}^2\underset{\tau\to T}{\longrightarrow}2\pi N_\ast.
\]
\end{lemma}
\begin{proof}
For $\tau<T$, mass conservation, \eqref{eq:blowup-interior-boundary-flow}, and \eqref{eq:blowup-Wold-ledger} give $\|\mathsf{P}_{\tau}^{\mathrm{W}}u_0\|_{L^2(\mathbb{R})}=0$. If this Wold
space were nonzero, Theorem~\ref{thm:blowup-dark-quantization} would provide a dark atom having nonzero overlap with $u_0$, a contradiction. Hence $\mathsf{P}_{\tau}^{\mathrm{W}}=0$ and
$W_{\tau}^{u_0}$ is an isometry on all of $L^2_+(\mathbb{R})$. Apply it to the orthogonal decomposition $u_0=(I-\mathsf{P}_T^{\mathrm{W}})u_0+\sum_{j=1}^{N_\ast}\chi_je_j$ to obtain \eqref{eq:S9-dark-ledger}--\eqref{eq:S9-dark-orthogonality}.

Scalar-row continuity, the contraction bound, and density of Hardy reproducing kernels imply $W_{\tau}^{u_0}f\rightharpoonup W_T^{u_0}f$ for every fixed $f$. By \eqref{eq:blowup-Wold-ledger}, $W_T^{u_0}e_j=0$ and $(W_T^{u_0})^*W_T^{u_0}=I-\mathsf{P}_T^{\mathrm{W}}$. Thus $\mathfrak{d}_j(\tau)\rightharpoonup0$, while $r_{\mathrm{end}}(\tau)\rightharpoonup z_T$. Interior isometry,
\eqref{eq:blowup-defect-quantization}, and \eqref{eq:blowup-endpoint-row} give
\[
\|r_{\mathrm{end}}(\tau)\|_{L^2(\mathbb{R})}^2=\|(I-\mathsf{P}_T^{\mathrm{W}})u_0\|_{L^2(\mathbb{R})}^2=\|z_T\|_{L^2(\mathbb{R})}^2.
\]
Weak convergence plus equality of norms gives the strong limit in \eqref{eq:blowup-dark-endpoint-limits}. Substitution into \eqref{eq:S9-dark-ledger} proves \eqref{eq:S9-no-diffuse-interface};
\eqref{eq:S9-dark-orthogonality} gives the first norm identity, and the strong convergence of the regular channel gives the second.
\end{proof}

\subsection{Dark rows and simultaneous graph--Schur capture}\label{subsec:blowup-graph-Schur}
We convert the dark-channel decomposition into physical $R$-bubbles for the potential $u(\tau)$.
\begin{lemma}[Adaptive snapshot rows]\label{lem:blowup-snapshot-rows}
Let $(\alpha_j,e_j)$ be as in \eqref{eq:blowup-dark-eigen-data}. Put
\[
\delta_\tau=T-\tau,\qquad\Lambda_\tau=\delta_\tau^{-1/2},\qquad F_\Lambda=\Lambda(\Lambda+\mathbb H)^{-1},\qquad J_\Lambda=-i\Lambda(\mathsf X-i\Lambda)^{-1},
\]
and define
\[
f_{j,\tau}=J_{\Lambda_\tau}F_{\Lambda_\tau}e_j,\qquad g_{j,\tau}=W_{\tau}^{u_0}f_{j,\tau},\qquad h_{j,\tau}=W_{\tau}^{u_0}e_j.
\]
Then
\begin{align}
\|g_{j,\tau}-h_{j,\tau}\|_{L^2(\mathbb{R})}&\underset{\tau\to T}{\longrightarrow}0,\label{eq:blowup-snapshot-close}\\
\|(\mathsf X-\alpha_j)g_{j,\tau}\|_{L^2(\mathbb R)}&\lesssim\Lambda_\tau^{-1},\label{eq:blowup-snapshot-X-row}\\
\|(L_{u(\tau)}+C)g_{j,\tau}\|_{L^2(\mathbb R)}&\lesssim\Lambda_\tau,\label{eq:blowup-snapshot-L-row}\\
\Lambda_\tau^{-1}\|(L_{u(\tau)}+C)g_{j,\tau}\|_{L^2(\mathbb R)}&\underset{\tau\to T}{\longrightarrow}0.\label{eq:blowup-snapshot-refined-row}
\end{align}
\end{lemma}

\begin{proof}
The Yosida commutator \eqref{eq:blowup-Yosida-commutator} and the resolvent commutator identity give
\begin{equation}\label{eq:blowup-F-commutator}
[\mathsf{A}_T,F_\Lambda]=-\frac{i}{\Lambda}F_\Lambda BF_\Lambda.
\end{equation}
The Ward identity also yields $[\mathbb{H},J_\Lambda]=\Lambda(\mathsf{X}-i\Lambda)^{-1}B(\mathsf{X}-i\Lambda)^{-1}$, \begin{equation}\label{eq:blowup-J-commutators}
[\mathsf{A}_T,J_\Lambda]=2T\Lambda(\mathsf{X}-i\Lambda)^{-1}B(\mathsf{X}-i\Lambda)^{-1},
\end{equation}
and hence
\begin{equation}\label{eq:blowup-J-commutator-bounds}
\|[\mathbb{H},J_\Lambda]\|_{\mathcal{B}(L^2(\mathbb{R}))}\leq\frac{\|B\|_{\mathcal{B}(L^2(\mathbb{R}))}}{\Lambda},\qquad\|[\mathsf{A}_T,J_\Lambda]\|_{\mathcal{B}(L^2(\mathbb{R}))}\leq2T\frac{\|B\|_{\mathcal {B}(L^2(\mathbb{R}))}}{\Lambda}.
\end{equation}
These identities hold for smooth approximants and pass to the rough realization by the Yosida commutator stability in Condition~\ref{hyp:blowup-KLV-Ward}.

Since $F_\Lambda e_j\in\operatorname{Dom}(\mathbb{H})$, the bounded commutator $[\mathbb{H},J_\Lambda]$ shows that $J_\Lambda F_\Lambda e_j\in\operatorname{Dom}(\mathbb{H})$, while the resolvent definition of $J_\Lambda$ places it in $\operatorname{Dom}(\mathsf{X})$. Thus $f_{j,\tau}$ belongs to $\operatorname{Dom}(\mathsf {X})\cap\operatorname{Dom}(L_{u_0})$. The common-domain identity \eqref{eq:blowup-common-domain} gives
\begin{equation}\label{eq:blowup-pencil-time-difference}
\mathsf{A}_\tau f_{j,\tau}=\mathsf{A}_T f_{j,\tau}-2\delta_\tau L_{u_0}f_{j,\tau}.
\end{equation}
Since $J_\Lambda,F_\Lambda\to I$ strongly, $f_{j,\tau}\to e_j$. Equations \eqref{eq:blowup-F-commutator}--\eqref{eq:blowup-J-commutator-bounds} give $\|(\mathsf{A}_T-\alpha_j)f_{j,\tau}\|_{L^2(\mathbb{R})}\lesssim\Lambda_\tau^{-1}$, \begin{equation}\label{eq:blowup-snapshot-H-bound}
\|\mathbb{H}f_{j,\tau}\|_{L^2(\mathbb{R})}\leq\Lambda_\tau+\frac{\|B\|_{\mathcal{B}(L^2(\mathbb{R}))}}{\Lambda_\tau}.
\end{equation}
Interior isometry and the two intertwinings prove \eqref{eq:blowup-snapshot-close} and \eqref{eq:blowup-snapshot-L-row}. Applying the pencil intertwining to \eqref{eq:blowup-pencil-time-difference} gives
\[
(\mathsf{X}-\alpha_j)g_{j,\tau}=W_{\tau}^{u_0}\bigl\{(\mathsf{A}_T-\alpha_j)f_{j,\tau}-2\delta_\tau L_{u_0}f_{j,\tau}\bigr\},
\]
which, together with \eqref{eq:blowup-snapshot-H-bound} and $\delta_\tau\Lambda_\tau=\Lambda_\tau^{-1}$, proves \eqref{eq:blowup-snapshot-X-row}. Finally, with
$G_\Lambda=I-F_\Lambda$,
\[
\Lambda^{-1}\mathbb{H}J_\Lambda F_\Lambda e_j=J_\Lambda G_\Lambda e_j+\Lambda^{-1}[\mathbb H,J_\Lambda]F_\Lambda e_j\underset{\Lambda\to\infty}{\longrightarrow}0,
\]
which proves \eqref{eq:blowup-snapshot-refined-row}.
\end{proof}

\begin{lemma}[Low-frequency evacuation]\label{lem:S9-low-evacuation}
For every fixed $K>0$ and every nonempty $I\subset\{1,\ldots,N_\ast\}$,
\begin{equation}\label{eq:S9-low-evacuation}
\left\|P_{\leq K}\sum_{j\in I}\mathfrak{d}_j(\tau)\right\|_{L^2(\mathbb{R})}\underset{\tau\to T}{\longrightarrow}0.
\end{equation}
\end{lemma}
\begin{proof}
Write $G_{j,\tau}(\xi)=e^{i\alpha_j\xi}\widehat g_{j,\tau}(\xi)$.
The half-line trace inequality, Plancherel, and \eqref{eq:blowup-snapshot-X-row} imply $\|\widehat{g}_{j,\tau}\|_{L^\infty(\mathbb{R})}^2\lesssim\Lambda_\tau^{-1}$. Consequently, $\|P_{\leq K}g_{j,\tau}\|_{L^2(\mathbb{R})}^2\lesssim \frac{K}{\Lambda_\tau}$. Now $\mathfrak d_j(\tau)=\chi_jh_{j,\tau}$, $g_{j,\tau}-h_{j,\tau}\to0$ in $L^2(\mathbb{R})$ as $\tau\to T$, and the number of channels is finite. Therefore \eqref{eq:S9-low-evacuation} follows.
\end{proof}

For $v\in L^2_+(\mathbb{R})$, write
\[
\mathsf{T}_{\bar{v}}f=\Pi_+(\bar{v}f),\qquad\mathfrak{l}_v(f,h)=\langle D^{1/2}f,D^{1/2}h\rangle-\langle\mathsf{T}_{\bar v}f,\mathsf{T}_{\bar{v}}h\rangle.
\]
\begin{lemma}[Sharp Hardy--Toeplitz inequality]\label{lem:blowup-sharp-Hardy-Toeplitz}
If $v\in L^2_+(\mathbb{R})$ and $f\in\dot H^{1/2}_+(\mathbb{R})$, then
\begin{equation}\label{eq:blowup-sharp-Toeplitz}
\|\mathsf{T}_{\bar{v}}f\|_{L^2(\mathbb{R})}^2\leq\frac{\|v\|_{L^2(\mathbb{R})}^2}{2\pi}\|D^{1/2}f\|_{L^2(\mathbb{R})}^2.
\end{equation}
If $v,f\ne0$, equality holds if and only if there are $\lambda>0$, $y\in\mathbb R$, and nonzero constants $A,B$ such that
\begin{equation}\label{eq:blowup-sharp-equality}
v=A\mathcal{M}_{\lambda,y,0,0}R,\qquad f=B\mathcal{M}_{\lambda,y,0,0}R.
\end{equation}
The carrier-relative version follows after a common nonnegative modulation.
\end{lemma}
\begin{proof}
For $\xi\geq0$, $\widehat{\mathsf{T}_{\bar{v}}f}(\xi)=\frac{1}{2\pi}\int_0^\infty\widehat{f}(\xi+\eta)\overline{\widehat{v}(\eta)}d\eta$. Cauchy--Schwarz in $\eta$, followed by Plancherel and Fubini, gives
\[
\|\mathsf{T}_{\bar{v}}f\|_{L^2(\mathbb{R})}^2\leq\frac{1}{8\pi^3}\left(\int_0^\infty|\widehat{v}(\eta)|^2d\eta\right)\left(\int_0^\infty s|\widehat{f}(s)|^2ds\right),
\]
which is \eqref{eq:blowup-sharp-Toeplitz}. Equality is equivalent to $\widehat{f}(\xi+\eta)=c(\xi)\widehat{v}(\eta)$ almost everywhere. The measurable Pexider equation obtained by comparing two right shifts gives
\[
\widehat{v}(\eta)=A_0e^{-z\eta},\qquad\widehat{f}(\eta)=B_0e^{-z\eta},\qquad\operatorname{Re}z>0.
\]
Since $\widehat{R}(\eta)=-2\pi i\sqrt2 e^{-\eta}\mathbf{1}_{[0,\infty)}(\eta)$, this gives \eqref{eq:blowup-sharp-equality}. The converse is immediate.
\end{proof}

\begin{lemma}[Graph closure under a kinetic bound]\label{lem:blowup-bounded-graph-closure}
Suppose $v_n\rightharpoonup v$ and $f_n\rightharpoonup f$ in $L^2_+(\mathbb R)$, the sequence $(v_n)$ is bounded in $L^2_+(\mathbb R)$, and
\begin{equation}\label{eq:blowup-bounded-graph-assumptions}
\sup_n\|D^{1/2}f_n\|_{L^2(\mathbb{R})}<\infty,\qquad f_n\in\operatorname{Dom}(L_{v_n}),\qquad\eta_n\to0,\qquad\|(L_{v_n}+\eta_n)f_n\|_{L^2(\mathbb{R})}\to0.
\end{equation}
Then $f\in H^{1/2}_+(\mathbb R)$,
$\mathsf{T}_{\bar v}f\in L^2_+(\mathbb{R})$, and
\[
\mathsf{T}_{\overline{v_n}}f_n\rightharpoonup\mathsf{T}_{\bar{v}}f,\qquad\mathfrak{l}_v(f,\varphi)=0\quad(\varphi\in H^{1/2}_+(\mathbb{R})).
\]
\end{lemma}
\begin{proof}
Weak lower semicontinuity gives $f\in H^{1/2}_+(\mathbb{R})$. For a positive-frequency Schwartz function $\varphi$, the map $v\mapsto\mathsf{T}_{\bar{v}}\varphi$ is Hilbert--Schmidt because its
Fourier kernel is a constant multiple of $\widehat{\varphi}(\xi+\eta)$. Hence $\mathsf{T}_{\overline{v_n}}\varphi \to\mathsf{T}_{\bar{v}}\varphi$ strongly in $L^2(\mathbb{R})$. The map $v\mapsto D^{-1/2}(v\varphi)$ is also compact: truncate the input frequency to obtain a Hilbert--Schmidt operator and estimate the tail by $O_\varphi(A^{-1/2})$. Hankel symmetry gives $\langle\mathsf{T}_{\overline{v_n}}f_n,\varphi\rangle_{L^2(\mathbb{R})}=\langle\mathsf{T}_{\bar{\varphi}}f_n,v_n\rangle_{L^2(\mathbb{R})}$. Multiplication by the rapidly decaying $\varphi$, followed by the Hardy projection, is compact from a bounded subset of $H^{1/2}_+(\mathbb{R})$ to $L^2(\mathbb{R})$; hence $\mathsf{T}_{\bar{\varphi}}f_n\to\mathsf{T}_{\bar{\varphi}}f$ strongly in $L^2(\mathbb{R})$. Together with $v_n\rightharpoonup v$, this identifies the weak output as $\mathsf{T}_{\bar{v}}f$. The uniform $L^2$ bound follows from Lemma~\ref{lem:blowup-sharp-Hardy-Toeplitz}. Pairing the last relation in \eqref{eq:blowup-bounded-graph-assumptions} with $\varphi$ and passing to the limit proves the form equation. Density extends it to $H^{1/2}_+(\mathbb{R})$.
\end{proof}

\begin{lemma}[Stability at the sharp self-row]\label{lem:blowup-sharp-stability}
Let $f_n\in H^{1/2}_+(\mathbb{R})$, $\|f_n\|_{L^2(\mathbb{R})}=1$, and $E_n=\|D^{1/2}f_n\|_{L^2(\mathbb{R})}^2>0$. If
\begin{equation}\label{eq:blowup-self-row-saturation}
\frac{E_n-2\pi\|\mathsf{T}_{\overline{f_n}}f_n\|_{L^2(\mathbb{R})}^2}{E_n}\underset{n\to\infty}{\longrightarrow}0,
\end{equation}
then
\begin{equation}\label{eq:blowup-sharp-stability}
\inf_{\lambda>0,\,y,\,\theta}\left\|f_n-(2\pi)^{-1/2}\mathcal{M}_{\lambda,y,0,\theta}R\right\|_{L^2(\mathbb{R})}\underset{n\to\infty}{\longrightarrow}0.
\end{equation}
After kinetic normalization and suitable translations and phases, the convergence is strong in $H^{1/2}_+(\mathbb{R})$.
\end{lemma}
\begin{proof}
Dilate so that $E_n=1$. The identity $2\|\mathsf{T}_{\bar{f}}f\|_{L^2(\mathbb{R})}^2=\|f\|_{L^4(\mathbb{R})}^4$ turns the hypothesis into $\pi\|f_n\|_{L^4(\mathbb{R})}^4\to1$, the endpoint value of
\begin{equation}\label{eq:blowup-sharp-Hardy-GN}
\pi\|f\|_{L^4(\mathbb{R})}^4\leq\|f\|_{L^2(\mathbb{R})}^2\|D^{1/2}f\|_{L^2(\mathbb{R})}^2.
\end{equation}
Nonvanishing of the $L^4$ norm supplies translations for which a subsequence converges weakly in $H^{1/2}(\mathbb{R})$ to a nonzero $f$. Local Rellich compactness gives almost-everywhere convergence. If $a=\|f\|_{L^2(\mathbb{R})}^2$ and $b=\|D^{1/2}f\|_{L^2(\mathbb{R})}^2$, Hilbert-space Pythagoras, the Brezis--Lieb splitting \cite{BrezisLieb1983}, and
\eqref{eq:blowup-sharp-Hardy-GN} yield
\[
1\leq ab+(1-a)(1-b)=1-a(1-b)-b(1-a)\leq1.
\]
Since $a,b>0$, this forces $a=b=1$. Thus convergence is strong in $H^{1/2}_+(\mathbb{R})$, and the equality classification in Lemma~\ref{lem:blowup-sharp-Hardy-Toeplitz} identifies the limit. Undo the normalization to obtain \eqref{eq:blowup-sharp-stability}.
\end{proof}

Fix a sequence $\tau_n\to T$ and put
\[
\Lambda_n=(T-\tau_n)^{-1/2},\qquad u_n=u(\tau_n),\qquad g_{j,n}=g_{j,\tau_n},\qquad h_{j,n}=h_{j,\tau_n}.
\]
For channel $j$, set
\[
M_{j,n}=\mathcal{M}_{\Lambda_n^{-1},\alpha_j,0,0},\qquad U_{j,n}=M_{j,n}^\ast u_n,\qquad G_{j,n}=M_{j,n}^\ast g_{j,n},\qquad E_{j,n}=\|D^{1/2}G_{j,n}\|_{L^2(\mathbb{R})}^2.
\]
\begin{lemma}[Capture at bounded kinetic energy]\label{lem:blowup-bounded-kinetic-capture}
If $(E_{j,n})$ is bounded, then, after passing to a subsequence, there is a zero-carrier $R$--bubble $Q_{j,n}$ such that
\begin{equation}\label{eq:blowup-bounded-channel-capture}
\|\mathfrak{d}_j(\tau_n)-Q_{j,n}\|_{L^2(\mathbb{R})}\underset{n\to\infty}{\longrightarrow}0.
\end{equation}
\end{lemma}
\begin{proof}
The centered $\mathsf X$ row gives $\|\mathsf{X}G_{j,n}\|_{L^2(\mathbb{R})}=O(1)$, and the kinetic bound gives $\int_K^\infty|\widehat{G}_{j,n}|^2\,d\xi\lesssim K^{-1}$. Rellich compactness on bounded Fourier intervals therefore yields
\[
G_{j,n}\underset{n\to\infty}{\longrightarrow}\gamma_j
\]
strongly in $L^2(\mathbb{R})$, with $\|\gamma_j\|_{L^2(\mathbb{R})}=1$. In the $j$ frame, every $k\ne j$ is centered at $\Lambda_n(\alpha_k-\alpha_j)$ and hence has weak limit zero; the fixed endpoint background also converges weakly to zero. Thus $U_{j,n}\rightharpoonup\phi_j=\chi_j\gamma_j$.
By \eqref{eq:blowup-snapshot-refined-row} and affine covariance,
\[
\|(L_{U_{j,n}}+C/\Lambda_n)G_{j,n}\|_{L^2(\mathbb{R})}\underset{n\to\infty}{\longrightarrow}0.
\]
Lemma~\ref{lem:blowup-bounded-graph-closure} gives $\mathfrak{l}_{\phi_j}(\gamma_j,\gamma_j)=0$, that is, $\|D^{1/2}\gamma_j\|_{L^2(\mathbb{R})}^2=\|\mathsf{T}_{\overline{\phi_j}}\gamma_j\|_{L^2(\mathbb{R})}^2$. Since $\phi_j=\chi_j\gamma_j$, $|\chi_j|^2=2\pi$, and $\|\gamma_j\|_{L^2(\mathbb{R})}=1$, this is equality in Lemma~\ref{lem:blowup-sharp-Hardy-Toeplitz}. Hence $\phi_j$ belongs
to the zero-carrier orbit of $R$. Finally, $\|g_{j,n}-h_{j,n}\|_{L^2(\mathbb{R})}\to0$ and unitarity of the frame give \eqref{eq:blowup-bounded-channel-capture}.
\end{proof}

\begin{lemma}[Capture at diverging kinetic energy]\label{lem:blowup-high-kinetic-source-purity}
Let $U_n\in L^2_+(\mathbb{R})$ and $F_n\in\operatorname{Dom}(L_{U_n})$ satisfy
\[
\|F_n\|_{L^2(\mathbb{R})}\underset{n\to\infty}{\longrightarrow}1,\qquad\|L_{U_n}F_n\|_{L^2(\mathbb{R})}\underset{n\to\infty}{\longrightarrow}0,\qquad E_n:=\|D^{1/2}F_n\|_{L^2(\mathbb{R})}^2\underset{n\to\infty}{\longrightarrow}\infty.
\]
If $|\chi|^2=2\pi$ and
\begin{equation}\label{eq:blowup-high-kinetic-source-purity}
\|\mathsf{T}_{\overline{U_n-\chi F_n}}F_n\|_{L^2(\mathbb{R})}=o(E_n^{1/2}),
\end{equation}
then there are normalized zero-carrier $R$-bubbles $r_n$ such that
\begin{equation}\label{eq:blowup-enhanced-capture}
\|F_n-r_n\|_{L^2(\mathbb{R})}\underset{n\to\infty}{\longrightarrow}0,\qquad\|D^{1/2}(F_n-r_n)\|_{L^2(\mathbb{R})}=o(E_n^{1/2}).
\end{equation}
\end{lemma}
\begin{proof}
The graph equation gives $E_n-\|\mathsf{T}_{\overline{U_n}}F_n\|_{L^2(\mathbb{R})}^2=o(1)$. The hypothesis and $|\chi|^2=2\pi$ imply
\[
\frac{2\pi\|\mathsf{T}_{\overline{F_n}}F_n\|_{L^2(\mathbb{R})}^2}{E_n}\underset{n\to\infty}{\longrightarrow}1.
\]
After normalizing $F_n$ in $L^2(\mathbb{R})$, this is \eqref{eq:blowup-self-row-saturation}. Apply Lemma~\ref{lem:blowup-sharp-stability}; its strong $H^{1/2}$ conclusion after kinetic normalization gives \eqref{eq:blowup-enhanced-capture}.
\end{proof}

Place all channels in the common frame
\[
\mathbf{M}_n=\mathcal{M}_{\Lambda_n^{-1},0,0,0},\qquad U_n=\mathbf{M}_n^\ast u_n,\qquad\Gamma_{j,n}=\mathbf M_n^*g_{j,n},\qquad x_{j,n}=\Lambda_n\alpha_j,
\]
and write $E(F)=\|D^{1/2}F\|_{L^2(\mathbb{R})}^2$. The snapshot lemma gives
\[
\|\Gamma_{j,n}\|_{L^2(\mathbb{R})}\underset{n\to\infty}{\longrightarrow}1,\qquad\|(\mathsf{X}-x_{j,n})\Gamma_{j,n}\|_{L^2(\mathbb{R})}=O(1),\qquad\|L_{U_n}\Gamma_{j,n}\|_{L^2(\mathbb{R})}\underset{n\to\infty}{\longrightarrow}0.
\]
Moreover,
\begin{equation}\label{eq:blowup-common-potential-decomposition}
U_n=\omega_n+\sum_{j=1}^{N_\ast}\chi_j\Gamma_{j,n},
\end{equation}
where $\omega_n=\mathbf{M}_n^\ast r_{\mathrm{end}}(\tau_n)+\sum_j\chi_j\mathbf{M}_n^*(h_{j,n}-g_{j,n})$. \begin{lemma}[Uniform removal of the endpoint background]\label{lem:blowup-background-source-small}
Let $F_n$ be any member of a fixed finite family of near-identity linear combinations of the rows $\Gamma_{j,n}$, with uniformly bounded $L^2(\mathbb{R})$ norm and a centered $\mathsf{X}$ row. If
$E(F_n)\to\infty$, then $\|\mathsf{T}_{\overline{\omega_n}}F_n\|_{L^2(\mathbb{R})}=o(E(F_n)^{1/2})$. \end{lemma}
\begin{proof}
The terms $r_{\mathrm{end}}(\tau_n)-z_T$ and $h_{j,n}-g_{j,n}$ tend to zero in $L^2(\mathbb{R})$. The sharp inequality bounds their Toeplitz sources by $o(E(F_n)^{1/2})$, uniformly over the finite
family. It remains to treat $\mathbf{M}_n^\ast z_T$. Given $\rho>0$, choose $z^{(b)}\in L^2_+(\mathbb{R})\cap L^\infty(\mathbb{R})$ with $\|z_T-z^{(b)}\|_{L^2(\mathbb{R})}<\rho$.
Affine covariance gives $\|\mathbf{M}_n^*z^{(b)}\|_{L^\infty(\mathbb{R})}=\Lambda_n^{-1/2}\|z^{(b)}\|_{L^\infty(\mathbb{R})}$, so its source is $o(1)$ because $\|F_n\|_{L^2(\mathbb{R})}=O(1)$. The remaining source is at most $C\rho E(F_n)^{1/2}$ by Lemma~\ref{lem:blowup-sharp-Hardy-Toeplitz}. Dividing by
$E(F_n)^{1/2}$, letting $n\to\infty$, and then letting $\rho\to0$ proves the claim.
\end{proof}

If $N_\ast\geq2$, the $\alpha_j$ are distinct and
\[
d_{jk,n}=|x_{j,n}-x_{k,n}|=\Lambda_n|\alpha_j-\alpha_k|,\qquad d_n=\min_{j\ne k}d_{jk,n}\underset{n\to\infty}{\longrightarrow}\infty.
\]
When $N_\ast=1$, we use the convention $d_n=+\infty$; every assertion below involving an off-diagonal label is then void. A normalized row satisfying a uniformly centered
$\mathsf{X}$ bound has $\liminf_{n\to\infty}E(F_n)>0$. Indeed, after modulation its Fourier transform is bounded in $H^1(0,\infty)$; the resulting $L^\infty$ bound prevents unit mass from concentrating at frequency zero.
\begin{lemma}[Trace--Cauchy decomposition and cross estimate]\label{lem:blowup-cross-estimate}
Let $f\in\operatorname{Dom}(\mathsf{X})\cap H^{1/2}_+(\mathbb{R})$ and $\|f\|_{L^2(\mathbb{R})}+\|(\mathsf{X}-y)f\|_{L^2(\mathbb{R})}\leq C_0$. With $r=R/\sqrt{2\pi}$ and $r_y=\mathcal{M}_{1,y,0,0}r$, one can write
\begin{equation}\label{eq:blowup-trace-decomposition}
f=\alpha r_y+f^\circ,\qquad|\alpha|+\|(x-y)f^\circ\|_{L^2(\mathbb{R})}\lesssim_{C_0}1,\qquad E(f^\circ)\leq2E(f)+O_{C_0}(1).
\end{equation}
Consequently, if $f,g$ have centers $y_f,y_g$ and $d=|y_f-y_g|\geq10$, then
\begin{equation}\label{eq:blowup-cross-estimate}
\|\mathsf{T}_{\bar g}f\|_{L^2(\mathbb{R})}^2\leq\|gf\|_{L^2(\mathbb{R})}^2\lesssim\frac{\sqrt{(E(f)+1)(E(g)+1)}}{d}+\frac{1}{d^2}.
\end{equation}
\end{lemma}
\begin{proof}
Set $F(\xi)=e^{iy\xi}\widehat{f}(\xi)$ and choose $\alpha=F(0)/\widehat{r}(0)$. The half-line trace inequality bounds $\alpha$, while the Fourier transform of $f^\circ=f-\alpha r_y$ has
zero trace. Fourier differentiation then produces no boundary Dirac mass and gives the moment bound in \eqref{eq:blowup-trace-decomposition}; the kinetic estimate follows from $|A-B|^2\leq2|A|^2+2|B|^2$.

Let $m=(y_f+y_g)/2$ and choose a smooth cutoff changing from zero to one on an interval of length comparable to $d$ around $m$. The moment bounds give $O(d^{-1})$ wrong-side $L^2$ mass for the zero-trace remainders. The fractional IMS estimate and the one-dimensional Gagliardo--Nirenberg inequality then yield
\[
\|f^\circ g^\circ\|_{L^2(\mathbb{R})}^2\lesssim d^{-1}\sqrt{(E(f^\circ)+d^{-1})(E(g^\circ)+d^{-1})}.
\]
The two mixed Cauchy terms and the Cauchy--Cauchy term are $O(d^{-2})$. Combining the four terms proves \eqref{eq:blowup-cross-estimate}.
\end{proof}

Define the critical colligation $\mathcal{C}_v=\sqrt{2\pi}\,\mathsf{T}_{\bar{v}}D^{-1/2}$. The sharp Hardy--Toeplitz inequality implies
\begin{equation}\label{eq:blowup-colligation-Lipschitz}
\|\mathcal{C}_v\|_{\mathcal{B}(L^2(\mathbb R))}\leq\|v\|_{L^2(\mathbb{R})},\qquad\|\mathcal{C}_v-\mathcal{C}_w\|_{\mathcal{B}(L^2(\mathbb{R}))}\leq\|v-w\|_{L^2(\mathbb{R})}.
\end{equation}
\begin{lemma}[Critical colligation gap]\label{lem:blowup-colligation-gap}
For the normalized bubble $r=R/\sqrt{2\pi}$, $\mathcal{C}_r$ is compact. Its top singular value is one and is simple, with right singular vector $e_r=D^{1/2}r/\|D^{1/2}r\|_{L^2(\mathbb{R})}$.
Hence there is $s_2<1$ such that $h\perp e_r\quad\Longrightarrow\quad\|\mathcal{C}_rh\|_{L^2(\mathbb{R})}\leq s_2\|h\|_{L^2(\mathbb{R})}$. The same gap holds after zero-carrier affine transformations.
\end{lemma}
\begin{proof}
Up to the fixed Fourier normalization, $(\mathcal{C}_rh)(\xi)=C\int_\xi^\infty e^{-(s-\xi)}h(s)s^{-1/2}ds$. Truncation to $s\leq A$ is Hilbert--Schmidt and the tail has norm $O(A^{-1/2})$, so the operator is compact. Its norm is at most one by \eqref{eq:blowup-colligation-Lipschitz}; the equality classification in Lemma~\ref{lem:blowup-sharp-Hardy-Toeplitz} shows that equality occurs only in the stated direction. Compactness gives the strict second singular-value gap. The affine covariance of $D^{-1/2}$ and of the Toeplitz row proves the final assertion.
\end{proof}

For a normalized $R$--bubble $r_p$, denote by $P^{\mathrm{R}}_{p,\eta}$ and $P^{\mathrm{L}}_{p,\eta}$ the right and left singular projections of $\mathcal{C}_{r_p}$ for singular values at least $\eta>0$; they have finite rank independent of $p$. For $p=(\lambda,y,0,\theta)$, abbreviate $\mathcal{M}_p=\mathcal{M}_{\lambda,y,0,\theta}$.
\begin{lemma}[Locking and two-sided block decay]\label{lem:blowup-block-decay}
Suppose $\|(\mathsf{X}-x_n)v_n\|_{L^2(\mathbb{R})}=O(1)$ and $v_n-r_{p_n}\to0$ in $L^2(\mathbb{R})$, where $p_n=(\lambda_n,y_n,0,\theta_n)$. Then
\begin{equation}\label{eq:blowup-frame-locking}
\lambda_n=O(1),\qquad |y_n-x_n|=O(1+\lambda_n).
\end{equation}
More generally, let $(w_n)$ be $L^2(\mathbb{R})$-bounded with $\|(\mathsf{X}-x_n)w_n\|_{L^2(\mathbb{R})}=O(1)$, and let $q_n=(\delta_n,z_n,0,\vartheta_n)$ satisfy
\[
\frac{|x_n-z_n|}{1+\delta_n}\underset{n\to\infty}{\longrightarrow}\infty.
\]
Then, for each fixed $\eta>0$,
\begin{equation}\label{eq:blowup-two-sided-block-decay}
\|\mathcal{C}_{w_n}P^{\mathrm{R}}_{q_n,\eta}\|_{\mathcal{B}(L^2(\mathbb{R}))}\underset{n\to\infty}{\longrightarrow}0,\qquad\|P^{\mathrm{L}}_{q_n,\eta}\mathcal{C}_{w_n}\|_{\mathcal{B}(L^2(\mathbb{R}))}
\underset{n\to\infty}{\longrightarrow}0.
\end{equation}
The analogous limits hold for two $R$ frames whose relative affine parameters escape every compact subset.
\end{lemma}
\begin{proof}
Pulling the centered row back by $p_n$ gives $\|(\lambda_n\mathsf{X}+y_n-x_n)\mathcal{M}_{p_n}^\ast v_n\|_{L^2(\mathbb{R})}=O(1)$. If
\[
\frac{|y_n-x_n|}{1+\lambda_n}\underset{n\to\infty}{\longrightarrow}\infty,
\]
division by the dominant coefficient and testing against the zero-trace core of $\operatorname{Dom}(\mathsf{X}^\ast)$ force the pulled-back row to converge weakly to zero, contradicting its strong limit $r$. The remaining failure would have $\lambda_n\to\infty$ and $(y_n-x_n)/\lambda_n\to a\in\mathbb{R}$; it would imply $(\mathsf{X}+a)r=0$, impossible because $\widehat {r}(\xi)=Ce^{-\xi}\mathbf{1}_{\xi\ge0}$ and $a$ is real. This proves \eqref{eq:blowup-frame-locking}.

For the second assertion, pull $w_n$ back by $q_n$. The centered $\mathsf{X}$ bound and the stated separation make this pullback converge weakly to zero. For a fixed right singular vector $\phi$, the map $v\mapsto\mathcal{C}_v\phi$ is Hilbert--Schmidt. For a Schwartz left vector $\psi$, $\mathcal{C}_v^\ast\psi=\sqrt{2\pi}\,D^{-1/2}(v\psi)$, and the compactness used in Lemma~\ref{lem:blowup-bounded-graph-closure} applies. Finite rank and density prove \eqref{eq:blowup-two-sided-block-decay}; affine weak dispersion gives the two-bubble version.
\end{proof}

\begin{lemma}[Saturation of the free critical block]\label{lem:blowup-free-block-saturation}
Let $v_n$ be normalized and centered, and let $r_{k,n}$ be finitely many pairwise separated normalized $R$--bubbles, all separated from $v_n$. Put $A_{0,n}=e^{i\vartheta_{0,n}}\mathcal{C}_{v_n},\qquad A_{k,n}=e^{i\vartheta_{k,n}}\mathcal{C}_{r_{k,n}}$. If $\|h_n\|_{L^2(\mathbb{R})}=1$, $\langle h_n,e_{r_{k,n}}\rangle_{L^2(\mathbb{R})}\to0$ for every $k$, and
\begin{equation}\label{eq:blowup-block-saturation-assumption}
\left\|A_{0,n}h_n+\sum_kA_{k,n}h_n\right\|_{L^2(\mathbb{R})}^2\underset{n\to\infty}{\longrightarrow}1,
\end{equation}
then
\[
\sum_k\|A_{k,n}h_n\|_{L^2(\mathbb{R})}^2\underset{n\to\infty}{\longrightarrow}0,\qquad\|A_{0,n}h_n\|_{L^2(\mathbb{R})}^2\underset{n\to\infty}{\longrightarrow}1.
\]
\end{lemma}
\begin{proof}
Fix $0<\eta<1-s_2$ and split each captured colligation into its finite singular block above $\eta$ and a tail of norm at most $\eta$. Lemma~\ref{lem:blowup-block-decay} makes all finite right and left blocks mutually orthogonal up to $o(1)$ and two-sided orthogonal to the free block. The top directions are removed by the hypotheses, while every remaining singular value is at most $s_2$. After an $o(1)$ Gram correction,
\[
\left\|A_{0,n}h_n+\sum_kA_{k,n}h_n\right\|_{L^2(\mathbb{R})}^2\leq1-(1-s_2^2)\sum_k\|P^{\mathrm{R}}_{k,n,\eta}h_n\|_{L^2(\mathbb{R})}^2+O(\eta)+o(1).
\]
Letting first $n\to\infty$ and then $\eta\to0$, all captured outputs vanish, and \eqref{eq:blowup-block-saturation-assumption} forces the free output to saturate.
\end{proof}

Call a previously captured row $Z_{k,n}$ with diverging kinetic energy enhanced if there is a normalized $R$--bubble $r_{k,n}$ satisfying
\begin{equation}\label{eq:blowup-enhanced-row}
\|Z_{k,n}-r_{k,n}\|_{L^2(\mathbb{R})}\to0,\qquad\|D^{1/2}(Z_{k,n}-r_{k,n})\|_{L^2(\mathbb{R})}=o(E(Z_{k,n})^{1/2}).
\end{equation}
Lemma~\ref{lem:blowup-sharp-stability} supplies this estimate for every high-kinetic captured row. Lemma~\ref{lem:blowup-block-decay} locks its center to the corresponding $x_{k,n}$, so bubbles with distinct dark labels are pairwise separated.

\begin{lemma}[Graph-preserving $D^{1/2}$--Schur correction]\label{lem:blowup-Schur-correction}
Let $Z_{k,n}$, $k\in P$, be finitely many enhanced captured rows, and let $V_{j,n}$ be a corrected row satisfying
\[
\|V_{j,n}-\Gamma_{j,n}\|_{L^2(\mathbb{R})}=o(1),\quad\|(\mathsf{X}-x_{j,n})V_{j,n}\|_{L^2(\mathbb{R})}=O(1),\quad\|L_{U_n}V_{j,n}\|_{L^2(\mathbb{R})}=o(1).
\]
If $E(Z_{k,n})\gtrsim d_{jk,n}^2E(V_{j,n})$ for every $k\in P$, there are unique coefficients $\beta_{jk,n}$ for which
\begin{equation}\label{eq:blowup-Schur-system}
F_{j,n}=V_{j,n}-\sum_{k\in P}\beta_{jk,n}Z_{k,n},\qquad\langle D^{1/2}F_{j,n},e_{r_{\ell,n}}\rangle_{L^2(\mathbb{R})}=0\quad(\ell\in P).
\end{equation}
They obey
\begin{equation}\label{eq:blowup-Schur-coefficients}
|\beta_{jk,n}|\lesssim\sqrt{E(V_{j,n})/E(Z_{k,n})}=O(d_{jk,n}^{-1}),
\end{equation}
and
\begin{equation}\label{eq:blowup-Schur-preservation}
\begin{aligned}
&\|F_{j,n}-V_{j,n}\|_{L^2(\mathbb{R})}\underset{n\to\infty}{\longrightarrow}0,\qquad\|(\mathsf{X}-x_{j,n})F_{j,n}\|_{L^2(\mathbb{R})}=O(1),\\
&\|L_{U_n}F_{j,n}\|_{L^2(\mathbb{R})}\underset{n\to\infty}{\longrightarrow}0,\qquad E(F_{j,n})\lesssim_P E(V_{j,n}).
\end{aligned}
\end{equation}
\end{lemma}
\begin{proof}
For normalized bubbles, direct Fourier integration gives
\[
\left|\langle e_{r_{\lambda,x}},e_{r_{\delta,y}}\rangle_{L^2(\mathbb{R})}\right|=\left|\frac{4\lambda\delta}{(\lambda+\delta+i(x-y))^2}\right|.
\]
Pairwise separation and \eqref{eq:blowup-enhanced-row} make the normalized $D^{1/2}$ Gram matrix in \eqref{eq:blowup-Schur-system} equal to $I+o(1)$. It is uniformly invertible, and Cauchy--Schwarz gives \eqref{eq:blowup-Schur-coefficients} and the kinetic bound. The factor $O(d_{jk,n}^{-1})$ compensates the displacement $|x_{j,n}-x_{k,n}|=d_{jk,n}$ in the centered $\mathsf {X}$ row. Linearity of $L_{U_n}$ on its common domain proves the last convergence in \eqref{eq:blowup-Schur-preservation}.
\end{proof}

\begin{thm}[Simultaneous graph--Schur capture]\label{thm:S9-simultaneous-capture}
For every sequence $\tau_n\to T$, there are a subsequence, a permutation $j_1,\ldots,j_{N_\ast}$, and zero-carrier $R$--bubbles $Q_{\ell,n}$ such that
\[
\|\mathfrak{d}_{j_\ell}(\tau_n)-Q_{\ell,n}\|_{L^2(\mathbb{R})}\underset{n\to\infty}{\longrightarrow}0\qquad(1\leq\ell\leq N_\ast).
\]
More generally, for every nonempty active set $I$, the ordering can be chosen so that some $j\in I$ is captured after at most $N_\ast-|I|$ labels from $I^c$.
\end{thm}
\begin{proof}
If $N_\ast=1$, take $F_{1,n}=\Gamma_{1,n}$. A bounded kinetic subsequence is captured by Lemma~\ref{lem:blowup-bounded-kinetic-capture}. Otherwise $E(F_{1,n})\to\infty$; then
\eqref{eq:blowup-common-potential-decomposition} contains only the self row and the endpoint background, and Lemma~\ref{lem:blowup-background-source-small} gives \eqref{eq:blowup-high-kinetic-source-purity}. Apply Lemma~\ref{lem:blowup-high-kinetic-source-purity} with $(U_n,F_n,\chi)=(U_n,\Gamma_{1,n},\chi_1)$. The resulting enhanced $R$--row, together with
$g_{1,n}-h_{1,n}\to0$ in $L^2(\mathbb{R})$, gives the required physical bubble. Thus the assertion holds for $N_\ast=1$.

Assume $N_\ast\geq2$. Maintain a captured set $C$ and corrected rows. Initially $C=\varnothing$ and $F_{j,n}=\Gamma_{j,n}$. For a remaining label $j$ and each previously captured enhanced row $Z_{k,n}$, pass to a subsequence so that exactly one of
\begin{equation}\label{eq:blowup-ratio-dichotomy}
\frac{E(Z_{k,n})}{d_{jk,n}^2E(F_{j,n})}\underset{n\to\infty}{\longrightarrow}0,\qquad\text{or}\qquad\liminf_n\frac{E(Z_{k,n})}{d_{jk,n}^2E(F_{j,n})}>0
\end{equation}
holds. Apply Lemma~\ref{lem:blowup-Schur-correction} simultaneously to the second class and recompute the ratios. Every nonterminal recomputation adds an index, so this inner procedure stops after at most $|C|$ steps. Its output differs from $\Gamma_{j,n}$ by $o_{L^2(\mathbb{R})}(1)$ and retains all three row properties in \eqref{eq:blowup-Schur-preservation}.

Among the remaining corrected rows choose one, again denoted $F_{j,n}$, with maximal kinetic energy and put $E_n=E(F_{j,n})$. If $(E_n)$ is bounded, the proof of Lemma~\ref{lem:blowup-bounded-kinetic-capture} applies verbatim: the near-identity correction preserves the dark label, the centered row supplies compactness, and the Lax equation closes the graph. Thus the entire original $j$ channel is captured.

Assume $E_n\to\infty$. The triangular Schur transformation and its inverse are $I+O(d_n^{-1})$, and therefore
\begin{equation}\label{eq:blowup-triangular-potential}
U_n=\omega_n+\sum_{\ell\notin C}\chi_\ell F_{\ell,n}+\sum_{k\in C}\widetilde{\chi}_{k,n}Z_{k,n},\qquad \widetilde{\chi}_{k,n}=\chi_k+o(1),
\end{equation}
with self coefficient $\chi_j$. For another active row, maximality of $E_n$ and Lemma~\ref{lem:blowup-cross-estimate} give
\[
\frac{\|\mathsf{T}_{\overline{F_{\ell,n}}}F_{j,n}\|_{L^2(\mathbb{R})}^2}{E_n}\lesssim d_{j\ell,n}^{-1}\sqrt{\frac{E(F_{\ell,n})+1}{E_n}}+o(1)\underset{n\to\infty}{\longrightarrow}0.
\]
The same cross estimate and the first alternative in \eqref{eq:blowup-ratio-dichotomy} remove all unprojected captured rows; Lemma~\ref{lem:blowup-background-source-small} removes $\omega_n$.

For every projected row, replace $Z_{k,n}$ by its enhanced bubble $r_{k,n}$. By \eqref{eq:blowup-colligation-Lipschitz}, the cost is $o(E_n^{1/2})$, independently of the possibly larger kinetic energy of $Z_{k,n}$. Put $h_n=D^{1/2}F_{j,n}/E_n^{1/2}$. The graph identity gives
\[
E_n-\|\mathsf{T}_{\overline{U_n}}F_{j,n}\|_{L^2(\mathbb{R})}^2=\langle L_{U_n}F_{j,n},F_{j,n}\rangle_{L^2(\mathbb{R})}=o(1).
\]
Using \eqref{eq:blowup-triangular-potential}, the preceding small terms, and $|\chi_j|^2=2\pi$, we obtain phases such that
\[
\left\|e^{i\vartheta_{j,n}}\mathcal{C}_{F_{j,n}}h_n+\sum_{k\in P}e^{i\vartheta_{k,n}}\mathcal{C}_{r_{k,n}}h_n\right\|_{L^2(\mathbb{R})}^2=1+o(1).
\]
The Schur equations make $h_n$ orthogonal to the top right singular directions of the captured blocks. Since $\|F_{j,n}\|_{L^2(\mathbb{R})}\to1$, normalize the free row and apply
Lemma~\ref{lem:blowup-free-block-saturation}; the normalization changes the colligation by $o(1)$ by \eqref{eq:blowup-colligation-Lipschitz}. Every projected source vanishes, and the self-row saturates. Restoring the discarded terms yields $\|\mathsf{T}_{\overline{U_n-\chi_jF_{j,n}}}F_{j,n}\|_{L^2(\mathbb{R})}=o(E_n^{1/2})$. Lemma~\ref{lem:blowup-high-kinetic-source-purity}, applied with $(U_n,F_n,\chi)=(U_n,F_{j,n},\chi_j)$, produces a normalized $R$--bubble $r_{j,n}$ with the enhanced estimate. Since
\[
F_{j,n}-\Gamma_{j,n}=o_{L^2(\mathbb{R})}(1)\quad\text{and}\quad g_{j,n}-h_{j,n}=o_{L^2(\mathbb{R})}(1),
\]
returning to physical coordinates gives a bubble $Q_{j,n}$ satisfying
\[
\|\mathfrak{d}_j(\tau_n)-Q_{j,n}\|_{L^2(\mathbb{R})}\underset{n\to\infty}{\longrightarrow}0.
\]

Add $j$ to $C$ and iterate. Every outer step adds one label, so the process terminates after at most $N_\ast$ steps. If an active set $I$ is prescribed, stop at the first selected label in $I$; all preceding labels belong to $I^c$, which proves the hereditary assertion.
\end{proof}

\begin{cor}[Finite descent]\label{cor:S9-finite-descent}
Every sequence $\tau_n\to T$ has a subsequence and zero-carrier $R$--bubbles $Q_{1,n},\ldots,Q_{N_\ast,n}$ such that
\begin{equation}\label{eq:S9-sequential-resolution}
u(\tau_n)=z_T+\sum_{j=1}^{N_\ast}Q_{j,n}+o_{L^2(\mathbb{R})}(1).
\end{equation}
If $\rho_n^{(m)}=u(\tau_n)-z_T-\sum_{j=1}^mQ_{j,n}$, then
\[
\begin{aligned}
\|\rho_n^{(m)}\|_{L^2(\mathbb{R})}^2&\underset{n\to\infty}{\longrightarrow}2\pi(N_\ast-m),\\
\langle Q_{j,n},Q_{k,n}\rangle_{L^2(\mathbb{R})}&\underset{n\to\infty}{\longrightarrow}0\quad(j\ne k),\\
\|\rho_n^{(m-1)}\|_{L^2(\mathbb{R})}^2&=2\pi+\|\rho_n^{(m)}\|_{L^2(\mathbb{R})}^2+o(1).
\end{aligned}
\]
In particular, $\rho_n^{(N_\ast)}\to0$ strongly in $L^2(\mathbb{R})$.
\end{cor}
\begin{proof}
Relabel the channels according to Theorem~\ref{thm:S9-simultaneous-capture}. Substitute $Q_{j,n}=\mathfrak{d}_j(\tau_n)+o_{L^2(\mathbb{R})}(1)$ into \eqref{eq:S9-dark-ledger} and use \eqref{eq:blowup-dark-endpoint-limits} to obtain \eqref{eq:S9-sequential-resolution}. The orthogonality in \eqref{eq:S9-dark-orthogonality} gives all three Pythagorean statements. At $m=N_\ast$ only
$r_{\mathrm{end}}(\tau_n)-z_T$ remains.
\end{proof}

\subsection{Borel selection for all times}\label{subsec:blowup-Borel}
To obtain Borel measurable bubbles for all times, let $\widehat{\mathcal{P}}_0=(0,\infty)\times\mathbb{R}\times\mathbb{R}/2\pi\mathbb{Z}$, identify $\mathbb{R}/2\pi\mathbb{Z}$ with the unit circle by $\zeta=e^{i\theta}$, and write $Q_{\lambda,y,\zeta}(x)=\lambda^{-1/2}\zeta R((x-y)/\lambda)$.
For unordered $J$--tuples in $L^2_+(\mathbb{R})$, use the quotient metric
\[
d_{J,2}([f_1,\ldots,f_J],[g_1,\ldots,g_J])=\min_{\pi\in\mathfrak{S}_J}\left(\sum_{j=1}^J\|f_j-g_{\pi(j)}\|_{L^2(\mathbb{R})}^2\right)^{1/2}.
\]
This is a complete separable metric because it is the finite isometric quotient of the Polish space $(L^2_+(\mathbb{R}))^J$. The map $(\lambda,y,\zeta)\mapsto Q_{\lambda,y,\zeta}$ is strongly continuous.

Put $J=N_\ast$ and let $\mathcal{C}_J^R$ be the set of unordered zero-carrier $J$--bubble configurations. With $\mathcal{D}(\tau)=[\mathfrak{d}_1(\tau),\ldots,\mathfrak{d}_J(\tau)]$, define $\Delta_J(\tau)=\operatorname{dist}_{d_{J,2}}(\mathcal{D}(\tau),\mathcal{C}_J^R)$.
The map $\tau\mapsto \mathfrak{d}_j(\tau)$ is strongly continuous on every compact subinterval of $[0,T)$: scalar-row continuity first gives weak continuity, while \eqref{eq:S9-dark-orthogonality} fixes its norm at $\sqrt{2\pi}$, so weak continuity and equality of norms give strong continuity. Thus $\Delta_J$ is continuous away from the endpoint.
\begin{lemma}[Sequential criterion and Borel near-minimizer]
\label{lem:blowup-Borel-selection}
One has
\begin{equation}\label{eq:blowup-full-time-configuration-zero}
\Delta_J(\tau)\underset{\tau\to T}{\longrightarrow}0.
\end{equation}
Moreover, there are Borel zero-carrier bubbles $Q_j(\tau)$ satisfying
\begin{equation}\label{eq:blowup-full-time-channel-matching}
E_J(\tau):=\left(\sum_{j=1}^J\|\mathfrak{d}_j(\tau)-Q_j(\tau)\|_{L^2(\mathbb{R})}^2\right)^{1/2}\underset{\tau\to T}{\longrightarrow}0,
\end{equation}
and, for $\tau$ sufficiently close to $T$,
\begin{equation}\label{eq:blowup-near-minimizer}
E_J(\tau)<\Delta_J(\tau)+(T-\tau).
\end{equation}
\end{lemma}
\begin{proof}
If \eqref{eq:blowup-full-time-configuration-zero} failed, there would be $\eta_0>0$ and $\tau_n\to T$ with $\Delta_J(\tau_n)\geq\eta_0$. Theorem~\ref{thm:S9-simultaneous-capture} gives a subsequence, bubbles, and a permutation whose channelwise cost tends to zero, contradicting the definition of $\Delta_J$.

For measurable selection, choose the countable dense grid
\[
\mathcal{G}_0=\{(e^a,b,e^{2\pi i r}):a,b\in\mathbb{Q},\ r\in\mathbb{Q}\cap[0,1)\}\subset\widehat{\mathcal{P}}_0,
\]
and enumerate $\mathcal{G}_0^J$ as $\widehat{\mathbf{g}}^{\,m}$, $m\geq1$. Define
\[
C_{m,\pi}(\tau)=\left(\sum_{j=1}^J\|\mathfrak{d}_j(\tau)-Q_{\widehat{g}_{\pi(j)}^{m}}\|_{L^2(\mathbb{R})}^2\right)^{1/2},\qquad
F_m(\tau)=\min_{\pi\in\mathfrak{S}_J}C_{m,\pi}(\tau).
\]
Every $F_m$ is continuous and density gives $\Delta_J(\tau)=\inf_mF_m(\tau)$. Set $m(\tau)=\min\{m:F_m(\tau)<\Delta_J(\tau)+(T-\tau)\}$. Its fibers are Borel because
\[
\{m(\tau)=m\}=\{F_m<\Delta_J+T-\tau\}\cap\bigcap_{k<m}\{F_k\geq\Delta_J+T-\tau\}.
\]
Order the finite group $\mathfrak{S}_J$ and choose the first minimizing permutation. Its fibers are finite intersections of strict or weak inequalities between the Borel costs, so the choice, denoted $\pi_m(\tau)$, is Borel. Set $\pi=\pi_{m(\tau)}(\tau)$. For $1\le j\le J$, set $Q_j(\tau)=Q_{\widehat g_{\pi(j)}^{\,m(\tau)}}$.
This proves \eqref{eq:blowup-near-minimizer}, and \eqref{eq:blowup-full-time-configuration-zero} proves \eqref{eq:blowup-full-time-channel-matching}. Taking the principal Borel argument of each selected phase produces the Borel functions $\theta_j$ in Theorem~\ref{thm:main}; extend all parameters constantly to an earlier compact part of $[0,T)$ if necessary.
\end{proof}

For the finite-time branch of Theorem~\ref{thm:main}, the weak endpoint follows from Theorem~\ref{thm:blowup-weak-endpoint}, and \eqref{eq:blowup-quantization} follows from Theorem~\ref{thm:blowup-dark-quantization}. Choose the Borel bubbles in Lemma~\ref{lem:blowup-Borel-selection} and set
\[
e(\tau)=u(\tau)-z_T-\sum_{j=1}^JQ_j(\tau)=(r_{\mathrm{end}}(\tau)-z_T)+\sum_{j=1}^J(\mathfrak{d}_j(\tau)-Q_j(\tau)).
\]
Equations \eqref{eq:blowup-dark-endpoint-limits} and \eqref{eq:blowup-full-time-channel-matching} give
\[
\|e(\tau)\|_{L^2(\mathbb{R})}\leq\|r_{\mathrm{end}}(\tau)-z_T\|_{L^2(\mathbb{R})}+\sqrt J E_J(\tau){\longrightarrow}0,
\]
which proves the finite-time decomposition.

For $p=(\lambda,y,0,\theta)$, the explicit Fourier transform of $R$ gives
\begin{equation}\label{eq:blowup-R-low-tail}
\|P_{\leq K}\mathcal{M}_pR\|_{L^2(\mathbb{R})}^2=2\pi(1-e^{-2\lambda K}).
\end{equation}
By Lemma~\ref{lem:S9-low-evacuation} and the channelwise matching, $P_{\leq K}Q_j(\tau)\to0$ in $L^2(\mathbb{R})$ as $\tau\to T$ for every $K$, proving the ultraviolet limit. Formula \eqref{eq:blowup-R-low-tail} then forces $\lambda_j(\tau)\to0$. Finally, $\|Q_j(\tau)\|_{L^2(\mathbb{R})}^2=2\pi$, and for $j\ne k$ the
orthogonality \eqref{eq:S9-dark-orthogonality} gives
\[
\left|\langle Q_j,Q_k\rangle_{L^2(\mathbb{R})}\right|\leq\sqrt{2\pi}\left(\|Q_j-\mathfrak{d}_j\|_{L^2(\mathbb{R})}+\|Q_k-\mathfrak{d}_k\|_{L^2(\mathbb{R})}\right)\underset{\tau\to T}{\longrightarrow}0.
\]
This proves the Gram limit.

\appendix
\section{Verified analyzable initial data classes}\label{app:verified}
Recall from Definition~\ref{def:analyzable-intro} that $q\in\mathcal{A}_{\mathrm{an}}$ if the following conditions hold.
\begin{enumerate}
 \item[(A1)] The actual Dirichlet compression of the transported KLV realization is the affine family generated by one static maximal dissipative operator, with compatible physical boundary rows.
 \item[(A2)] The actual boundary functional has the bounded Riesz lift required by the enhanced scalar graph driven by $\beta_q$.
 \item[(A3)] A phase-fixing anchor exists in the actual/scalar graph and has the prescribed nonzero endpoint trace and coefficient-one sewing at every embedded point label.
 \item[(A4)] The point return satisfies either the subcritical operator estimate with exponent $\rho<\tfrac12$, or the stated oscillatory weak-return limit.
\end{enumerate}

For $(C_tG)(\lambda)=e^{-it\lambda^2}G(\lambda)$, the dynamic scalar domain is $C_t^{-1}H^1(\mathbb{R}_+)$, which is not contained in $H^1(\mathbb{R}_+)$ when $t\ne0$. Define the dynamic row by
\begin{equation}\label{app:eq:dynamic-row}
\ell_{D,t}G:=\frac{1}{i\sqrt{2\pi}}(C_tG)(0+)=\frac{1}{i\sqrt{2\pi}}G(0+),\qquad G\in C_t^{-1}H^1(\mathbb{R}_+).
\end{equation}
It is the unique graph-continuous extension agreeing with the static row on the common smooth core.

Denote the analyzable class with this dynamic row by $\mathcal{A}_{\mathrm{an}}^{\mathrm{corr}}$. It coincides with $\mathcal{A}_{\mathrm{an}}$ when its endpoint row is interpreted by \eqref{app:eq:dynamic-row}.

For $\mu\in\sigma_p(L_q)\cap[0,\infty)$, let $I_\mu$ be a bounded interval containing $\mu$ and no other point eigenvalue; at $\mu=0$, $I_\mu$ is understood as a right-neighborhood of $0$.
\begin{definition}[Spectral Dini--moment class]\label{app:def:XDM}
Let $\mathcal{X}_{\mathrm{DM}}$ be the set of $q\in L^1(\mathbb{R})\cap L^2_+(\mathbb{R})$ such that
\begin{align}
D_\mu(q)&:=\int_{I_\mu\cap[0,\infty)}\frac{|\beta_q(\lambda)|}{|\lambda-\mu|}d\lambda<\infty&&\text{for every }\mu\in\sigma_p(L_q)\cap[0,\infty),\label{app:eq:Dini}\\
x\phi_j&\in L^2(\mathbb{R})&&\text{for every normalized }\phi_j\in\ker(L_q-\mu_j).\label{app:eq:point-moment}
\end{align}
The first condition is independent of the sufficiently small choice of $I_\mu$. Both conditions are vacuous when $L_q$ has no point spectrum.
\end{definition}

The defining conditions depend nonlinearly on $L_q$, so $\mathcal{X}_{\mathrm{DM}}$ need not be a vector space.
\begin{theorem}[Spectral criterion]\label{app:thm:XDM}
With the dynamic row interpreted by \eqref{app:eq:dynamic-row},
\[
\mathcal{X}_{\mathrm{DM}}\subset\mathcal{A}_{\mathrm{an}}^{\mathrm{corr}}.
\]
For every $q\in\mathcal{X}_{\mathrm{DM}}$, the operator alternative in (A4) holds with the stronger exponent $\rho=0$.
\end{theorem}

For $p\in\{1,2\}$ and $R\ge1$, define $Q_p(q;R):=\|q\|_{L^p(\{|x|>R\})}$.

\begin{definition}[Hardy tail space]\label{app:def:tail}
Set
\[
\mathcal{T}_+:=\left\{q\in L^1(\mathbb{R})\cap L^2_+(\mathbb{R}):\int_1^\infty\bigl(Q_1(q;R)^2+Q_2(q;R)^2\bigr)dR<\infty\right\},
\]
with norm
\begin{equation}\label{app:eq:tail-norm}
\|q\|_{\mathcal{T}_+}:=\|q\|_{L^1(\mathbb{R})}+\|q\|_{L^2(\mathbb{R})}+\|Q_1(q;\cdot)\|_{L^2(1,\infty)}+\|Q_2(q;\cdot)\|_{L^2(1,\infty)}.
\end{equation}
\end{definition}

\begin{prop}[Basic structure of $\mathcal{T}_+$]\label{app:prop:tail-Banach}
The space $(\mathcal{T}_+,\|\cdot\|_{\mathcal{T}_+})$ is a Banach space, is invariant as a set under finite translations, CM-DNLS scalings, phase rotations, and Hardy-preserving modulations, and contains $X_1:=\{q\in L^2_+(\mathbb{R}):\langle x\rangle q\in L^2\}$. Moreover, $X_1\subsetneq\mathcal{T}_+$, and $\mathcal{S}(\mathbb{R})\cap L^2_+(\mathbb{R})$ is dense in $L^2_+(\mathbb{R})$ and contained in $\mathcal{T}_+$.
\end{prop}
\begin{proof}
The triangle inequality for each tail function gives the norm inequality. If $(q_n)$ is Cauchy in $\mathcal{T}_+$, it converges in $L^2$ to some $q\in L^2_+(\mathbb{R})$. For every $R>1$, monotonicity gives $Q_1(q_n-q_m;R)^2\le\int_{R-1}^{R}Q_1(q_n-q_m;r)^2dr$, so $(q_n)$ is Cauchy in $L^1(\{|x|>R\})$. Combining this with local $L^1$-convergence from $L^2$-convergence identifies the tail limit with $q$; Fatou's lemma, applied also to differences, proves completeness. The $Q_2$ part follows directly from
\begin{equation}\label{app:eq:Q2-Fubini}
\int_1^\infty Q_2(q;R)^2dR=\int_\mathbb{R}(|x|-1)_+|q(x)|^2dx.
\end{equation}
Finite changes of origin and CM-DNLS symmetries preserve finiteness of the terms in \eqref{app:eq:tail-norm}.

For $f\ge0$, the dual Hardy inequality on each half-line gives $\int_1^\infty\left(\int_R^\infty f(x)\,dx\right)^2dR\le4\int_1^\infty (x-1)^2f(x)^2dx$. Apply this to $|q(x)|$ and $|q(-x)|$, and use \eqref{app:eq:Q2-Fubini}, to obtain
\[
\|Q_1(q;\cdot)\|_{L^2(1,\infty)}+\|Q_2(q;\cdot)\|_{L^2(1,\infty)}\lesssim\|\langle x\rangle q\|_2.
\]
Thus $X_1\subset\mathcal{T}_+$.

For strictness, choose a nonzero $g\in\mathcal{S}(\mathbb{R})\cap L^2_+(\mathbb{R})$ whose Fourier transform is supported in a compact subinterval of $(0,\infty)$. Choose $R_n\to\infty$ sufficiently rapidly, for example inductively with $R_{n+1}\ge R_n^4$, so that distinct translates below are asymptotically disjoint on $I_n=[R_n-1,R_n+1]$, and put $a_n=\frac{1}{R_n\sqrt n},\qquad q(x)=\sum_{n=1}^\infty a_ng(x-R_n)$. The series converges in $L^2$, retains positive Fourier support, and the centres may be chosen sufficiently rapidly that $\sum_n a_n\sqrt{R_n}<\infty$. Minkowski's inequality and the rapid decay of $g$ imply $\||x|^{1/2}q\|_{L^2(\mathbb{R})}+\|Q_1(q;\cdot)\|_{L^2(1,\infty)}<\infty$; indeed the corresponding norm of one translate is $O_g(1+\sqrt{R_n})$. Hence $q\in\mathcal{T}_+$. The centers may also be chosen so that on every $I_n$, the sum of all other packets has $L^2(I_n)$-norm at most one half of that of $a_ng(\cdot-R_n)$. Therefore $\int_\mathbb{R}x^2|q(x)|^2dx\gtrsim_g\sum_n R_n^2a_n^2=\sum_n\frac{1}{n}=\infty$. Thus $q\notin X_1$, proving strictness. Fourier cutoff and smoothing inside $(0,\infty)$ prove density.
\end{proof}

\begin{theorem}[Physical-space criterion]\label{app:thm:tail}
Every datum in the Hardy tail space is analyzable:
\[
\mathcal{T}_+\subset\mathcal{X}_{\mathrm{DM}}\subset\mathcal{A}_{\mathrm{an}}^{\mathrm{corr}}.
\]
Thus $q\in\mathcal{T}_+$ suffices for the analyzability hypothesis of the main theorem, with the dynamic row defined by \eqref{app:eq:dynamic-row}.
\end{theorem}

For $p\in\{1,2\}$ and $\alpha\ge0$, write $L^p_\alpha:=\{f:\langle x\rangle^{\alpha}f\in L^p(\mathbb{R})\}$.

\begin{lemma}[Fractional moments supplied by the tail norm]\label{app:lem:fractional-moments}
If $q\in\mathcal{T}_+$, then $q\in L^1\cap L^2_+(\mathbb{R})$ and, for every $0<\alpha<\frac{1}{2}$, $q\in L^1_\alpha\cap L^2_\alpha$. More quantitatively,
\begin{align}
\int_{|x|>1}|x|^\alpha|q(x)|dx&\le Q_1(q;1)+\alpha\left(\int_1^\infty R^{2\alpha-2}dR\right)^{1/2}\|Q_1(q;\cdot)\|_{L^2_R},\label{app:eq:L1-moment}\\
\int_{|x|>1}|x|^{2\alpha}|q(x)|^2dx&\le Q_2(q;1)^2+2\alpha\int_1^\infty Q_2(q;R)^2dR.\label{app:eq:L2-moment}
\end{align}
\end{lemma}
\begin{proof}
The layer-cake identity gives
\[
\int_{|x|>1}|x|^\alpha|q(x)|dx=Q_1(q;1)+\alpha\int_1^\infty R^{\alpha-1}Q_1(q;R)dR.
\]
Cauchy--Schwarz proves \eqref{app:eq:L1-moment}, since $2\alpha-2<-1$.  Applying the same identity to $|q|^2$, and noting that $R^{2\alpha-1}\le1$ for $R\ge1$, proves \eqref{app:eq:L2-moment}.
\end{proof}

\begin{lemma}[Endpoint point-position estimate]\label{app:lem:point-position}
Let $q\in\mathcal{T}_+$, and let $\phi$ be any point eigenfunction of $L_q$. Then
\begin{equation}\label{app:eq:xphi}
x\phi\in L^2(\mathbb{R}).
\end{equation}
Moreover, for every $0<\alpha<\frac{1}{2}$, $\phi\in L^1_\alpha\cap L^\infty_\alpha$. \end{lemma}
\begin{proof}
Since $q\in L^1\cap L^2_+$, Frank--Read's Theorem~3.1 and Lemma~3.2 apply. Inspection of their proof gives, for $|x|$ sufficiently large,
\begin{equation}\label{app:eq:FR-tail}
|\phi(x)|\le \frac{C_{q,\phi}}{|x|}\max\left\{Q_1(q;|x|/2),Q_2(q;|x|/2)\right\}.
\end{equation}
Consequently,
\[
\int_{|x|>R_0}|x\phi(x)|^2dx\lesssim_{q,\phi}\int_{R_0/2}^\infty\bigl(Q_1(q;R)^2+Q_2(q;R)^2\bigr)dR<\infty.
\]
Every eigenfunction belongs to $H^1_+$, hence is locally bounded; this supplies the compact part and proves \eqref{app:eq:xphi}.

The right side of \eqref{app:eq:FR-tail} is bounded by $C|x|^{-1}$, so $\phi\in L^\infty_\alpha$ for $\alpha<1$. For the weighted $L^1$ estimate, Cauchy--Schwarz gives
\[
\int_{|x|>R_0}|x|^\alpha|\phi(x)|dx\lesssim\left(\int_{R_0/2}^\infty[Q_1(q;R)^2+Q_2(q;R)^2]dR\right)^{1/2}\left(\int_{R_0}^\infty x^{2\alpha-2}dx\right)^{1/2},
\]
which is finite precisely for $\alpha<\frac{1}{2}$.  Local integrability follows from local boundedness.
\end{proof}

\begin{lemma}[H\"older divisibility at embedded point spectrum]\label{app:lem:beta-holder}
Let $q\in\mathcal{T}_+$. For every $0<\alpha<\frac{1}{2}$,
\begin{equation}\label{app:eq:beta-holder}
\beta_q\in C^{0,\alpha}_{\mathrm{loc}}([0,\infty)).
\end{equation}
If $\mu\in\sigma_{\mathrm{p}}(L_q)\cap[0,\infty)$, then
\begin{equation}
|\beta_q(\lambda)|\lesssim_{q,\mu,\alpha}|\lambda-\mu|^\alpha,\qquad\frac{\beta_q(\lambda)}{\lambda-\mu}\in L^1_{\mathrm{loc}},\label{app:eq:beta-divisibility}
\end{equation}
one-sided at $\mu=0$.
\end{lemma}
\begin{proof}
We quantify the reduced Fredholm argument in the proof of Frank--Read's Theorem~8.1.

Fix one of the two limiting-resolvent banks and a compact interval $I$. The free first-order resolvent kernel and $|e^{i h y}-1|\lesssim_\alpha |h|^\alpha|y|^\alpha$ give $\left\|\Pi_+\overline q[R_0^\pm(\lambda)-R_0^\pm(\nu)]q\Pi_+\right\|_{\mathcal{S}_2}\lesssim_{q,I,\alpha}|\lambda-\nu|^\alpha$. Indeed, after squaring the kernel, use $|x-y|^{2\alpha}\lesssim |x|^{2\alpha}+|y|^{2\alpha}$ and Lemma~\ref{app:lem:fractional-moments}. Away from the finite exceptional set, the resolvent identity and Fredholm inversion now yield an operator-norm $C^{0,\alpha}$ inverse family.

At $\mu\in\sigma_{\mathrm{p}}(L_q)\cap[0,\infty)$, let $P_\mu$ be the rank-one projection onto its normalized eigenfunction $\phi_\mu$, and put $Q_\mu=\mathrm{Id}-P_\mu$.  Because $\phi_\mu\in L^1\cap L^\infty$, the same formula extends canonically to bounded functions:
\[
P_\mu h=\frac{\langle h,\phi_\mu\rangle_{L^\infty,L^1}}{\|\phi_\mu\|_{L^2(\mathbb{R})}^2}\phi_\mu.
\]
Lemma~\ref{app:lem:point-position} and the same kernel difference estimate give
\[
\left\|q[R_0^\mp(\lambda)-R_0^\mp(\nu)]\phi_\mu\right\|_{L^2(\mathbb{R})}\lesssim |\lambda-\nu|^\alpha\bigl(\|q\|_{L^2_\alpha(\mathbb{R})}\|\phi_\mu\|_{L^1(\mathbb{R})}+\|q\|_{L^2(\mathbb{R})}\|\phi_\mu\|_{L^1_\alpha(\mathbb{R})}\bigr),
\]
where $\|f\|_{L^p_\alpha(\mathbb{R})}:=\|\langle x\rangle^{\alpha}f\|_{L^p(\mathbb{R})}$. For $f\in L^2_+$, the projection term is
\begin{equation}\label{app:eq:rank-one}
\|\phi_\mu\|_{L^2(\mathbb{R})}^2\Pi_+\overline{q}P_\mu R_0^\pm(\lambda)q\Pi_+ f={\left\langle f,\Pi_+\overline qR_0^\mp(\lambda)\phi_\mu\right\rangle}\Pi_+(\overline q\phi_\mu).
\end{equation}
Consequently the reduced family $\mathcal{B}_\lambda:=\mathrm{Id}-\Pi_+\overline qQ_\mu R_0^\pm(\lambda)q\Pi_+$ is $C^{0,\alpha}$ in operator norm.

Put
\[
g_\mu=\Pi_+(\overline q\phi_\mu),\qquad\mathcal{A}_\mu^\pm=\mathrm{Id}-\Pi_+\overline qR_0^\pm(\mu)q\Pi_+.
\]
The exceptional-space correspondence and simplicity of $\mu$ give
\[
\ker\mathcal{A}_\mu^+=\ker\mathcal{A}_\mu^-=\operatorname{span}\{g_\mu\},\qquad (\mathcal A_\mu^\pm)^\ast=\mathcal{A}_\mu^\mp,
\]
so $\mathrm{Ran}\mathcal{A}_\mu^\pm=g_\mu^\perp$. Formula \eqref{app:eq:rank-one} shows that $\mathcal{B}_\mu-\mathcal{A}_\mu^\pm$ has range in $\operatorname{span}\{g_\mu\} $. If $\mathcal{B}_\mu f=0$, the two orthogonal range components vanish separately, so $f=cg_\mu$. The eigenfunction equation yields $R_0^\pm(\mu)q\Pi_+ g_\mu=\phi_\mu,\qquad Q_\mu R_0^\pm(\mu)q\Pi_+ g_\mu=0$, and hence $\mathcal B_\mu g_\mu=g_\mu$; therefore $c=0$. Since $\mathcal{B}_\mu$ is identity minus compact and has index zero, it is invertible. After shrinking $I$, $\mathcal{B}_\lambda^{-1}-\mathcal{B}_\nu^{-1}=\mathcal{B}_\lambda^{-1}(\mathcal{B}_\nu-\mathcal{B}_\lambda)\mathcal{B}_\nu^{-1}$ is $C^{0,\alpha}$ in operator norm.

To obtain $L^\infty_x$ estimates for the Jost functions, let $e_\lambda(x)=e^{i\lambda x}$ and set
\[
b_\lambda=\Pi_+(\overline qQ_\mu e_\lambda),\qquad w_\lambda=\mathcal{B}_\lambda^{-1}b_\lambda,\qquad F_\lambda=q\Pi_+ w_\lambda.
\]
Here
\[
Q_\mu e_\lambda=e_\lambda-\frac{\langle e_\lambda,\phi_\mu\rangle_{L^\infty,L^1}}{\|\phi_\mu\|_{L^2(\mathbb{R})}^2}\phi_\mu.
\]
The elementary modulation estimate and the fractional moments already proved give
\[
\|b_\lambda-b_\nu\|_{L^2(\mathbb{R})}\lesssim |\lambda-\nu|^\alpha\left(\|q\|_{L^2_\alpha(\mathbb{R})}+\frac{\|q\phi_\mu\|_{L^2(\mathbb{R})}\|\phi_\mu\|_{L^1_\alpha(\mathbb{R})}}{\|\phi_\mu\|_{L^2(\mathbb{R})}^2}\right).
\]
Thus $b_\lambda,w_\lambda$ are $C^{0,\alpha}$ in $L^2$, while $F_\lambda$ is $C^{0,\alpha}$ in $L^1$ and uniformly bounded in $L^1_\alpha$. Define the modulated free-resolvent row $U_\lambda=e^{-i\lambda x}R_0^\pm(\lambda)F_\lambda$. The Volterra kernel gives the explicit estimate
\[
\|U_\lambda-U_\nu\|_{L^\infty(\mathbb{R})}\lesssim \|F_\lambda-F_\nu\|_{L^1(\mathbb{R})}+|\lambda-\nu|^\alpha\|F_\nu\|_{L^1_\alpha(\mathbb{R})},
\]
so $U_\lambda$ is $C^{0,\alpha}$ in $L^\infty$.

Finally set
\[
d_\lambda=\langle e_\lambda,\phi_\mu\rangle,\qquad v_\lambda=e^{-i\lambda x}\phi_\mu,\qquad a_\lambda=\langle U_\lambda,v_\lambda\rangle_{L^\infty,L^1}.
\]
The conditions
$\phi_\mu\in L^1_\alpha\cap L^\infty_\alpha$ imply that $d_\lambda$, $v_\lambda$, and $a_\lambda$ are all $C^{0,\alpha}$ in their respective scalar, $L^1\cap L^\infty$, and scalar topologies. Steps~2a--2c of Frank--Read's proof give
\begin{equation}\label{app:eq:reduced-reconstruction}
e^{-i\lambda x}Q_\mu m_{\mathrm{e}}^\pm(\lambda)=1-\|\phi_\mu\|_{L^2(\mathbb{R})}^{-2}d_\lambda v_\lambda+U_\lambda-\|\phi_\mu\|_{L^2(\mathbb{R})}^{-2}a_\lambda v_\lambda.
\end{equation}
The reduced orthogonality identity gives $P_\mu m_{\mathrm e}^\pm(\lambda)=0$ away from $\mu$; hence \eqref{app:eq:reduced-reconstruction} proves $M_\lambda^\pm(x):=e^{-i\lambda x}m_{\mathrm{e}}^\pm(x,\lambda)\in C^{0,\alpha}(I;L^\infty_x)$. Away from the finite exceptional set the same proof is the unreduced case $P_\mu=0$. Finally, the incoming scattering pairing is $i\overline{\beta_q(\lambda)}=\int_\mathbb{R}\overline{q(x)}e^{i\lambda x}M^-_\lambda(x)dx$. Consequently,
\begin{align*}
|\beta_q(\lambda)-\beta_q(\nu)|&\lesssim \|q\|_{L^1(\mathbb{R})}\|M^-_\lambda-M^-_\nu\|_{L^\infty(\mathbb{R})}+\|q(e^{i\lambda x}-e^{i\nu x})\|_{L^1(\mathbb{R})}\sup_{\theta\in I}\|M^-_\theta\|_{L^\infty(\mathbb{R})}\\
&\lesssim |\lambda-\nu|^\alpha,
\end{align*}
which proves \eqref{app:eq:beta-holder}. Frank--Read's Lemma~8.5 gives $\beta_q(\mu)=0$. Hence $|\beta_q(\lambda)|/|\lambda-\mu|\lesssim |\lambda-\mu|^{\alpha-1}$, which is locally integrable for every $\alpha>0$, proving \eqref{app:eq:beta-divisibility}.
\end{proof}

\begin{proof}[Proof of Theorem~\ref{app:thm:tail}]
Lemma~\ref{app:lem:fractional-moments} gives $q\in L^1\cap L^2_+(\mathbb{R})$. Lemma~\ref{app:lem:beta-holder} proves \eqref{app:eq:Dini} at every nonnegative point eigenvalue, and Lemma~\ref{app:lem:point-position} proves \eqref{app:eq:point-moment} for every point eigenfunction. Thus $q\in\mathcal{X}_{\rm DM}$, and Theorem~\ref{app:thm:XDM} gives analyzability.
\end{proof}

To prove Theorem~\ref{app:thm:XDM}, we verify (A1)--(A4).

Let $\Phi_q:P_{\mathrm{ac}}(L_q)L^2_+(\mathbb{R})\to L^2(\mathbb{R}_+)$ be the distorted Fourier transform of \cite[Theorem~5.3]{Frank-2026-arXiv}.  With an $L^2$-orthonormal point basis $e_1,\ldots,e_N$, it extends to a unitary map
\[
\mathscr{U}_q:L^2_+(\mathbb{R})\longrightarrow\mathcal{E}\oplus\mathcal{H}_c,\qquad\mathcal{E}=\mathbb{C}^N,\quad \mathcal{H}_c=L^2(\mathbb{R}_+),
\]
under which
\[
\mathscr{U}_qL_q\mathscr{U}_q^{-1}=M_{\rm p}\oplus M_\lambda,\qquad M_{\rm p}=\operatorname{diag}(\mu_1,\ldots,\mu_N).
\]
Let $P,Q$ be the projections onto $\mathcal{E},\mathcal{H}_c$, respectively, and let $J:\mathcal{H}_c\hookrightarrow\mathcal{E}\oplus\mathcal{H}_c$ be the inclusion.

The Hardy position operator is
\[
\mathrm{Dom}(\mathsf{X})=\{f\in L^2_+(\mathbb{R}):\widehat{f}|_{\mathbb{R}_+}\in H^1(\mathbb{R}_+)\},\qquad \widehat{\mathsf{X}f}=i\partial_\xi\widehat{f},\qquad I_+f=\widehat{f}(0+).
\]
Put $T=\mathscr{U}_q\mathsf{X}\mathscr{U}_q^{-1}$, and let $T_D$ be the actual time-zero Dirichlet compression from the main text.

Extend $\beta_q$ by zero to $( -\infty,0)$, and set
\[
A_{\beta_q}=i\partial_\lambda+B_+(\beta_q),\qquad B_+(\beta)=M_\beta\mathcal{C}_H^+M_{\overline{\beta}},\qquad \mathrm{Dom}(A_{\beta_q})=H^1(\mathbb{R}_+),
\]
where $\mathcal{C}_H^+=(4\pi)^{-1}(\mathrm{Id}+i H)$ on the half-line. Frank--Read's Theorems~7.1 and~8.1 and Corollary~8.4 imply, for $q\in L^1\cap L^2_+$,
\begin{equation}\label{app:eq:beta-bounded}
\beta_q\in C_0([0,\infty))\cap L^2(\mathbb{R}_+)\cap L^\infty(\mathbb{R}_+).
\end{equation}
Thus $B_+(\beta_q)$ is a bounded self-adjoint operator.  Moreover, $A_{\beta_q}$ is maximal dissipative; one graph core is $C_c^\infty([0,\infty))$.

\begin{prop}[Static graph and nonzero endpoint]\label{app:prop:static-graph}
If $q\in\mathcal{X}_{\mathrm{DM}}$, then
\begin{equation}\label{app:eq:static-graph}
T_D=A_{\beta_q},\qquad \mathrm{Dom}(T_D)=H^1(\mathbb{R}_+),
\end{equation}
and, for every $G\in H^1(\mathbb{R}_+)$,
\begin{equation}\label{app:eq:static-row}
I_+(\mathscr{U}_q^{-1}JG)=\sqrt{2\pi}G(0+),\qquad\ell_DG=\frac{1}{i\sqrt{2\pi}}G(0+).
\end{equation}
The scalar graph has coefficient-one sewing at every positive embedded eigenvalue.
\end{prop}
\begin{proof}
It suffices to identify the actual and scalar operators on the core $\mathcal{C}=C_c^\infty([0,\infty))$.  Let $G\in\mathcal{C}$, set $a=G\overline{\beta_q}$, and define $\mathcal{R}_{G,\varepsilon}:=\int_0^\infty a(\lambda)(L_q-\lambda+i\varepsilon)^{-1}qd\lambda$. In continuous spectral coordinates,
\[
[\Phi_q\mathcal{R}_{G,\varepsilon}](\nu)=-\frac{i\beta_q(\nu)}{\sqrt{2\pi}}\int_0^\infty\frac{G(\lambda)\overline{\beta_q(\lambda)}}{\nu-\lambda+i\varepsilon}d\lambda.
\]
The half-line Cauchy transforms converge strongly on $L^2$, and \eqref{app:eq:beta-bounded} preserves this convergence. A point coordinate at $\mu_j\ge0$ is a fixed multiple of
\[
\int_0^\infty\frac{G(\lambda)\overline{\beta_q(\lambda)}}{\mu_j-\lambda+i\varepsilon}d\lambda,
\]
which converges by \eqref{app:eq:Dini}; at a negative eigenvalue the denominator is separated from the integration ray. Since the point spectrum is finite, distorted Plancherel yields $\mathcal{R}_{G,\varepsilon}\longrightarrow\mathcal{R}_G \quad\text{strongly in }L^2_+(\mathbb{R})$.
Near a nonnegative point eigenvalue, rescale an eigenfunction to the Frank--Read normalization $\langle q,\psi_j\rangle=2\pi i$. Frank--Read's Theorem~8.1 then gives the incoming Laurent expansion $m_0^-(\lambda)=-\frac{i\psi_j}{\lambda-\mu_j}+h_j^-(\lambda)$, where the regular part is weak-$\ast$ continuous and locally bounded in $L^\infty_x$. Condition \eqref{app:eq:Dini} makes the pole in $a(\lambda)m_0^-(\lambda)$ an ordinary weak-$\ast$ $L^1_\lambda$ integral. Define the resulting positive-frequency distribution $\mathcal{J}_G:=\mathrm{w}^\ast\int_0^\infty G(\lambda)\overline{\beta_q(\lambda)}m_0^-(\lambda)d\lambda$. For every $\varphi\in\mathcal{S}(\mathbb{R})$ with $\operatorname{supp}\widehat{\varphi}\Subset(0,\infty)$, the resolvent identity $(L_q-\lambda+i\varepsilon)^{-1}q
=m_0(\lambda-i\varepsilon)$, followed by dominated convergence in the Laurent chart, gives $\langle\mathcal{R}_G,\varphi\rangle=\langle\mathcal{J}_G,\varphi\rangle$. The domination near each nonnegative eigenvalue is precisely $C\frac{|\beta_q(\lambda)|}{|\lambda-\mu_j|}+C|\beta_q(\lambda)|$, and away from the finite exceptional set it is $C|\beta_q|$. Thus $\operatorname{supp}\widehat{(\mathcal{J}_G-\mathcal{R}_G)}\subset\{0\}$. The structure theorem for point-supported Fourier distributions makes this difference a polynomial. The same majorants give $\mathcal{J}_G\in L^\infty$, whereas $\mathcal{R}_G\in L^2$, so $\mathcal{J}_G-\mathcal{R}_G=c$ for a constant $c$.

To show $c=0$, convolve this identity with $\chi\in C_c^\infty(\mathbb{R})$, $\int\chi=1$. The spectral term tends to zero at $+\infty$, since $\widehat{\mathcal{R}_G}\widehat{\chi}\in L^1$. The Jost term does so by the incoming resolvent boundary condition, dominated near each pole by $C\frac{|\beta_q(\lambda)|}{|\lambda-\mu_j|}+C|\beta_q(\lambda)|$, which is integrable by \eqref{app:eq:Dini}. Hence $c=0$ and $\mathcal{R}_G=\mathcal{J}_G =\mathrm{w}^\ast\int_0^\infty G(\lambda)\overline{\beta_q(\lambda)}m_0^-(\lambda)d\lambda$.
Frank--Read's Jost derivative identity now permits integration by parts on the whole half-line. Coefficient-one continuity of the modulated Jost column through an embedded label and \eqref{app:eq:Dini} cancel the interior boundary terms; the incoming normalization gives the term at zero.  Thus, in $\mathcal{S}'(\mathbb{R}_x)$,
\begin{equation}\label{app:eq:Jost-IBP}
\int_0^\infty G(\lambda)x m_{\mathrm{e}}^-(\lambda)d\lambda=i G(0)\mathbf{1}+i\int_0^\infty G'(\lambda)m_{\mathrm{e}}^-(\lambda)d\lambda-\frac{1}{2\pi}\mathcal{R}_G.
\end{equation}

Set $f=\mathscr{U}_q^{-1}JG=\Phi_q^\ast G$. Applying the inverse distorted Fourier formula to \eqref{app:eq:Jost-IBP} yields $xf=\frac{iG(0)}{\sqrt{2\pi}}+i\Phi_q^\ast G'-\frac{1}{2\pi\sqrt{2\pi}}\mathcal{R}_G$. The last two terms belong to $L^2_+(\mathbb{R})$. Fourier transforming the elementary distributional implication
\[
xf=h+c,\quad h\in L^2_+(\mathbb{R})\quad\Longrightarrow\quad f\in\mathrm{Dom}(\mathsf{X}),\quad I_+f=-2\pi ic,
\]
gives
\[
I_+f=\sqrt{2\pi}G(0),\qquad QTJG=i G'+\beta_q\mathcal{C}_H^+(\overline{\beta_q}G)=A_{\beta_q}G.
\]
Thus $A_{\beta_q}\subset T_D$ on a graph core. Both operators are maximal dissipative, so they are equal. Equation \eqref{app:eq:static-row} follows from the definition
$\ell_D=(2\pi i)^{-1}I_+\mathscr{U}_q^{-1}J$. Every $H^1$-function has a single trace at each interior point, giving coefficient-one sewing.
\end{proof}

\begin{lemma}[Point lines lie in both Hardy position domains]\label{app:lem:point-domains}
If $q\in\mathcal{X}_{\mathrm{DM}}$, then every vector $e_j$ in the orthonormal point basis satisfies
\begin{equation}\label{app:eq:point-domains}
e_j\in\mathrm{Dom}(\mathsf{X})\cap\mathrm{Dom}(\mathsf{X}^\ast),\qquad I_+e_j=0,\qquad \mathsf{X}e_j=\mathsf{X}^\ast e_j=xe_j.
\end{equation}
\end{lemma}
\begin{proof}
Frank--Read's Lemma~3.2 gives $e_j\in L^1$ because $q\in L^1\cap L^2_+$. Condition \eqref{app:eq:point-moment} gives $\widehat{e}_j\in H^1(\mathbb{R})$. Its continuous representative is zero on $(-\infty,0]$, because $e_j\in L^2_+(\mathbb{R})$; in particular its trace at zero vanishes. The Fourier definitions of $\mathsf{X}$ and its adjoint
give \eqref{app:eq:point-domains}.
\end{proof}

\begin{prop}[Bounded point--continuous block]\label{app:prop:block}
There are a self-adjoint matrix $A_{\rm p}$ on $\mathcal{E}$ and a bounded operator $B:\mathcal{H}_c\to\mathcal{E}$ such that
\begin{equation}\label{app:eq:block}
T=\begin{pmatrix}A_{\rm p}&B\\B^\ast&A_{\beta_q}\end{pmatrix},\qquad \mathrm{Dom}(T)=\mathcal{E}\oplus H^1(\mathbb{R}_+).
\end{equation}
If
$b_j=\Phi_qP_{\mathrm{ac}}(L_q)\mathsf{X}e_j$, then
\begin{equation}\label{app:eq:B-bound}
(BG)_j=\langle G,b_j\rangle,\qquad\|B\|\le\left(\sum_{j=1}^N\|b_j\|_{L^2(\mathbb{R})}^2\right)^{1/2}.
\end{equation}
\end{prop}

\begin{proof}
For $G\in H^1$, Proposition~\ref{app:prop:static-graph} and Lemma~\ref{app:lem:point-domains} give
\[
(PTJG)_j=\langle\mathsf{X}\mathscr{U}_q^{-1}JG,e_j\rangle=\langle G,\Phi_qP_{\mathrm{ac}}\mathsf{X}^\ast e_j\rangle,
\]
 which proves \eqref{app:eq:B-bound}. The Hardy Green identity and $I_+e_j=0$ identify the opposite block with $B^\ast$ and make $A_{\rm p}$ self-adjoint. The domain identity follows by subtracting the finite point component from an arbitrary vector in $\mathrm{Dom}(T)$.
\end{proof}

\begin{prop}[Dynamic common graph]\label{app:prop:dynamic-graph}
For every $t\in\mathbb{R}$ and $\eta\in\mathbb{R}$, the transported KLV realization coincides with the closed block realization of
\begin{equation}\label{app:eq:dynamic-pencil}
T+2t(M-\eta)-z_0,
\end{equation}
whose continuous domain is $C_t^{-1}H^1(\mathbb{R}_+)$. Its Dirichlet row is $\ell_{D,t}$ from \eqref{app:eq:dynamic-row}. Hence (A1), including strong Borel dependence on $\eta$, holds.
\end{prop}

\begin{proof}
Let $\mathscr{G}:=\mathscr{U}_q^{-1}\bigl(\mathcal{E}\oplus C_c^\infty([0,\infty))\bigr)$. Lemma~\ref{app:lem:point-domains} places every point basis vector in $\mathrm{Dom}(\mathsf{X})\cap\mathrm{Dom}(L_q)$. Proposition~\ref{app:prop:static-graph} places $\Phi_q^\ast G$ in $\mathrm{Dom}(\mathsf{X})$ for smooth $G$, while compact spectral support gives $\Phi_q^\ast G\in\mathrm{Dom}(L_q)$. Thus $\mathscr{G}\subset\mathrm{Dom}(\mathsf{X})\cap\mathrm{Dom}(L_q)$. KLV Proposition~4.10 states that its abstract maximal dissipative realization agrees with the algebraic pencil on this common domain.

Equation \eqref{app:eq:block} and conjugation by $C_t$ show that the explicit block closure of \eqref{app:eq:dynamic-pencil} has domain $\mathcal{E}\oplus C_t^{-1}H^1$, is maximal dissipative, and has $\mathcal{E}\oplus C_c^\infty([0,\infty))$ as a graph core. The KLV realization is closed and contains this core restriction; hence it contains the explicit block closure. Maximal dissipativity of both operators forces equality.

For the boundary row, KLV Proposition~4.11, equation~(4.42), gives the Ward estimate
\[
|I_+\mathcal{A}(t,z;q)f|\le C_{t,z,q}\|f\|_{L^2(\mathbb{R})},
\]
 which makes the actual physical row graph-continuous on the KLV resolvent range. On the explicit side, $C_t$ is a graph isomorphism from $C_t^{-1}H^1$ to $H^1$, and $C_t(0)=1$; hence $\ell_{D,t}$ is graph-continuous. The two rows agree on $C_c^\infty([0,\infty))$, so core density makes them equal on the full dynamic graph.  This proves \eqref{app:eq:dynamic-row}.

Resolvent continuity in the real shift $\eta$, followed by finite-dimensional compression and inversion, gives the asserted Borel properties.
\end{proof}

\begin{prop}[Gates (A2)--(A4)]\label{app:prop:last-gates}
Every $q\in\mathcal{X}_{\mathrm{DM}}$ satisfies the actual-form and anchor gates. Moreover, with $K_t(\eta)$ denoting the point compression of the full resolvent,
\begin{equation}\label{app:eq:Schur-bound}
\sup_{t\ge0,\,\eta\in\mathbb{R}}\|K_t(\eta)^{-1}-2t(M_{\mathrm{p}}-\eta)\|\le \|A_{\rm p}-z_0\|+\frac{\|B\|^2}{\operatorname{Im}z_0}.
\end{equation}
Thus (A4) holds with $\rho=0$.
\end{prop}
\begin{proof}
For $G_h=(T_D-\zeta_\ast)^{-1}h$, Proposition \ref{app:prop:static-graph} and \eqref{app:eq:beta-bounded} give
\[
Z_h=\mathcal{C}_H^+(\overline{\beta_q}G_h)\in L^2,\qquad \|Z_h\|_{L^2(\mathbb{R})}\le\|\mathcal{C}_H^+\|\|\beta_q\|_{L^\infty(\mathbb{R})}\|G_h\|_{L^2(\mathbb{R})}\lesssim_{\zeta_\ast}\|h\|_{L^2(\mathbb{R})}.
\]
Hence, for smooth $v$,
\[
\left|\int_0^\infty\overline{\beta_q}G_h\overline{\mathcal{C}_H^+v}d\lambda\right|\le \|\beta_q\|_{L^\infty(\mathbb{R})}\|G_h\|_{L^2(\mathbb{R})}\|\mathcal{C}_H^+\|\|v\|_{L^2(\mathbb{R})},
\]
and the Riesz representative is $Z_h$. The identity $T_D=A_{\beta_q}$ supplies the punctured distributional equation and its coefficient-one traces. This proves (A2).

Choose $E_0\in C_c^\infty([0,\infty))$ with $E_0(0)=(2\pi)^{-1/2}$. Equations \eqref{app:eq:static-graph}--\eqref{app:eq:static-row} give
\[
I_+(\mathscr{U}_q^{-1}JE_0)=1,\qquad J_+E_0=(2\pi)^{-1/2},\qquad \ell_DE_0=(2\pi i)^{-1},
\]
and $E_0\in H^1$ has identical traces on both banks of every embedded label. This proves (A3).

Finally, take the Schur complement of the continuous diagonal block in the closed dynamic matrix from Proposition~\ref{app:prop:dynamic-graph}. If $R^D_{t,\eta}$ is the Dirichlet resolvent, then
\begin{equation}\label{app:eq:Schur-identity}
K_t(\eta)^{-1}=A_{\mathrm{p}}+2t(M_{\mathrm{p}}-\eta)-z_0-BR^D_{t,\eta}B^\ast.
\end{equation}
Maximal dissipative resolvent control gives $\|R^D_{t,\eta}\|\le(\operatorname{Im}z_0)^{-1}$. Subtracting the linear point motion in \eqref{app:eq:Schur-identity} proves
\eqref{app:eq:Schur-bound}. The point-return proposition in the main text then gives the weak oscillatory alternative as well.
\end{proof}

\begin{proof}[Proof of Theorem~\ref{app:thm:XDM}]
Propositions~\ref{app:prop:static-graph} and \ref{app:prop:dynamic-graph} give (A1). Proposition \ref{app:prop:last-gates} gives (A2)--(A4), with $\rho=0$. Therefore $q\in\mathcal{A}_{\mathrm{an}}^{\mathrm{corr}}$.
\end{proof}

\begin{prop}[Strict inclusions]\label{app:prop:hierarchy}
One has
\[
\bigcup_{s>1}X_s\subsetneq X_1\subsetneq\mathcal{T}_+\subsetneq\mathcal{X}_{\mathrm{DM}}\subset\mathcal{A}_{\mathrm{an}}^{\mathrm{corr}}.
\]
\end{prop}
\begin{proof}
For the first strict inclusion, choose $g\in\mathcal{S}(\mathbb{R})\cap L^2_+(\mathbb{R})$ as in the proof of Proposition~\ref{app:prop:tail-Banach}, choose centres $S_n$ increasing super-exponentially and with the same local packet-dominance property, and put $d_n=\frac{1}{n^2S_n},\qquad q_0(x)=\sum_{n=1}^\infty d_ng(x-S_n)$. Minkowski's inequality gives $\|\langle x\rangle q_0\|_{L^2(\mathbb{R})}\lesssim_g\sum_n d_n(1+S_n)<\infty$, so $q_0\in X_1$. On the dominance interval around $S_n$, for every $\varepsilon>0$, $\int_\mathbb{R}\langle x\rangle^{2(1+\varepsilon)}|q_0(x)|^2dx\gtrsim_g S_n^{2(1+\varepsilon)}d_n^2=\frac{S_n^{2\varepsilon}}{n^4}$. The right side is unbounded in $n$, hence $q_0\notin X_{1+\varepsilon}$. This proves $\bigcup_{s>1}X_s\subsetneq X_1$.

For the strictness $\mathcal{T}_+\subsetneq\mathcal{X}_{\mathrm{DM}}$, take $q_c(x)=c(1-i x)^{-3/2},\qquad 0<|c|<\sqrt\pi$. The Gamma-integral representation gives
\[
\widehat{q}_c(\xi)=\frac{2\pi c}{\Gamma(3/2)}\xi^{1/2}e^{-\xi}\mathbf{1}_{(0,\infty)}(\xi),\qquad\|q_c\|_{L^2(\mathbb{R})}^2=2|c|^2<2\pi.
\]
Hence $q_c\in L^1\cap L_+^2(\mathbb{R})$. Frank--Read's mass trace formula
\begin{equation}\label{app:eq:mass-trace}
\|q\|_{L^2(\mathbb{R})}^2=2\pi N(L_q)+\frac{1}{2\pi}\|\beta_q\|_{L^2(\mathbb{R})}^2
\end{equation}
forces $N(L_{q_c})=0$, so both spectral conditions in Definition \ref{app:def:XDM} are vacuous and $q_c\in\mathcal{X}_{\mathrm{DM}}$. Since $Q_1(q_c;R)\asymp_c R^{-1/2}$, we have $Q_1(q_c;\cdot)\notin L^2(1,\infty)$ and $q_c\notin\mathcal{T}_+$.
\end{proof}

\begin{cor}[All subthreshold integrable Hardy data are analyzable]\label{app:cor:subthreshold}
If
\begin{equation}\label{app:eq:subthreshold}
q\in L^1(\mathbb{R})\cap L^2_+(\mathbb{R}), \qquad \|q\|_{L^2(\mathbb{R})}^2<2\pi,
\end{equation}
then $q\in\mathcal{X}_{\mathrm{DM}}\subset\mathcal{A}_{\mathrm{an}}^{\mathrm{corr}}$. More generally, the same conclusion holds for every $q\in L^1\cap L^2_+(\mathbb{R})$ whose Lax operator has no point spectrum, regardless of its mass.
\end{cor}
\begin{proof}
Equation \eqref{app:eq:mass-trace} gives $N(L_q)=0$ under \eqref{app:eq:subthreshold}. Both requirements in Definition \ref{app:def:XDM} are then vacuous. The second assertion is immediate from the same definition.
\end{proof}

\section*{Acknowledgements}
The authors are deeply grateful to Rupert L. Frank, Soonsik Kwon, and Taegyu Kim for their insightful comments throughout this work.

\section*{Funding information}
This work was supported by the National Natural Science Foundation of China under Grant No. 12371255, the Fundamental Research Funds for the Central Universities of CUMT under Grant No. 2024ZDPYJQ1003, and the Postgraduate Research \& Practice Program of Education \& Teaching Reform of CUMT under Grant No. 2025YJSJG031.

\end{document}